\documentclass[12pt,reqno]{amsart}
\usepackage{amssymb}
\usepackage{cases}
\usepackage{bbm}
\usepackage{mathrsfs}
\usepackage{graphicx}
\usepackage{amsfonts}
\usepackage{enumerate}
\usepackage{hyperref}

 \usepackage[dvipsnames]{xcolor}

\usepackage{amssymb, amsmath, amsfonts, latexsym}
\usepackage{enumerate}
\usepackage{color}

\newtheorem{thm}{Theorem}[section]
\newtheorem{cor}[thm]{Corollary}
\newtheorem{lem}[thm]{Lemma}
\newtheorem{prop}[thm]{Proposition}
\newtheorem{example}[thm]{Example}
\newtheorem{remarks}[thm]{Remark}
\newtheorem{defn}[thm]{Definition}
\newtheorem{hyp}[thm]{Hypothesis}
\numberwithin{equation}{section}

\date{}

\def\BB{\mathcal B}

\def\<{\langle}
\def\>{\rangle}

\def\d"{^{\prime\prime}}

\def\bhyp{\begin{hyp}}
\def\nhyp{\end{hyp}}

\def\bbeq{\begin{equation}}
\def\nneq{\end{equation}}

\def\bdef{\begin{defn}}
\def\ndef{\end{defn}}

\def\bthm{\begin{thm}}
\def\nthm{\end{thm}}

\def\bprop{\begin{prop}}
\def\nprop{\end{prop}}

\def\brmk{\begin{remarks}}
\def\nrmk{\end{remarks}}

\def\bexa{\begin{example}}
\def\nexa{\end{example}}

\def\blem{\begin{lem}}
\def\nlem{\end{lem}}

\def\bcor{\begin{cor}}
\def\ncor{\end{cor}}

\def\bexe{\begin{exe}}
\def\nexe{\end{exe}}

\def\bprf{\begin{proof}}
\def\nprf{\end{proof}}

\def\bdes{\begin{description}}
\def\ndes{\end{description}}

\def\benu{\begin{enumerate}}
\def\nenu{\end{enumerate}}

\usepackage{geometry}
\begin{document}

 \title[Killed Generalized Langevin  and Nos\'e-Hoover processes]
 {Generalized Langevin  and  Nos\'e-Hoover processes absorbed at the boundary of a metastable domain}

\author[A. Guillin]{\textbf{\quad {Arnaud} Guillin$^{\dag}$    }}
\address{{\bf Arnaud Guillin}. Universit\'e Clermont Auvergne, CNRS, LMBP, F-63000 CLERMONT-FERRAND, FRANCE}
 \email{arnaud.guillin@uca.fr}
 
 \author[D. Lu]{\textbf{\quad Di Lu$^{\dag}$  }}
 \address{{\bf Di Lu}. School of Mathematical Sciences, Dalian University of Technology, Dalian 116024,  China}
 \email{diluMath@hotmail.com}
  
\author[B. Nectoux]{\textbf{\quad Boris Nectoux$^{\dag}$  }}
\address{{\bf Boris Nectoux}.Universit\'e Clermont Auvergne, CNRS, LMBP, F-63000 CLERMONT-FERRAND, FRANCE}
 \email{boris.nectoux@uca.fr}

\author[L. Wu]{\textbf{\quad Liming Wu$^{\dag}$ \, \, }}
\address{{\bf Liming Wu}. Universit\'e Clermont Auvergne, CNRS, LMBP, F-63000  CLERMONT-FERRAND, FRANCE, and, 
Institute for Advanced Study in Mathematics, Harbin Institute of Technology, Harbin 150001, China}
\email{Li-Ming.Wu@uca.fr}

\date{\today}

\begin{abstract} 
In this paper, we   prove in a very weak regularity setting existence and uniqueness of quasi-stationary distributions as well as exponential convergence towards the quasi-stationary distribution for  the generalized Langevin and the Nos\'e-Hoover processes, two  processes  which are widely used in molecular dynamics. The case of singular potentials is considered.
With the techniques used in this work, we are also  able to   greatly  improve  existing results on quasi-stationary distributions for the  kinetic Langevin process to a weak regularity setting. 
\end{abstract}


\maketitle
\vspace{-0.2cm}

 \section{Setting and main results}

\subsection{Introduction}

\subsubsection{Purpose of this work}
The basic ingredient in molecular dynamics is a potential energy function $V:(\mathbf R^d)^N\to [1,+\infty]$ which associates to a set of position coordinates of $N$ $\mathbf R^d$-valued particles the energy of the system.   With this function $V$, several  continuous state space   models exist. When the system is  thermostated,  the time evolution of the position-velocity pair  $(x_t=(x^1_t,\ldots, x^N_t),v_t=(v^1_t,\ldots, v^N_t))\in (\mathbf R^d)^N\times (\mathbf R^d)^N$ of the particles  is commonly described by the so-called kinetic Langevin process, which is the solution in $(\mathbf R^d)^N\times (\mathbf R^d)^N$ to the equation
\begin{equation}\label{eq.K}
 dx_t=v_tdt, \ dv_t=-\nabla V(x_t)dt-\gamma v_tdt +\sqrt{2\gamma}\,  dW_t, 
 \end{equation}
 where $(W_t,t\ge 0)$ is  a  standard $(\mathbf R^d)^N$-Brownian motion and $-\nabla V$ is the force field. Due to energetic barriers, the position process $(x_t,t\ge 0)$ remains trapped   for very long times in a  basin of attraction  $ \mathscr B_V(x_*)$ of some  local minimum $x_*$ of $V$ for the dynamics $\dot x=-\nabla V(x)$  in $(\mathbf R^d)^N$. The process~\eqref{eq.K} is therefore said to be metastable and  such  subdomains  of $(\mathbf R^d)^N$, which typically describe  the macroscopic states of the system, are the metastable regions. The move from one metastable region to another is typically related to a macroscopic transition.
 The metastable phenomenon    precludes us from having access to the macroscopic transitions by simulating directly the trajectories of the process \eqref{eq.K} since such  transitions occur over very  long  periods of  time.   In simulations in molecular dynamics, many algorithms have been designed to have   access to macroscopic transitions such as e.g. the powerful and widely used   \textit{accelerated dynamics algorithms} introduced by  A.F. Voter \textit{$\&$ al.}~\cite{sorensen-voter-00,voter-98,bai2010efficient,perez2009accelerated}.
 Recently,   it has been shown~\cite{le2012mathematical,perez2016long,di-gesu-lelievre-le-peutrec-nectoux-17,DLLN,lelievre2015accelerated,ramil2022mathematical} that the notion of quasi-stationary distribution  (see Definition \ref{QSD})  is the cornerstone  to analyse the  mathematical foundations of these accelerated dynamics algorithms. 
    For that reason, the study of quasi-stationary distributions for metastable continuous state space   model has     recently  attracted a lot of attention, especially for the hypoelliptic and non reversible process \eqref{eq.K}, see indeed~\cite{guillinqsd,guillinqsd2,ramilarxiv1,ramilarxiv2,benaim2021degenerate,lelievre2022estimation}.
 
In this work, we will  address   the question of existence, uniqueness, and exponential convergence to the quasi-stationary distribution for   two other widely  used, for example in molecular dynamics, metastable  continuous state space   models which are structurally more complicated than \eqref{eq.K}: the Mori-Zwanzig  Markovian approximation of the generalized Langevin process (see \eqref{eq.GLintro} below) and the Nos\'e-Hoover process (see  \eqref{eq.NHintro} below), which we   introduce now.  

 \begin{sloppypar}
The generalized Langevin process, namely the solution $(x_t=(x^1_t,\ldots, x^N_t),v_t=(v^1_t,\ldots, v^N_t), t\ge 0) \in (\mathbf R^d)^N\times (\mathbf R^d)^N$ to the  integro-differential equation 
\begin{equation}\label{eq.G1}
 dx_t=v_tdt, \ dv_t=-\nabla V(x_t)dt-\gamma v_tdt + \int_0^t\mathsf K(t-s)v_s ds \, dt + F_tdt+\sqrt{2\gamma}\,  dW_t, 
 \end{equation}
has been derived to describe the evolution in time of a system of $N$ $\mathbf R^d$-valued particles interacting with a heat bath (see~\cite{zwanzig1973nonlinear,zwanzig2001nonequilibrium,pavliotisbook,mori1965continued,mori1965transport}). Here $\gamma\ge 0$ is the friction paramater. The diagonal square bloc matrix $\mathsf K_t= {\rm diag }(\mathsf K^1_tI_{(\mathbf R^d)^N},\ldots, \mathsf K^N_tI_{(\mathbf R^d)^N})$ is the memory kernel, where $\mathsf K^i_t\in \mathbf R_+$, is a characteristic of the bath, which encodes the long memory property of the generalized Langevin process. The process $(F_t=(F^1_t,\ldots,F_t^N)\in (\mathbf R^d)^N,t\ge 0)$ is a stationary zero mean   Gaussian stochastic forcing and the fluctuation-dissipation principle  writes $\mathbf E[F^i_s\cdot F^i_u]=\mathsf K^i(|s-u|)$. 
  The generalized Langevin process, which is non Markovian,  is one  of the standard models in  nonequilibrium statistical mechanics~\cite{stella2014generalized,lindenberg1990nonequilibrium,ness2015applications} and is widely used in many areas of   science    
such as    surface scattering~\cite{ala2002collective,doll1975generalized}, polymer dynamics~\cite{snook2006langevin}, sampling in molecular dynamics~\cite{ceriotti1,ceriotti2}, and global optimization using  simulated annealing~\cite{gidas1985global,chak2023generalized}.   
When  $\mathsf K^i_t =\sum_{l=1}^{k_i} \lambda_{i,l}^2e^{-\alpha_{i,l}t}$ (where $\alpha_i,\lambda_{i,l}>0$ for all $i\in \{1,\ldots,N\}$ and  all $l\in \{1,\ldots,k_i\}$, $k_i\ge 1$),  the process \eqref{eq.G1} is quasi-Markovian, i.e. 
it can be   written as a Markovian process by adding a finite number of additional variables \cite{pavliotisbook}  (see  also~\cite[Section 1.2]{ottobre2012asymptotic},~\cite{glatt}, and references therein). More precisely, for such $\mathsf K^i$'s, \eqref{eq.G1} is  equivalent to a Markovian system of stochastic differential equations: 
\begin{equation}\label{eq.A1}
  \left\{
    \begin{array}{ll}
    d x^{i}_t&= v^{i}_tdt \\
       dv^{i}_t&= -\nabla_{x_i}V(x_t)dt-\gamma v^{i}_tdt+ \sum_{l=1}^{k_i} \lambda_{i,l} z^{i,l}_tdt + \sqrt{2\gamma} \, dW^{i}_t  \\
      dz^{i,l}_t&= -\alpha_{i,l}z^{i,l}_tdt-\lambda_{i,l} v^{i}_tdt + \sqrt{2\alpha_{i,l}}\,  dB^{i,l}_t,
       \end{array}
\right.
\end{equation}
   where for $i\in \{1,\ldots,N\}$ and  $l\in \{1,\ldots,k_i\}$, $(x^{i}_t,v^{i}_t,z^{i,l}_t)\in \mathbf R^d\times \mathbf R^d\times \mathbf R^d$, and where the $(W^i_t,t\ge 0)$'s  and  the $(B^{i,l}_t,t\ge 0)$'s are mutually  independent $\mathbf R^d$-Brownian motions. We also mention that for certain kinds of memory kernels $\mathsf K$, \eqref{eq.G1}  can be conveniently approximated by the finite dimensional systems~\eqref{eq.A1}, see e.g. \cite{pavliotisbook,kupferman2004fractional,ottobre2011asymptotic,ness2015applications}. 
For ease of expository and since all our proofs extend trivially to the case when $k_i>1$,  we have decided in this work  to only consider  the case when  $k_i=1$ for all $i\in \{1,\ldots,N\}$. In this case,~\eqref{eq.A1} writes:
\begin{equation}\label{eq.GLintro}
  \left\{
    \begin{array}{ll}
        dx_t&= v_tdt \\
        dv_t&= -\nabla V(x_t)dt-\gamma v_tdt+\lambda z_tdt + \sqrt{2\gamma} \, dW_t\\
        dz_t&= -\alpha z_tdt-\lambda v_tdt +  \sqrt{2\alpha}\, dB_t,
    \end{array}
\right.
\end{equation}
where $(x_t,v_t,z_t)\in (\mathbf R^d)^N\times (\mathbf R^d)^N\times  (\mathbf R^d)^N$, $\alpha, \lambda>0$, $\gamma \ge 0$, and where  $(W_t,t\ge 0)$   and   $(B_t,t\ge 0)$ are independent $(\mathbf R^d)^N$-Brownian motions. Note that when $\gamma=0$, which is the  case often   used in practice, the process \eqref{eq.GLintro} is more degenerated than \eqref{eq.K} in the sense that the noise does not act on the positions and on the velocities but  on an auxiliary variable. 
With a slight abuse of language,~\eqref{eq.GLintro} will be referred in this work       as the generalized Langevin process.  
 \end{sloppypar}

The second  process we will consider is the  Nos\'e-Hoover process which is 
another extension of the kinetic Langevin process \eqref{eq.K}. 
The  Nos\'e-Hoover process  is the  solution  on $(\mathbf R^d)^N\times (\mathbf R^d)^N\times  \mathbf R$ to the following stochastic differential equation: 
  \begin{equation} \label{eq.NHintro}
\left\{
\begin{aligned}
 d  x_t &= v_t dt\\
dv_t&=-\nabla V(x_t)dt - \gamma v_t dt -  v_ty_tdt+  \sqrt{2\gamma} \, dB_t\\
dy_t&= |v_t|^2 dt -  d N dt,
\end{aligned}
\right.
\end{equation}
where $\gamma>0$. 
The process \eqref{eq.NHintro}  is  also sometimes  referred as the adaptive Langevin process. It has been introduced in the context of sampling in  molecular dynamics~\cite{KCLS,LR09,jones2011adaptive,leimkuhler2016adaptive}, where $y_t$ acts as a thermostat. See also~\cite{ding2014bayesian,shang2015covariance} for applications in Bayesian sampling. The friction parameter in the kinetic Langevin process~\eqref{eq.K}  is   considered in \eqref{eq.NHintro} as a dynamical variable. This stochastic correction, which is renewed  according to a negative feedback loop control law (as in the Nos\'e-Hoover thermostat)  models   random perturbations of unknown magnitude which can occur on the  potential gradient, and  serves as variable which restores the canonical distribution associated with the prescribed inverse temperature, see indeed~\cite{stoltzadaptive}.

 Studying these two metastable processes, via their quasi-stationary distribution, is thus of paramount importance for practical applications, e.g. in molecular dynamics. 
 
Since $V$  might have singularities, we  consider  a connected component  $\mathbf O_V$ of  the set $\{x\in (\mathbf R^d)^N, V(x)<+\infty\}$. Then, if  collisions between particles  do not occur, which will be the case in this work, the two processes~\eqref{eq.GLintro} and~\eqref{eq.NHintro}  evolve   on the state space
 $$\mathscr E= \mathbf O_V\times \mathbf R^m,$$ 
 where $m=Nd+Nd$ for the generalized Langevin process \eqref{eq.GLintro} and $m=Nd+1$ for the Nosé-Hoover process \eqref{eq.NHintro}.

In this work, we prove, for the two processes \eqref{eq.GLintro} and \eqref{eq.NHintro}, existence and uniqueness in some weighted spaces of the quasi-stationary distribution on metastable domains $\mathscr D$, which are regions of the forms~$\mathscr D=\mathscr O\times \mathbf R^m$, where~$\mathscr O$ is a subdomain (i.e. a  nonempty, connected, and open subset) of $\mathbf O_V$, bounded or not,  
as well as the exponential convergence towards the quasi-stationary distribution.  Our main results  are Theorem~\ref{th.1} (and its extension to non gradient force fields,  Theorem \ref{th.1-NE}), and Theorems~\ref{th.2}, and~\ref{th.3}. We also mention    Theorem~ \ref{th.K} for a significant extension of  existing results on the quasi-stationary distribution for  the kinetic Langevin  process~\eqref{eq.K}.

\subsubsection{Contributions}

The first main contribution of this work is the very weak regularity setting we consider on both the domain $\mathscr D$ and the force field (see items $\mathbf{a}$ and $\mathbf{b}$ just below). One of the  main novelty of this setting, compared to previous works on quasi-stationary distributions for  hypoelliptic  degenerate  processes~\cite{guillinqsd,guillinqsd2,ramilarxiv1,ramilarxiv2,ramilarxiv3,lelievre2022estimation} on such domains~$\mathscr D$,  is that  
we have managed to get rid of any regularity assumption on the boundary of~$\mathscr O$. This is a paramount improvement in order to   treat the  cases which are  considered in practice where~$\mathscr O$ is defined as a basin  of attraction $\mathscr B_V(x_*)$ (see e.g.~\cite{perez2015b,perez2016long,di-gesu-lelievre-le-peutrec-nectoux-17,ramil2022mathematical}), or as   (the interior of the closure of the) union of neighboring basins  of attraction. In this case,  the properties of $\partial \mathscr O$ are very arduous to infer, and even worse, it is known that in this case  $\partial \mathscr O$ is far from being regular\footnote{E.g. $\partial \mathscr O$ may have corners and/or may contain points at which it does not satisfy the exterior sphere condition. Worst geometric situations can also occur for the boundary of such domains.}, even for smooth potentials $V$.  \\
More precisely, in the weak regularity we consider is the case when: 
\begin{enumerate}
\item[$\mathbf{a}$.] $\mathscr O$ is any subdomain of $\mathbf O_V$ (note that there is thus no assumption on the regularity of $\partial \mathscr O$, and that 
 $\partial \mathscr O$  can also intersect  the set $\partial \mathbf O_V$ where the potential $V$ is  infinite, due to collisions between particles),
\item[$\mathbf{b}$.]  the force field $-\nabla V$ is  only locally Lipschitz over $\mathbf O_V$ and infinite on $\partial \mathbf O_V$.
\end{enumerate}
In this previous setting, we will  in particular prove  that:
\begin{itemize}
\item[$\mathbf 1$.]  The nonkilled semigroup, defined in \eqref{eq.KS-0}, and the killed semigroup, defined in \eqref{eq.KS}, are respectively  strong Feller and weak Feller (actually it will be shown that the latter is also strong Feller). These are respectively  the conditions \textbf{(C1)} and \textbf{(C4)} in Section~\ref{sec.nota}.
\item[$\mathbf 2$.] The killed semigroup is  topologically irreducible (this is  the condition  \textbf{(C5)}  below in Section~\ref{sec.nota}).
\end{itemize}
 In the weak regularity setting $\mathbf{a}$ and $\mathbf{b}$,  the tools used  to prove Item $\mathbf 1$ in~\cite{guillinqsd,guillinqsd2} for the kinetic process~\eqref{eq.K} are not adapted anymore. We   will   rather rely on a different approach to  prove  Item $\mathbf 1$. This approach, which is explained in details in Section~\ref{sec.22},  is  fundamentally based on   the   \textit{energy splitting}~Equation~\eqref{eq.splitting}  together with  the   analysis   of the behavior of the process  at low and at high energy $\mathsf H$  (where~$\mathsf H$ denotes  the Hamiltonian  of the processes).  

On the other hand, to ensure existence of quasi-stationary distributions and exponential convergence to it, we also have to prove
 \begin{itemize}
\item[$\mathbf 3$.] Enhanced exponential integrability of hitting times of larger and larger sets (see the condition \textbf{(C3)} below).
\end{itemize}
The second main contribution of this work is  that we  consider singular potentials  $V$
  for both processes and ensure that Item $\mathbf 3$ still holds. 
  To this end, we have in particular to  construct Lyapunov functions satisfying a strong return from $\infty$ (see \textbf{(C3)}). The Lyapunov functions we construct are bounded from above by~$C\exp ( c\mathsf H^\delta)$, for   $\delta\in (0,1]$,  where we recall that~$\mathsf H$ denotes  the Hamiltonian  of the processes. It will turn  out    that for singular potentials~$V$ and for  the generalized Langevin process, a right choice of Lyapunov function  is  $e^{\mathsf F^\delta}$, where $\mathsf F$ is the modified Hamiltonian introduced in \cite{duong2023asymptotic} (see Section~\ref{sec.GL-singular}). 
However, for  the Nos\'e-Hoover process \eqref{eq.NHintro}, the Lyapunov function   introduced in \cite{herzogNH} does not satisfy \textbf{(C3)}, and for that reason we modify it to  obtain  the asymptotic return from $\infty$ in the position variable $x$ and also to obtain smaller Lyapunov functions.

  The starting point of our analysis is~\cite[Theorem~2.2]{guillinqsd} when we will consider  the generalized Langevin process, whereas, for the Nosé-Hoover process,  we will rely on a   more general result, namely Theorem~\ref{th.ext} (stated and proved in Section \ref{sec.extC5}). This extension  is required since   the  Nosé-Hoover process   does not satisfy the condition \textbf{(C5)} (defined at the end of Section \ref{sec.nota}).


 \subsubsection{On the approach to check the regularity conditions}

The approach we  use  in this work to check the regularity conditions \textbf{(C1)}, \textbf{(C2)}, and \textbf{(C4)} is introduced in  Section \ref{sec.22}. It is  particularly well-suited  for    solutions to degenerate (or not) SDEs in $\mathbf R^\ell$ (driven by different kind of noises)  with  non smooth drifts  on   any kind of   subdomains $\mathscr D\subset \mathbf R^\ell$. Indeed, in addition to having been successfully used in this work  for   degenerate kinetic processes driven by Brownian noises,  this approach has also been used   after writing this article   by the authors  to derive \textbf{(C1)}, \textbf{(C2)}, and \textbf{(C4)}  for other types of  processes which are listed   in Remark~\ref{sec.reSF} below. In particular, these are processes for which  standard techniques to derive \textbf{(C1)}, such as the H\"ormander's condition and the Malliavin calculus, do not apply due to the lack of regularity of the coefficients. In some cases,  the drift is indeed  not smooth enough or even worst, it becomes    infinite on some parts of the state space which are almost surely visited by the process~\cite{guillin2024Dyson}.  
  We finally mention that this approach will be also employed in a future work to study the existence and uniqueness of (quasi- or not) stationary distributions for kinetic  SDEs with jump-noises.
  \subsubsection{Improving existing results on the kinetic Langevin process}
Using the  approach introduced in  Section \ref{sec.22} and the tools used in this work, we can  also considerably    improve 
 existing results on the existence and uniqueness of quasi-stationary distributions for the kinetic Langevin process~\cite[Theorems 2.4 and 3.2]{guillinqsd2} with singular potentials, as well as those obtained~\cite{ramilarxiv2}: they are   valid without any regularity assumption on the boundary of the subdomain $\mathscr D$ (see respectively Theorems~\ref{th.K} and~\ref{th.Ks} below) and when the drift is only locally Lipschitz. As explained, such extensions are of real interest to match with the non smooth domains, namely the basins of attraction $\mathscr B_V(x^*)$ of $\dot x=-\nabla V(x)$, on which is considered the quasi-stationary distribution in practice. We also refer to Theorems \ref{th.ell} and~\ref{th.ell2} when elliptic processes are considered.

\subsubsection{Related results}
The ergodic properties of the  nonkilled semigroup of the generalized Langevin  process  is now well-known, see for instance~\cite{ottobre2011asymptotic,ottobre2012asymptotic,pavliotis2021scaling,Sachs,glatt,duong2023asymptotic} and references therein (see also~\cite{carmona2007existence,monmarche2023almost,herzog2023gibbsian}).  
On the other hand, the long time behavior of nonkilled semigroup of the Nos\'e-Hoover process has been studied in~\cite{herzogNH}, where we also mention that the case of Lennard-Jones type potentials is also considered there (see also~\cite{stoltzadaptive} for a similar study for the adaptive Langevin process, through a hypocoercivity analysis).  
We also refer to~\cite{bony2022eyring} (resp.~\cite{lois}) where the metastable behavior of the process \eqref{eq.GLintro} (resp. \eqref{eq.NHintro}) has been studied through the derivation of sharp asymptotic equivalents in the small temperature regime of the smallest eigenvalues of  its generator (see also~\cite{herauEK}). 
 
Several general criteria have been introduced in the mathematical literature to ensure existence and uniqueness of quasi-stationary distributions, see~\cite{meleard2012quasi,champagnat2017general} and references therein (see also the recent works~\cite{arnaudon2024lyapunov,velleret2022unique,sanchez2023krein}). 
When dealing with biological systems, we refer e.g. to~\cite{cattiaux09,diaconis2009times,cattiaux2010competitive,chazottes2016sharp,kolbQSDpersistance,champagnat2021lyapunov}.  As already mentioned, quasi-stationary distributions  for the kinetic Langevin process~\eqref{eq.K} have been studied recently in~\cite{guillinqsd,guillinqsd2,ramilarxiv2,benaim2021degenerate} (see also~\cite{tough2022infty}). 

\subsubsection{Notation}\label{sec.nota}
Let  $(\Omega, \mathcal F, (\mathcal F_t)_{t\ge 0}, \mathbf P)$ be a  filtered probability space  (where the filtration satisfies the usual condition). 
Consider  $(X_t,t\ge 0)$  a time  homogeneous  strong Markov process valued in  $\mathscr E$ with continuous sample paths, where  $\mathscr E$ is a nonempty open subset of $\mathbf R^k$, $k\ge 1$. 
Let $\mathcal B(\mathscr E)$ be the Borel $\sigma$-algebra of $\mathscr E$, $b\mathcal B(\mathscr E)$ the space of all bounded and Borel measurable (real-valued) functions $f$ on $\mathscr E$,  $ \mathcal C^b(\mathscr E)$  the space of all bounded  and continuous  (real-valued)  functions   on $\mathscr E$, and $\mathcal P (\mathscr E)$ as the set of probability measures  on $\mathscr E$.   We will denote by $\mathsf L_{{\rm Lip}}^{1, {\rm loc}}(\mathscr E)$ the set of functions $V:\mathscr E\to \mathbf R$ such that over $\mathscr E$, $V$ is   
differentiable  and     $\nabla V$ is locally Lipschitz.   
The transition  probability semigroup is denoted by $(P_t,t\ge 0)$, i.e. 
\begin{equation}\label{eq.KS-0}
P_tf(\mathsf x)=\mathbf E_{\mathsf  x}[f(X_t)], \ f\in b\mathcal B(\mathscr E), \ \mathsf x\in \mathscr E.
\end{equation}
  $(P_t,t\ge 0)$ will be referred as the nonkilled semigroup (where by definition,  $\mathbf E_{\mathsf  x}[f(X_t)]=\mathbf E[f(X_t)| X_0=\mathsf  x]$). 
 The space $\mathcal C([0,T],\mathscr E)$ of $\mathscr E$-valued continuous functions   defined on $[0,T]$ is endowed with the  sup-norm over $[0,T]$. When $\mathscr D$ is a nonempty open subset of $\mathscr E$, we  denote by  $(P_t^{\mathscr D},t\ge 0)$  the semigroup of the killed process $(X_t,t\ge 0)$:
\begin{equation}\label{eq.KS}
P_t^{\mathscr D}f(\mathsf x)=\mathbf E_{\mathsf  x}[f(X_t)\mathbf 1_{t<\sigma_{\mathscr D}}],  \ f\in b\mathcal B (\mathscr D), \ \mathsf  x\in \mathscr D,
\end{equation}
 where $\sigma_{\mathscr D}=\inf\{t\ge 0, X_t\notin \mathscr D\}$
 is the first exit time of the process   $(X_t,t\ge 0)$  from~$\mathscr D$  ($(P_t^{\mathscr D},t\ge 0)$ will be referred as the killed semigroup). 
For $\mathsf W:\mathscr E\to [1,+\infty)$ and when $\mathscr D$ is open subset of $\mathscr E$, we define the set $\mathcal P_{\mathsf W}(\mathscr D)$ as the set of probability measures $\nu$ over  $\mathscr D$ such that $\nu(\mathsf W)<+\infty$.  We also define $b_{\mathsf W}\mathcal B (\mathscr D)$ the Banach space of measurable functions $f:\mathscr D\to \mathbf R^d$ such that $f/\mathsf W$ is bounded over~$\mathscr D$ (its norm is denoted by $\Vert f \Vert_{\mathsf W}=\sup_{\mathscr E}|f/\mathsf W|$).  We finally denote by $(\mathcal L ,\mathbb D_e(\mathcal L))$ the extended generator of $(X_t,t\ge 0)$ (for a definition, see \cite{guillinqsd,guillinqsd2} and references therein).  
 
As already mentioned, to prove existence and uniqueness of quasi-stationary distributions for the generalized Langevin process, we will use~\cite[Theorem~2.2]{guillinqsd} (for the Nos\'e-Hoover process, we will use an extension of~\cite[Theorem~2.2]{guillinqsd}, see Theorem~\ref{th.ext} below). To this end, we  recall the  conditions introduced in~\cite{guillinqsd} which ensure existence and uniqueness of the quasi-stationary distribution as well as the exponential convergence  in $\mathcal P_{\mathsf W^{1/p}}(\mathscr D)$ towards this quasi-stationary distribution:
\begin{enumerate}
\item[{\bf (C1)}]  There exists $t_0>0$ such that for each $t\ge t_0$, $P_t$ is strong Feller.

\item[{\bf (C2)}]  For every $T>0$, $\mathsf x\in \mathscr E \mapsto\mathbf P_{\mathsf x}(X_{[0,T]}\in \cdot)$  (the law of $X_{[0,T]}:=(X_t)_{t\in [0,T]}$) is continuous from $\mathscr E$ to the space $\mathcal P(\mathcal C([0,T], \mathscr E))$ of probability measures on  $\mathcal C([0,T], \mathscr E)$, equipped with the weak convergence topology.

\item[{\bf (C3)}] There exist a continuous  function $\mathsf W : \mathscr E\to [1,+\infty)$,   with $\mathsf W\in \mathbb D_e(\mathcal L)$,   two sequences of positive constants $(r_n)$ and $(b_n)$ where $r_n\to +\infty$, and an increasing sequence of compact subsets $(K_n)$ of $\mathscr E$, such that
$$
-\mathcal L\mathsf W ({\mathsf x}) \ge r_n \mathsf W({\mathsf x}) - b_n \mathbf 1_{K_n}({\mathsf x}), \ \text{ quasi-everywhere.} 
$$
\item[{\bf (C4)}]  For all $t\ge 0$ and all   $f\in \mathcal C_c^\infty(\mathscr D)$,  $P_t^{\mathscr D}f\in \mathcal C^b(\mathscr D)$.

\item[{\bf (C5)}] There exists $t_1>0$, for all $t\ge t_1$,  ${\mathsf x}\in \mathscr D$ and nonempty open subset $ \mathsf O$ of $\mathscr D$, $
P^{\mathscr D}_t({\mathsf x},\mathsf O)>0$. In addition, there exists  $\mathsf x_0\in \mathscr D$ such that $\mathbf P_{\mathsf x_0}(\sigma_{\mathscr D}<+\infty)>0$.

\end{enumerate}
We recall also the definition of a quasi-stationary distribution (see for instance the classical textbook~\cite{collet2012quasi}).
 \begin{defn}\label{QSD}
A quasi-stationary distribution   of the Markov process $(X_t,t\ge 0)$ in $\mathscr D$ is a probability measure $\mu_{\mathscr D}$ on $\mathscr D$
such that $ 
\mu_{\mathscr D}(\mathcal A)=\mathbf P_{\mu_{\mathscr D}}(X_t\in \mathcal A| t<\sigma_{\mathscr D})$,
 $\forall t>0$, $\forall \mathcal A\in \mathcal B(\mathscr D)$, 
 where $\mathcal B(\mathscr D):=\big \{\mathcal A\cap \mathscr D, \mathcal A\in\mathcal B(\mathscr E)\big \}$.
 \end{defn}
We end this section by recalling a powerful  general convergence result  which will be used many times in this work. 
\begin{prop}(\cite[Lemma 3.2]{wu1999}) 
\label{pr.wu1999}
Assume that a sequence of random variables $(Y_n)_{n\ge 0}$ defined on a probability space $(\Omega, \mathcal F, \mathbf P)$ with values in a Polish space $\mathscr S$    converges in $\mathbf P$-probability to $Y$. Assume also that there exists   a  fixed probability measure $\mu$ such that   for all $n\ge 0$, the law of $Y_n$ writes   $\mathbf P[Y_n\in d\mathsf x ]=p_n(\mathsf x) \mu(  d\mathsf x )$. If $(p_n)_n$ is   uniformly integrable with respect to  $\mu$, then as $n\to +\infty$, $f(Y_n)\to f(Y)$ in $\mathbf P$-probability for all measurable function $f:\mathscr S\to \mathbf R$. 
\end{prop}

\subsection{Quasi-stationary distributions for the  generalized Langevin process}

In this section, we study existence and uniqueness of the quasi-stationary distribution for the generalized Langevin equation~\eqref{eq.GLintro}. We first consider in Section \ref{sec.GLTD}, the case when $N=1$ and  $\nabla V$ is locally Lipschitz on $\mathbf R^d$ (see Theorem~\ref{th.1}). In Section \ref{sec.GL-I}, we then consider the case of $N\ge 2$   particles evolving according to  the generalized Langevin equation and interacting through singular potentials (see Theorem~\ref{th.2}). The 
proof of Theorem~\ref{th.1} 
 when $N=1$ and when  $\nabla V$ is locally Lipschitz over $\mathbf R^d$ is very instructive, which explains our choice to split  this section.   

\subsubsection{Generalized Langevin process with locally Lipschitz  drifts} 
 
 \label{sec.GLTD}

 In this section, $N=1$ and $V:\mathbf R^d\to [1,+\infty)$ (thus  $\mathbf O_V=\mathbf R^d$),  and  we  consider the  process $(X_t=(x_t,v_t,z_t),t\ge 0)$  solution  in $\mathscr E=\mathbf R^{d}\times \mathbf R^{d}\times \mathbf R^{d}$ to the generalized Langevin equation \eqref{eq.GLintro} when $N=1$ (see Proposition \ref{pr.existence}). Recall that   $\lambda,\alpha>0$, $\gamma\ge 0$.
The basic assumption of this section is the following.
\medskip

\noindent
\textbf{Assumption}  \textbf{[V{\tiny loc}]}.\textit{
 $N=1$ and  $V:\mathbf R^d\to [1,+\infty)$  belongs to $\mathsf L_{{\rm Lip}}^{1, {\rm loc}}(\mathbf R^d)$. In addition $V$ is  
 coercive (i.e. $V(x)\to +\infty$ as $|x|\to +\infty$).}
\medskip
 
In the following, we simply denote $\mathbf R^{d}\times \mathbf R^{d}\times \mathbf R^{d}$ by $ \mathbf R^{3d}$. 
We denote the Hamiltonian function of the process \eqref{eq.GLintro} by 
\begin{equation}\label{eq.H-GL}
\mathsf H_{{\rm GL}}: (x,v,z)\in \mathbf R^{3d}\mapsto  V(x)+\frac 12 |v|^2+ \frac 12 |z|^2,
\end{equation}
and its infinitesimal generator
\begin{equation}\label{eq.L-GL}
\mathcal L_{{\rm GL}}= v\cdot \nabla _x  + (-\gamma v -  \nabla _x V +\lambda z)\cdot \nabla_v + \gamma \Delta_v- (\alpha z+\lambda v)\cdot \nabla _z+ \alpha \Delta_z.
\end{equation}

 
  \begin{prop}\label{pr.existence}
  Assume {\rm\textbf{[V{\tiny loc}]}}.
For all $\mathsf x_0\in \mathbf R^{3d}$, there exists a  unique strong solution  $(X_t=(x_t,v_t,z_t),t\ge 0)$ to \eqref{eq.GLintro} such that $X_0=\mathsf x_0$ (this process is denoted by $(X_t(\mathsf x_0),t\ge 0)$). In addition, $(X_t(\mathsf x_0),t\ge 0)$ is  a strong Markov process.
  \end{prop}

 \begin{proof} Note that the coefficients in \eqref{eq.GLintro} are locally Lipschitz. 
  The proof of the non explosion is very standard, but we write it since Equation~\eqref{eq.tauR} will play a key role in this work. Let $c>0$   such that $\mathcal L_{{\rm GL}} \mathsf H_{{\rm GL}} \le c \mathsf H_{{\rm GL}} $ over $\mathbf R^{3d}$. Set for $R\ge 0$, $\sigma_{\mathscr H_R}:= \inf \{t\ge 0, X_t\notin \mathscr H_R\}$, 
   where
  $\mathscr H_R:=\{\mathsf x\in  \mathbf R^{3d}, \mathsf H_{{\rm GL}}(\mathsf x)<R\}$ is open and bounded (since $V$ is coercive). 
      Since $ \mathsf H_{{\rm GL}}(X_{\sigma_{\mathscr H_R}})= R$, it holds by the Itô formula, for all $\mathsf x\in \mathscr H_R$, $R> 0$:  
    \begin{equation}\label{eq.tauR}
\mathbf P_{\mathsf x}[\sigma_{\mathscr H_R}\le t]\le \frac{e^{ct} }{R} \,  \mathsf H_{{\rm GL}}(\mathsf x), \ \forall t\ge 0.
\end{equation}
The proof is complete.
 \end{proof}

  In addition to \textbf{[V{\tiny loc}]}, we will assume the following growth condition on $V$, which basically implies that $V(x)$ behaves like $|x|^k$ as $|x|\to +\infty$, for some $k>1$. 
\medskip

\noindent
\textbf{Assumption}  \textbf{[V{\tiny poly-$x^k$}]}.  \textit{In addition to {\rm\textbf{[V{\tiny loc}]}}, there exist  $k>1 $ and  $M_V,r_V,c_V>0$  such that for all $x\in\mathbf R^{d}$ with $|x|\ge r_V$:}
$$
 \textit{$ c_V|x|^k \leq V(x)\leq M_V |x|^{k} \text{ and } 
   c_V|x|^{k}\leq x\cdot\nabla V(x).$}
$$  
\textit{When $\gamma=0$, we assume moreover that $|\nabla V(x)|\le M_V|x|^{k-1}$ if $|x|\ge r_V$.}
 \medskip

The first main result of this work is the following (see also its note  just below).

 \begin{thm}\label{th.1}
 Assume {\rm\textbf{[V{\tiny loc}]}} and {\rm\textbf{[V{\tiny poly-$x^k$}]}}. 
 Let $\mathscr D=\mathscr O\times \mathbf R^d\times \mathbf R^d$ where $\mathscr O$ is a subdomain of $\mathbf R^d$ (not necessarily smooth nor bounded) such that $ \mathbf R^d\setminus \overline{\mathscr O}$ is nonempty. Assume moreover that $k\in (1,2]$ when $\gamma=0$. For  $\delta \in ((1-\beta)/k,1]$ fixed, let $\mathsf W_\delta:\mathbf R^{3d}\to [1,+\infty)$ be  the Lyapunov function defined in \eqref{eq.Lyapunov2}  where the parameter $\beta>0$ satisfies \eqref{eq.Bbeta} if $\gamma>0$, and  \eqref{eq.Bbeta2} when $\gamma=0$. Then, $\mathbf P_{\mathsf x}(\sigma_{\mathscr D}<+\infty)=1$ for all ${\mathsf x}\in \mathscr D$, and for all $p\in (1,+\infty)$:
 \begin{enumerate}
 \item [i)] There exists a unique quasi-stationary distribution $\mu_{\mathscr D}^{(p)}$ for the process \eqref{eq.GLintro} on $\mathscr D$  in the space $\mathcal P_{\mathsf W_\delta^{1/p}}(\mathscr D)$. 
 \item [ii)] There exists $\lambda_{\mathscr D}^{(p)}>0$  such that for all $t\ge0$, the  spectral radius of $P_t^{\mathscr D}$ equals  
 $$\mathsf r_{sp}(P_t^{\mathscr D}|_{b_{\mathsf W_\delta^{1/p}} \mathcal B(\mathscr D)})=e^{-\lambda_{\mathscr D}^{(p)} t}.$$ In addition, for all $t\ge 0$,  $\mu^{(p)}_{\mathscr D} P_t^{\mathscr D}=e^{-\lambda_{\mathscr D}^{(p)} t}\mu^{(p)}_{\mathscr D}$  and $\mu_{\mathscr D}^{(p)}(\mathsf O)>0$ for all nonempty open subsets $ \mathsf O$ of $\mathscr D$. Furthermore,  there is a unique continuous function $\varphi^{(p)}$ in  $b_{\mathsf W_\delta^{1/p}}\mathcal B(\mathscr D)$   such that   
   $$
  \varphi^{(p)}>0 \text{ on } \mathscr D, \  \mu^{(p)}_{\mathscr D}(\varphi^{(p)})=1,\ \text{ and }   \ P_t^{\mathscr D} \varphi^{(p)}= e^{-\lambda_{\mathscr D}^{(p)} t} \varphi^{(p)} \text{ on }\mathscr D, \ \forall t\ge0.
$$

 \item [iii)] There exist  $M>0$ and $C\ge 1$ such that:
$$
\sup_{\mathcal A\in\mathcal B(\mathscr D)}\big |\mathbf P_\nu[X_t\in \mathcal A| t<\sigma_{\mathscr D}]-\mu^{(p)}_{\mathscr D}(\mathcal A)\big |\le C e^{-M t} \frac{\nu(\mathsf W_\delta^{1/p})}{\nu(\varphi^{(p)})}, \forall t>0,\forall  \nu\in \mathcal P_{\mathsf W_\delta^{1/p}}(\mathscr D).
$$

 \end{enumerate}


 \end{thm} 
 



\medskip

\noindent
 \textbf{Comments on Theorem \ref{th.1}}.  
 Let us recall (see \cite{guillinqsd}) that $\mu_{\mathscr D}^{(p)}$ is independent of $p$, i.e. $\mu_{\mathscr D}^{(p)}=\mu_{\mathscr D}^{(q)}$  and $\lambda_{\mathscr D}^{(p)}=\lambda_{\mathscr D}^{(q)}$ for any $p,q>1$. 
 We now discuss how to choose $\delta>0$ depending on the values of the growth parameter $k>1$ of $V$  in \textbf{[V{\tiny poly-$x^k$}]}, and how $\mathsf W_\delta$  behaves  at high energy.
  \begin{enumerate}
 \item [\textbf -]  When $\gamma>0$ (resp. $\gamma=0$) and \eqref{eq.Bbeta} holds  (resp.  \eqref{eq.Bbeta2} holds), there exists $c>0$ such that  $\mathsf W_\delta\le \exp[ c^\delta \mathsf E^\delta]$, where $\mathsf E(x,v,z)=1+|x|^k+  |v|^2+ |z|^2$.   Hence, the smaller $\delta\in ((1-\beta)/k,1]$ is (see \eqref{eq.Lyapunov2}), the smaller $\mathsf W_\delta$ is. 
Moreover,  for any probability measure $\mu(d\mathsf x)=g(\mathsf x)d\mathsf x$ such that   for some $\kappa,r,K,R>0$,  $g \le K\exp[ -\kappa  \mathsf E^r ] $  on    $\{\mathsf x\in \mathbf R^{3d}, \mathsf E(\mathsf x)>R\}$, it holds: 
$$\int_{\mathbf R^{3d}} \mathsf W_\delta^{1/p} (\mathsf x)\mu(d\mathsf x)<+\infty \text{ for any  $r>\delta$ and $p>1$}.$$  
 \item [\textbf -] Let us first consider the case when $\gamma>0$.  In view of  \eqref{eq.Bbeta}, when $k\ge 2$, one can choose any $\delta\in (0,1)$ in the definition of $\mathsf W_\delta$ in~\eqref{eq.Lyapunov2}. Indeed, in this case   $\min(1,k/2,k-1)=1$ and therefore $\beta>0$ can be chosen as close as desired (from below) to $1$.   
 Let us now consider the case when $\gamma=0$ (recall in this case that $k\in (1,2]$). When $k= 2$, $\beta=1$ and therefore  one can choose any $\delta\in (0,1)$ in the definition of $\mathsf W_\delta$ in~\eqref{eq.Lyapunov2}. 
 \end{enumerate}
 
  The proof of Theorem~\ref{th.1} is given in Section \ref{sec.2V}. 
    \medskip
  
 \noindent
  \textbf{Extension   to non gradient force fields}.  
  In several cases in molecular dynamics the force field $-\nabla V$ is subject to nonequilibrium perturbations (the dynamics is then said to be out of equilibrium~\cite{reimann2002brownian,lelievre2016partial,leimkuhler2016computation}). This is the setting  we would like to treat here. To this end we consider the following assumption on the force field. 
  \medskip

\noindent
\textbf{Assumption}  \textbf{[$\boldsymbol b${\tiny non-gradient}]}.  \textit{
The vector field $\boldsymbol b:\mathbf R^d \to \mathbf R^d$ decomposes over  $\mathbf R^d$ as  $\boldsymbol b=-\nabla V + \boldsymbol \ell$, 
 where  $V$ satisfies {\rm\textbf{[V{\tiny poly-$x^k$}]}}  and  $\boldsymbol \ell :\mathbf R^d \to \mathbf R^d$ is a locally Lipschitz vector field such that for all $x\in \mathbf R^d$, $|\boldsymbol \ell(x)|\le C (|x|^{(k-1)/2}+1)$ for some $C>0$. }
 \medskip
 
 
 \begin{thm}\label{th.1-NE}
When {\rm\textbf{[$\boldsymbol b${\tiny non-gradient}]}} is satisfied, all the assertions of Theorem \ref{th.1} are still valid verbatim for the   solution $(X_t=(x_t,v_t,z_t),t\ge 0)$ to the generalized Langevin equation 
$$
  \left\{
    \begin{array}{ll}
        dx_t&= v_tdt \\
        dv_t&= \boldsymbol b(x_t)dt-\gamma v_tdt+\lambda z_tdt + \sqrt{2\gamma} \, dW_t\\
        dz_t&= -\alpha z_tdt-\lambda v_tdt +  \sqrt{2\alpha}\, dB_t.
    \end{array}
\right.
$$ 
 \end{thm}
 Theorem \ref{th.1-NE} is proved in Section \ref{sec.pr2}. 
 

 
 
 \subsubsection{Generalized Langevin process: extension of the results  to singular   interaction potentials} 
  \label{sec.GL-I}
 In this section, we extend the results of the previous section to a system of $N\ge 2$ particles $X_t=(x_t,v_t,z_t)\in \mathbf R^{dN}\times \mathbf R^{dN} \times \mathbf R^{dN} $  whose positions $x_t=(x_t^1, \ldots, x^N_t)\in \mathbf R^{dN} $, velocities $v_t=(v_t^1, \ldots, v^N_t)\in \mathbf R^{dN} $, and noises $z_t=(z_t^1, \ldots, z^N_t)\in \mathbf R^{dN} $ evolve in time $t\ge 0$ through  the generalized Langevin equation \eqref{eq.GLintro} and which interact through  singular  potentials (note that we simply write $\mathbf R^{dN}$ for $(\mathbf R^d)^N$). Let us more precisely  introduce the assumptions on the potential function   $V:\mathbf R^{dN}\to \overline{\mathbf R}$. 
  \medskip

\noindent
\textbf{Assumption {[V{\tiny coercive}]}}. \textit{The set $\mathbf O_V=\{x\in \mathbf R^{dN}, V(x)<+\infty\}$ is a subdomain  of $\mathbf R^{dN}$ and   $V:\mathbf O_V\to [1,+\infty)$ belongs to $ \mathsf L_{{\rm Lip}}^{1, {\rm loc}}(\mathbf O_V)$. In addition, for all $R>0$,  $\mathscr V_R:=\{y\in \mathbf R^{dN}, V(y)<R\}$ 
has compact closure in $\mathbf O_V$.  Finally,   $|\nabla V(x)|\to +\infty$ when $V(x)\to +\infty$.} 
\medskip

Recall that the state space is   $
\mathscr E= \mathbf O_V\times  \mathbf R^{dN} \times \mathbf R^{dN}$. 
 Note that under \textbf{[V{\tiny coercive}]}, $V$ is coercive, i.e. if  $\mathbf O_V$ is unbounded (resp. bounded), $V\to +\infty$ if and only if $x\to    \{\infty\}   \cup \partial   \mathbf O_V$ (resp. $x\to  \partial \mathbf O_V$).  Thus, without loss of generality, we can   assume that 
  $V\ge 1$ over $\mathbf O_V$ and consequently $\mathsf H_{\text{GL}}\ge 1$ over $\mathscr E$. Assumption \textbf{[V{\tiny coercive}]} is more general than  \textbf{[V{\tiny loc}]} since it   allows to consider singular potentials.

  To prove existence and uniqueness of quasi-stationary distributions in weighted spaces (see Theorem~\ref{th.2} below), the starting point is~\cite[Theorem 2.2]{guillinqsd}. The strategy will thus consists of  proving \textbf{(C1)}$\to$\textbf{(C5)} for the process \eqref{eq.GLintro}.   Assumptions  \textbf{(C1)}, \textbf{(C2)}, \textbf{(C4)}, and \textbf{(C5)} will be   proved for the process \eqref{eq.GLintro} when  \textbf{[V{\tiny coercive}]} holds. However, to construct a Lyapunov function satisfying  \textbf{(C3)}, we will  impose     more explicit assumptions on $V$  than    \textbf{[V{\tiny coercive}]} (see \textbf{[V{\tiny sing1}]} below), that we introduce now.


 \medskip
 

For ease of notation we simply denote by  $\{x^i\neq x^j\}$  the set  $ \{x=(x^1,\ldots,x^N)\in \mathbf R^{dN}, \, x^i\neq  x^j \text{ if } i\neq j  \}$.  Since what matters in this section is to treat the interaction potential part of the  potential~$V$ of the system, we will assume  that the confining potential is the simple quadratic function (see also Remark \ref{eq.RRgamma>0-i} below for extensions).
\medskip

\noindent
\textbf{Assumption {\textbf{[V-{\tiny 2}]}}}. 
 \textit{The confining potential $V_{\mathbf c}$ is the quadratic function on $ \mathbf R^d$, i.e. $V_{\mathbf c}(y)=a_0|y|^2/2, \ y\in \mathbf R^d, a_0>0$}. 
\medskip

\noindent
\textbf{Assumption {\textbf{[V-{\tiny int}]}}}. 
\textit{The interaction potential  $V_{\mathbf I} :\mathbf R^d  \to \overline{\mathbf R}$    satisfies $V_{\mathbf I}(0)=+\infty$. In addition, there exist $\mathsf B,\beta>0$ and  a  symmetric function $\Phi\in \mathsf L_{{\rm Lip}}^{1, {\rm loc}}(\mathbf R^d\setminus\{0\})$ such that 
$$V_{\mathbf I}(y)= \frac{\mathsf B}{|y|^\beta}+\Phi(y), \text{   for all~$y\in \mathbf R^d\setminus\{0\}$}.$$
 Furthermore, there exist $r_\Phi, C_\Phi, c_\Phi>0$,  and $0\le q_\Phi<\beta+1$,   such that both $\Phi$ and  $\nabla \Phi$  are bounded on  $\{y\in \mathbf R^d, \, |y|>r_\Phi\}$,     $|\nabla \Phi(y)| \le \frac{C_\Phi}{|y|^{q_\Phi}} + c_\Phi \, \text{ and }\,    \lim_{  |y|\to 0}  |y|^\beta   |\Phi(y)|   =0$.}
\medskip

\noindent
Note that $V_{\mathbf I}$ is a symmetric function. Assumption  \textbf{[V-{\tiny int}]} covers  the   cases of singular potentials   used in molecular dynamics, namely: the  Lennard-Jones potentials  $V_{\mathbf I}(y)= c_1/|y|^{12}-c_2/|y|^6$ as well as Coulomb  potentials    $V_{\mathbf I}(y)=c_3/|x|^{d-2}$  when $d\ge 3$ ($c_1,c_2,c_3>0$). Note that in many applications in molecular dynamics,  $d=3$.  
\medskip

\noindent
\textbf{Assumption {[V{\tiny sing1}]}}. 
\textit{$d\ge 2$ and the potential function $V:  \mathbf R^{dN} \to \overline{\mathbf R}$  writes 
\begin{equation}\label{eq.V-i}
V(x)=\sum_{i=1}^N V_{\mathbf c}(x^i)+ \sum_{i,j=1;  i<j}^N V_{\mathbf I}(x^i-x^j) + V_{\mathbf p}(x),
\end{equation}
 where  
 $V_{\mathbf p}\in \mathsf L_{{\rm Lip}}^{1, {\rm loc}}(\mathbf R^{dN})$, $ {\rm supp}(V_{\mathbf p})\subset  \{ x^i\neq  x^j    \}$, $V_{\mathbf I}$ satisfies {\rm \textbf{[V-{\tiny int}]}},  and $V_{c}$ satisfies {\rm \textbf{[V-{\tiny 2}]}}.}
\medskip
  
  \noindent
   When \textbf{[V{\tiny sing1}]} is satisfied, it holds:
 \begin{enumerate}
 \item  $\{V<+\infty\}= \{ x^i\neq  x^j  \}$  which is open and connected (since $d\ge 2$). Hence 
$$\mathbf O_V=  \{ x^i\neq  x^j  \}.$$
\item $V\in   \mathsf L_{{\rm Lip}}^{1, {\rm loc}}(\mathbf O_V)$.     
  \item    $V$  is coercive on $\mathbf O_V$, i.e.  $V(x)\to +\infty$ if and only if $|x|\to +\infty $ or for some $i\neq j$, $|x^i-x^j|\to 0$, that we write $x\to    \{\infty\}   \cup \partial   \mathbf O_V$.  Therefore, in this case: 
 $$\mathsf H_{\text{GL}}(\mathsf x)\to +\infty \text{ if and only if }  \mathsf x\to \{\infty\} \text{ or } x\to \partial \mathbf O_V.$$
 \item $|\nabla V|$  is coercive on $\mathbf O_V$ (by~\cite[Lemma A.1]{herzog}). 
 \end{enumerate}
Note that   \textbf{[V{\tiny coercive}]} holds when  \textbf{[V{\tiny sing1}]} is satisfied.   We have assumed that $d\ge 2$ in \textbf{[V{\tiny sing1}]} for convenience: all the results  stated under the Assumption \textbf{[V{\tiny sing1}]} extend trivially to the case when $d=1$ provided to choose  in this case $\mathbf O_V$ as a connected component of $\{V<+\infty\}$.


   \begin{prop}\label{pr.ex-GL-i}
   Assume that $V$ satisfies {\rm\textbf{[V{\tiny coercive}]}}.  Then, for all $\mathsf x_0\in \mathscr E$, there exists a  unique strong solution  $(X_t=(x_t,v_t,z_t),t\ge 0)$ to \eqref{eq.GLintro} such that $X_0=\mathsf x_0$ and such that a.s. $X_t\in \mathscr E=\mathbf O_V\times \mathbf R^{dN} \times\mathbf R^{dN}$ for all $t\ge 0$. In addition, $(X_t(\mathsf x_0),t\ge 0)$ is  a strong Markov process.
   \end{prop}

 \begin{proof}
Proposition \ref{pr.ex-GL-i}  follows from the fact that $\mathcal L_{\text{GL}}\mathsf H_{\text{GL}}\le c \mathsf H_{\text{GL}}$ over $\mathscr E$ for some $c>0$. We also mention that this inequality implies the key inequality $
\mathbf P_{\mathsf x}[\sigma_{\mathscr H_R}\le t]\le \frac{e^{ct} }{R} \,  \mathsf H_{{\rm GL}}(\mathsf x)$ valid for all $\mathsf x\in \mathscr H_R$ and $t\ge 0$ ($\sigma_{\mathscr H_R}= \inf \{t\ge 0, X_t\notin \mathscr H_R\}$ is the first exit time from the open and bounded set 
  $\mathscr H_R:=\{\mathsf x\in  \mathbf R^{3dN}, \mathsf H_{{\rm GL}}(\mathsf x)<R\}\subset \mathscr E$). 
 \end{proof}

   
 The main result of this section is the following extension of Theorem~\ref{th.1} to the case of   singular potentials~$V$. 
 \begin{thm}\label{th.2}
Assume {\rm\textbf{[V{\tiny sing1}]}}.  Let $\mathscr O$  be a subdomain of $\mathbf O_V$   such that $ \mathbf O_V\setminus \overline{\mathscr O}$ is nonempty, and   set $\mathscr D=\mathscr O\times \mathbf R^d\times \mathbf R^d$.
Let also   $\delta \in (0,1]$. Then, for all $p\in (1,+\infty)$, Items i)$\to$iii) in Theorem~\ref{th.1} are  valid verbatim  for the killed semigroup $( P_t^{\mathscr D} , t\ge 0)$ of the process $(X_t,t\ge 0)$ solution to~\eqref{eq.GLintro} over~$\mathscr E=\mathbf O_V\times \mathbf R^{dN} \times\mathbf R^{dN}$ (see Proposition~\ref{pr.ex-GL-i})   with the Lyapunov function $\mathsf W_\delta :\mathscr E\to [1,+\infty)$ defined in \eqref{eq.lyapunovint} (and which satisfies the upper bound   \eqref{eq.Up-GL-i1}). 
  \end{thm} 
  The proof of Theorem~\ref{th.2} is made in Section \ref{sec.GL-singular}.

   \begin{remarks}\label{eq.RRgamma>0-i}
Let us mention that when $\gamma>0$ (resp. $\gamma=0$), with minor changes in the proof of \textbf{(C3)} in Section \ref{sec.GL-singular}, one deduces that  Theorem~\ref{th.2}  is still valid when $V_{\mathbf c}$ satisfies {\rm \textbf{[V{\tiny poly-$x^k$}]}}   (resp. when $V_{\mathbf c}$ satisfies {\rm \textbf{[V{\tiny poly-$x^k$}]}}   with $k\in (1,2]$), and for  $\delta \in ((1-\beta)/k,1]$, see \eqref{eq.Bbeta} and \eqref{eq.Bbeta2} for the definition of $\beta$.
 \end{remarks}
 %
 %
 %
  

 
 \subsection{Quasi-stationary distribution for the Nosé-Hoover process}
 
 \subsubsection{Main result}
 In this section, we consider a system of $N\ge 2$ particles with positions $x_t=(x^1_t,\ldots,x^N_t)\in \mathbf R^{dN}$ and velocities  $v_t=(v^1_t,\ldots,v^N_t)\in \mathbf R^{dN}$ evolving  according to the Nos\'e-Hoover process~\eqref{eq.NHintro} where $V$ satisfies  \textbf{[V{\tiny coercive}]}. 
  Recall that the state space is  
 \begin{equation}\label{eq.EE-NH}
 \mathscr E= \mathbf O_V\times \mathbf R^{dN}\times \mathbf R.
 \end{equation} 
The Hamiltonian function of the process \eqref{eq.NHintro} is denoted by 
\begin{equation}\label{eq.H-NH}
\mathsf H_{{\rm NH}}: (x,v,y)\in \mathscr E\mapsto  V(x)+\frac 12 |v|^2+ \frac 12 |y|^2,
\end{equation}
and its infinitesimal generator by
\begin{equation}\label{eq.L-NH}
\mathcal L_{{\rm NH}}= v\cdot \nabla _x  - ( y+\gamma)v\cdot \nabla_v-\nabla V\cdot  \nabla_v + \gamma   \Delta_v +  (|v|^2-d N  )\partial_y.
\end{equation}

  \begin{prop}\label{pr.existence-NHC2}
  Assume that $V$ satisfies {\rm\textbf{[V{\tiny coercive}]}}.
For all $\mathsf x_0\in \mathscr E$, there exists a  unique strong solution  $(X_t=(x_t,v_t,y_t),t\ge 0)$ to \eqref{eq.NHintro} such that $X_0=\mathsf x_0$.  In addition, a.s. $X_t(\mathsf x_0)\in \mathscr E$ for all $t\ge 0$, and $(X_t(\mathsf x_0),t\ge 0)$ is  a strong Markov process.
  \end{prop}
 \begin{proof}
Since for some $c>0$, $\mathcal L_{\text{NH}}\mathsf H_{\text{NH}}\le c \mathsf H_{\text{NH}}$ over $\mathscr E$, the proof is the same as Proposition~\ref{pr.existence},     and thus omitted. Again, one has  $\mathbf P_{\mathsf x}[\sigma_{\mathscr H_R}\le t]\le \frac{e^{ct} }{R} \,  \mathsf H_{{\rm NH}}(\mathsf x)$   for all $\mathsf x\in \mathscr H_R$ and $t\ge 0$, where  $\sigma_{\mathscr H_R}= \inf \{t\ge 0, X_t\notin \mathscr H_R\}$ is the first exit time from the open and bounded subset 
  $\mathscr H_R:=\{\mathsf x \in \mathbf R^{2dN+1}, \mathsf H_{{\rm NH}}(\mathsf x)<R\}\subset \mathscr E$. 
 \end{proof}


As already explained in the introduction, to prove   existence and uniqueness of quasi-stationary distributions in weighted spaces (see Theorem~\ref{th.3} below), the starting point is Theorem~\ref{th.ext} below.  Assumptions  \textbf{(C1)}, \textbf{(C2)}, \textbf{(C4)}, and \textbf{(C5')}  (defined at the beginning of  Section \ref{sec.extC5}) will be  proved when  \textbf{[V{\tiny coercive}]} holds. To prove \textbf{(C3)},  we will rely on the work~\cite{herzogNH}. For that reason, we impose   similar   extra assumptions on the potential~$V$. 
 \medskip

\noindent
\textbf{Assumption {[V{\tiny sing2}]}}.
\textit{The condition } \textbf{[V{\tiny coercive}]}  \textit{holds. In addition: $V\in \mathcal C^2(\mathbf O_V)$, and  for some  $\zeta\in (1,2)$ and $\delta \in (1/2,1]$: } 
\begin{equation}\label{eq.Conds}
 \frac{|  \text{{\rm Hess}}V(x)|}{|\nabla V(x)|^\zeta}\to 0   \text{ and } \frac{|\nabla V(x)|^{2-\zeta }}{ V(x)^{1-\delta}}\to +\infty \text{ when $V(x)\to +\infty$}. 
\end{equation}
Note that if there exists such a $\zeta\in (1,2)$, then $\delta=1$ always satisfies the second condition in \eqref{eq.Conds}. As it will be clear,   the better $\delta$'s are those closed to $1/2$, since  much smaller Lyapunov functions can be constructed in this case.


 Assumption \textbf{[V{\tiny coercive}]} differs from Assumption  \textit{\textbf{normal}} in~\cite{herzogNH} since:   (i) $V$ is much less regular (ii)  we have an extra condition in \eqref{eq.Conds} (namely the one involving $\delta>0$), and (iii)   because  we do not assume that  $\int_{\mathscr E} e^{-\mathsf H_{{\rm NH}}(\mathsf x)}d\mathsf x<+\infty$ (since this will be automatically satisfied, see   Remark~\ref{eq.Reoo}). The last (extra) condition involving $\delta>0$ in \eqref{eq.Conds} will be needed in order to build a smaller Lyapunov function than in~\cite{herzogNH}. Indeed, with such a $\delta>0$, we will build Lyapunov function $\mathsf W$ such that  $\mathsf W\le e^{ c^\delta \mathsf H_{{\rm NH}}^\delta} \text{ on } \mathscr E$ for some $c>0$ independent of $\delta>0$.

 \begin{remarks}\label{eq.Reoo} 
It will be shown in this work   that   the nonkilled process $(P_t,t\ge 0)$ is strong Feller and  satisfies the topological transitivity \cite[Eq. (2.2)]{wu2001}.
In addition,   under  {\rm \textbf{[V{\tiny coercive}]}},  there exists a Lyapunov function $\mathsf W$ over $\mathscr E$ such that $\mathcal L_{{\rm NH}}\mathsf W\le -c+b\mathbf 1_K$ (see \cite{herzogNH}). Therefore, since the measure $\mu_{Gibbs}(d\mathsf x)=e^{-\mathsf H_{{\rm NH}}(\mathsf x)}d\mathsf x$ is invariant,  $\int_{\mathscr E} e^{-\mathsf H_{{\rm NH}}(\mathsf x)}d\mathsf x$ is finite.  
 \end{remarks}

   
 The main result of this section is the following and provides existence and  uniqueness of the quasi-stationary distribution of the Nos\'e-Hoover process~\eqref{eq.NHintro} in a domain $\mathscr D=\mathscr O\times \mathbf R^{dN}\times \mathbf R$, as well as the exponential convergence towards this quasi-stationary distribution. 
 
 \begin{thm}\label{th.3}
Assume  that $V$ satisfies {\rm \textbf{[V{\tiny sing2}]}}. Let $\delta \in (1/2,1]$ be as in the latter condition in \eqref{eq.Conds}.  Let $\mathscr D=\mathscr O\times \mathbf R^{dN}\times \mathbf R$ where $\mathscr O$ is a subdomain of $\mathbf O_V$ such that $ \mathbf O_V\setminus \overline{\mathscr O}$ is nonempty.   Then, for all $p\in (1,+\infty)$, Items i)$\to$iii) in Theorem~\ref{th.1} are  valid verbatim  for the killed semigroup $( P_t^{\mathscr D}, t\ge 0)$ of the process $(X_t,t\ge 0)$, the solution to~\eqref{eq.NHintro} over~$\mathscr E=\mathbf O_V\times \mathbf R^{dN} \times \mathbf R $ (see Proposition~\ref{pr.existence-NHC2}), with the Lyapunov function~$\mathsf W_\delta:\mathscr E\to [1,+\infty)$      defined in~\eqref{eq.Lyapunov-NH}, which satisfies the upper bound~\eqref{eq.Upp-NH}. 
  \end{thm}
 Theorem~\ref{th.3} is proved in  Section \ref{sec.NH-p}. 
  
 \subsubsection{On Assumption {\rm \textbf{[V{\tiny sing2}]}}}
In this section, we  provide some examples of singular  potential functions used in molecular dynamics which satisfy \textbf{[V{\tiny sing2}]}.   To this end, let us consider a potential $V$  of the form  \eqref{eq.V-i} where:
 \begin{enumerate}
  \item[\textbf -] The perturbation potential $V_{\mathbf p}\in \mathcal C^2_c(\mathbf O_V)$. 
 \item[\textbf -]  The confining potential $V_{\mathbf c}$ is $\mathcal C^2$ over $\mathbf R^d$ and  behaves   at infinity as $|x|^k$ with  $k>1$. 
  \item[\textbf -]   Either  $V_{\mathbf I}\equiv 0$ over $\mathbf R^{dN}$  (in this  case,  $\mathbf O_V=\mathbf R^{dN}$) or  $V_{\mathbf I}$ satisfies \textbf{[V{\tiny int}]} with $\Phi \in \mathcal C^2(\mathbf R^d\setminus \{0\})$ (in this case,  $\mathbf O_V$ is given by $\{ x^i\neq  x^j  \}$).
 \end{enumerate}
As already seen, Assumption \textbf{[V{\tiny coercive}]}  is satisfied.  Since in addition $ (k-2)/(k-1)<1$ and  $1< (\beta+2)/(\beta+1)<2$,  the first condition in \eqref{eq.Conds} is also always satisfied for any $\zeta \in (1+ \frac{1}{\beta +1},2)$ (using also~\cite[Lemma A.1]{herzog}).  The last condition in \eqref{eq.Conds} writes 
$$\delta>\max\Big [1/2,1-\frac{k-1}k(2-\zeta), 1-\frac{\beta+1}\beta(2-\zeta)\Big ].$$ Note that $2-\zeta\in (0,  \frac{\beta}{\beta +1})$. The best choice of $\zeta$ is to take it small enough  such that     $2-\zeta$ is close to $\frac{\beta}{\beta +1}$. In conclusion,    \textbf{[V{\tiny sing2}]} holds in this case  if 
$$\max\Big [1/2,1-\frac{k-1}k\frac{\beta}{\beta +1}\Big ] <\delta\le 1.$$
If $\beta \ge 1$,  by choosing the external potential $V_{\mathbf c}$ such that  $k\gg 1$, this previous condition simply becomes  empty, i.e. one can choose any $\delta \in (1/2,1]$.



 
  \section{Generalized Langevin process: proof of Theorem \ref{th.1}}
  \label{sec.2V}
   
  This   section is   dedicated  to the proof of Theorem~\ref{th.1} (see also Subsection \ref{sec.pr2} for the proof of Theorem~\ref{th.1-NE}). We recall that  we will rely  on~\cite[Theorem~2.2]{guillinqsd} and, the strategy thus consists  to show that      \textbf{(C1)}$\to$\textbf{(C5)} hold. 
  
  In all this section, we assume that \textbf{[V{\tiny loc}]} holds. 

\subsection{On the conditions  \textbf{(C1)},  \textbf{(C2)}, and \textbf{(C4)}}
 
Assumption \textbf{(C2)} is proved in Section~\ref{sec.C2-}, and  \textbf{(C1)} and \textbf{(C4)} are proved in Section~\ref{sec.CC}  (see Proposition~\ref{pr.SFgamma0}, Theorem~\ref{pr.SFgamma-neq-0}, and Proposition~\ref{co.SFD}). We first provide  in Section~\ref{sec.22}  
 explanations of the method we use in this work   to check the regularity conditions \textbf{(C1)},  \textbf{(C2)}, and \textbf{(C4)} for the process \eqref{eq.GLintro} as well as for the process \eqref{eq.NHintro}. 
  \subsubsection{The  energy splitting  approach  to  prove the regularity conditions}
  \label{sec.22}
 
In this section, we explain the  strategy developed in this work  to prove      \textbf{(C1)},  \textbf{(C2)},  and  \textbf{(C4)} in a  weak regularity setting. 
  
    \medskip
  
  \noindent
  \textbf{On condition \textbf{(C1)}.}   We first observe that if  $V\in \mathcal C^\infty(\mathbf R)$,  using the H\"ormander's theory,   it is straightforward to deduce  \textbf{(C1)}.   When $\nabla V$ is only locally Lipschitz, one must argue differently. 
When $\gamma>0$, thanks to  the global Girsanov formula \eqref{eq.Girsanov},  we can make the non smooth term $\nabla V$ appear only in the    so-called Doléans-Dade exponential. With this formula,  the proof of \textbf{(C1)}  relies  on  tools from the theory of SDEs. When 
 $\gamma=0$, we cannot argue as previously. This is because  there is no Brownian motion anymore in the velocity variable $v_t$ in \eqref{eq.GLintro}. As explained in the introduction above, we will thus adopt a different approach to prove \textbf{(C1)} (and also \textbf{(C4)}).
 
 This approach is  based on a rather natural strategy which  can be summarized as  the study of the behavior of the process  at low and at high energy $\mathsf H_{\text{GL}}$ with respect to  initial conditions in the compact sets. To explain this approach,  the  starting point consists of  splitting the measure $\mathbf P_{\mathsf x}[X_t \in \cdot]$ as follows:
\begin{align} 
\label{eq.splitting}
\mathbf P_{\mathsf x}[X_t \in \cdot]&=\mathbf P_{\mathsf x}[X_t \in \cdot, \sigma_{\mathscr H_R}\le t ]+\mathbf P_{\mathsf x}[X_t \in \cdot, t<\sigma_{\mathscr H_R}  ] =:    \rho^R_{\mathsf x}(\cdot) + \theta^R_{\mathsf x}(\cdot).
\end{align}
 The asymptotic behavior  of  the function $ R  \mapsto   \rho^R_{\mathsf x}(\cdot)$ as $R\to +\infty$ uniformly w.r.t the initial conditions in the compact sets provides the behavior of the process w.r.t  high energies $\mathsf H_{\text{GL}}$, while the  continuity property of  $ \mathsf x\mapsto   \theta^R_{\mathsf x}(\cdot)$ at fixed $R>0$ describes  the  behavior of the process  at low energy w.r.t the initial conditions. 
The approach then relies on the observation that   \textbf{(C1)} holds if 
\begin{equation}\label{eq.rho}
\lim_{R\to +\infty} \rho^R_{\mathsf x}(\cdot)\to 0 \text{ uniformly with respect to  } \mathsf x \text{ in the compact sets},
\end{equation}
 and if for all $R>0$, 
 \begin{equation}\label{eq.theta}
  \mathsf x\mapsto   \theta^R_{\mathsf x}(\cdot) \text{  is continuous for the $\tau$-topology over } \mathscr H_R,
  \end{equation} 
  i.e. for the topology corresponding to convergence against   test functions $f\in b\mathcal B(\mathscr H_R)$. 
Equation \eqref{eq.rho} is in our setting a consequence of   \eqref{eq.tauR}, which we recall follows from the fact that $\mathcal L_{{\rm GL}} \mathsf H_{{\rm GL}} \le c \mathsf H_{{\rm GL}} $ (see also Remark~\ref{re.ReG}). The behavior of the process  at low energy (namely \eqref{eq.theta}) is  trickier to deduce. It will be obtained by combining  mainly two general     ingredients:     the  powerful  convergence    result Proposition~\ref{pr.wu1999} which, we recall, requires in particular uniform integrability of the transition probabilities, and    the uniform continuity   of the distribution functions near $0^+$ of exit times from $\mathbf R^{3d}$-balls  uniformly w.r.t the initial conditions   (see Lemma \ref{le.C4-2}).     The Markov property of the process \eqref{eq.GLintro} will be  also fully exploited.


\begin{remarks} (On condition \eqref{eq.theta})
\label{sec.reSF}
There are many ways to check   \eqref{eq.theta} and the chosen way    depends on the considered process. Let us provide  some of them. As far as solutions to stochastic differential equations driven by   Brownian noises are concerned,  the uniform integrability of the transition probability required in Proposition~\ref{pr.wu1999} can be obtained with   so-called Gaussian upper bound (as we did in the proof of Theorem \ref{pr.SFgamma-neq-0} below), see~\cite{baldi,delarue,menozzi2011parametrix} and the references therein. To circumvent situations where there is no Gaussian upper bounds, one can use other tools.  
In such cases,  \eqref{eq.theta} can   be for instance  obtained with the help of a  local Girsanov formula combined with techniques from stochastic calculus, as we did in the proof of Theorem \ref{th.NH-C} below for the   Nosé-Hoover process. 
Another way   to prove  \eqref{eq.theta} is to use derivative formulas of Bismut–Elworthy–Li’s type, which has been derived for several  kinds of processes (in finite and infinite dimensions) and whose literature is very rich,  see e.g.~\cite{peszat1995strong,takeuchi2010bismut,guillin2012degenerate,zhang2013derivative} and the references therein. This is the technique employed in~\cite[Proposition 2]{guillin2024large}    to prove  \eqref{eq.theta} (and more generally \textbf{(C1)} and \textbf{(C4)}) for solutions to stochastic differential equations driven by $\alpha$-stable processes or in~\cite{guillinFK} for Feynman-Kac semigroups associated with jump processes. In~\cite{guillin2024Dyson}, \eqref{eq.theta} (and actually \textbf{(C1)}$\to$\textbf{(C5)}) is proved for a system of particules with log-singular interactions subject to collisions\footnote{The drift is thus, in this model,  infinite for an infinitely number of times.} using~\cite[Corollary~2.2]{schilling2012strong}, which can be seen as  an alternative to Proposition~\ref{pr.wu1999}, and an underlying  sequence of processes approximating (in distribution) the system.   
\end{remarks}
  
  \noindent
  \textbf{On condition \textbf{(C4)}.} 
    The proof of the continuity of the mapping $\mathsf x\mapsto  \mathbf P_{\mathsf x}[X_t \in \cdot, t<\sigma_{\mathscr D}]$ (i.e. the proof of \textbf{(C4)}), where $\mathscr D= \mathscr O\times \mathbf R^m$,  requires extra analysis  since it involves studying  the continuity property of $\mathsf x\mapsto \sigma_{\mathscr D}(\mathsf x)$. 
The techniques developed for the kinetic Langevin  process~\eqref{eq.K}  in \cite{guillinqsd,guillinqsd2} (see also \cite{ramilarxiv1,ramilarxiv2})  to prove  \textbf{(C4)}  require    $\partial \mathscr O \cap \mathbf O_V$ to be~$\mathcal C^2$. These techniques cannot longer be used without regularity assumption on  $\partial \mathscr O$. To prove  \textbf{(C4)},  we will  use  the  energy splitting approach (so-called because based on the energy splitting Equation  \eqref{eq.splitting})  described previously. However, we mention that to check \textbf{(C4)}, we   have to  study    the  distribution functions near $0^+$ of exit times from $\mathbf R^d$-balls   uniformly w.r.t the initial conditions, which is the purpose of  Lemma~\ref{le.C4}.   
 \medskip

 We mention that  the energy splitting approach  is also used at many other places in this work, and in particular  it is the starting point  for the proofs of
 \begin{itemize}
 \item[\textbf -] Lemma~\ref{le.C4} as well as in the proof of \textbf{(C2)}   (see  the proof of  Proposition~\ref{pr.C2}).
 \item[\textbf -] Conditions \textbf{(C1)}, \textbf{(C2)}, and \textbf{(C4)} for both the generalized Langevin process~\eqref{eq.GLintro} and the Nosé-Hoover process~\eqref{eq.NHintro}    when $V$ is a singular  potential  (see the proofs of Proposition \ref{pr.PP-S},  Theorem~\ref{th.NH-C}, and Proposition~\ref{pr.Vs3-t}).
  \end{itemize}

 \begin{remarks}\label{re.ReG}
 Finally, we  mention that the energy splitting approach can be also  applied to solutions of other solutions to SDEs as soon  as  there is a Lyapunov function  $\mathsf V$ such that $\mathcal L \mathsf V\le c\mathsf V$ ($c>0$, and $\mathcal L$ being the extended generator of the process). What replaces $\mathsf H$ in this case is  the energy~$\mathsf V$.  
  \end{remarks}

  \subsubsection{On Assumption  \textbf{(C2)}} 
\label{sec.C2-}

 In this section, we check Assumption \textbf{(C2)}. 
 
  \begin{prop}\label{pr.C2}  Assume {\rm\textbf{[V{\tiny loc}]}}. 
 Let  $t\ge 0$ and $(\mathsf x_n)_n$ in $\mathbf R^{3d}$ such that $\mathsf x_n\to \mathsf x\in \mathbf R^{3d}$ as $n\to +\infty$. Then, for all $\epsilon>0$, 
  $\mathbf P  [\sup_{s\in [0,t]}| X_s(\mathsf x_n)-X_s(\mathsf x)|\ge \epsilon   ]\to 0$ as $n\to +\infty$. 
  In particular, Assumption \textbf{(C2)} is satisfied.
   \end{prop}

 \begin{proof} 
 We denote by $\mathbf a$ the drift of the equation \eqref{eq.GLintro}, namely for  $\mathsf x=(x,v,z)\in \mathbf R^{3d}$, 
       $$
       \mathbf a(\mathsf x)= \begin{pmatrix}
v  \\
-\nabla V(x)  -\gamma v +\lambda z\\
 -\alpha z -\lambda v
\end{pmatrix}.$$
Note that  $\mathbf a$ is a locally Lipschitz vector field over $\mathbf R^{3d}$. 
 Let  $R_0>0$ such that  $\mathsf x_n,\mathsf x\in  \mathscr H_R=\{\mathsf x, \mathsf H_{{\rm GL}}(\mathsf x)<R\}$ for all $R\ge R_0$ and $n\ge 0$. In the following we assume that $R\ge  R_0$. By Gronwall's inequality, we have for all $R\ge R_0$, when $t<\sigma_{\mathscr H_R}(\mathsf x)\wedge \sigma_{\mathscr H_R}(\mathsf  x_n)$, 
$$\sup_{s\in [0,t]}| X_s(\mathsf x)-X_s(\mathsf x_n)|\le |\mathsf x-\mathsf x_n|e^{\mathbf a_R t},$$
for some  $\mathbf a_R>0$. Thus,  one has for all $\epsilon>0$ and $R\ge R_0$, as $n\to \infty$, 
$$\mathbf p_n(R):=\mathbf P\big [\sup_{s\in [0,t]}| X_s(\mathsf x)-X_s(\mathsf x_n)|\ge \epsilon, t< \sigma_{\mathscr H_R}(\mathsf x)\wedge \sigma_{\mathscr H_R}(\mathsf  x_n)\big]\to 0.$$
Consequently, it holds:
\begin{align*}
&\mathbf P\big [\sup_{s\in [0,t]}| X_s(\mathsf x)-X_s(\mathsf x_n)| \ge \epsilon\big]\\
&=\mathbf p_n(R)  + \mathbf P\big [\sup_{s\in [0,t]}| X_s(\mathsf x)-X_s(\mathsf x_n)|\ge  \epsilon, t\ge  \sigma_{\mathscr H_R}(\mathsf x)\wedge \sigma_{\mathscr H_R}(\mathsf x_n)\big]\\
&\le \mathbf p_n(R) + \mathbf P\big [t\ge  \sigma_{\mathscr H_R}(\mathsf x)\wedge \sigma_{\mathscr H_R}(\mathsf x_n)]\\
&\le \mathbf p_n(R) + \mathbf P\big [t\ge  \sigma_{\mathscr H_R}(\mathsf x) ]+ \mathbf P\big [t\ge   \sigma_{\mathscr H_R}(\mathsf x_n)]\\
&\le \mathbf p_n(R)+ 2R_0\frac{e^{ct} }{R}, 
\end{align*}
where we have used \eqref{eq.tauR} to get the last inequality together with  the fact that $\mathsf H_{{\rm GL}}(\mathsf x_n)< R_0$ and $\mathsf H_{{\rm GL}}(\mathsf x)< R_0$. Let us now consider $\delta>0$ and $\epsilon>0$. Pick $R_\delta>0$ such that $2R_0\frac{e^{ct} }{R_\delta}\le \delta/2$. For this  fixed $R_\delta$, $\mathbf p_n(R_\delta)\to 0$ as $n\to +\infty$, and thus, there exists $N_\delta\ge 1 $ such that  for all $n\ge N_\delta$, $\mathbf p_n(R_\delta)\le\delta/2$. The proof of the proposition is complete.
  \end{proof}

   \subsubsection{On Assumptions \textbf{(C1)} and \textbf{(C4)}}
 \label{sec.CC}

 In this section,  we prove that   \textbf{(C1)} and \textbf{(C4)} are satisfied for the process \eqref{eq.GLintro} when \textbf{[V{\tiny loc}]} holds. 
 \medskip

 \noindent
\textbf{On Assumption~\textbf{(C1)}}. 
  In this section, we prove~\textbf{(C1)}.   
  \medskip
  
  \noindent
  \underline{Assumption \textbf{(C1)} when $\gamma>0$}.   
 
 \begin{prop}\label{pr.SFgamma0}
Assume  {\rm\textbf{[V{\tiny loc}]}} and  $\gamma>0$. Then, the nonkilled semigroup $P_t$ of the solution to the generalized Langevin equation \eqref{eq.GLintro}  is strong Feller for every $t>0$. 
\end{prop}

The proof of Proposition \ref{pr.SFgamma0} we give below  actually only requires that $V$ is coercive and  differentiable, and that  $\nabla V$ is   just continuous (and in this case, one rather considers the weak solution to \eqref{eq.GLintro}). 


\begin{proof}

Let   $(X_t^0=(x_t^0,v_t^0,z^0_t),t\ge 0)$ in $\mathbf R^{3d}$ be the strong solution to the stochastic differential equation
 \begin{equation}\label{eq.GL0}
       dx_t^0= v_t^0dt, \
     dv_t^0=\sqrt{2\gamma}  \, dW_t, \
       dz_t^0= \sqrt{2\alpha}\, dB_t.
 \end{equation}  
We denote by $\mathbf F$ the following vector field over $\mathbf R^{3d}$ and by $\boldsymbol \Sigma$ the following square constant matrix of size $2d$:
       $$
       \mathbf F(\mathsf x)= \begin{pmatrix} 
-\nabla V(x)  -\gamma v +\lambda z\\
 -\alpha z -\lambda v
\end{pmatrix} \text{ and } \boldsymbol \Sigma= \begin{pmatrix}
 \sqrt{2\gamma}\, I_{\mathbf R^d}&0\\
0&\sqrt{2\alpha}\, I_{\mathbf R^d}
\end{pmatrix}.$$
Arguing as in~\cite[Lemma 1.1]{wu2001}, the following  Girsanov's formula holds for all $T>0$,  
\begin{equation}\label{eq.Girsanov}
\frac{d \mathbf P_{\mathsf x}}{d\mathbf P_{\mathsf x}^0}{  \Big|_{\mathcal F_T}}= \mathscr  M_T(\mathsf x)  ,
\end{equation}
where $ \mathbf P_{\mathsf x}$ (resp. $ \mathbf P^0_{\mathsf x}$)  is the law of  $(X_t(\mathsf x),t\ge 0)$ (resp. of $(X_t^0(\mathsf x),t\ge 0)$), and for $t\ge 0$, 
  \begin{align*}
     \mathscr  M_t(\mathsf x)&=\exp\Big[    \int_0^t \!  \boldsymbol \Sigma^{-1}\mathbf F(X_s^0(\mathsf x) ) \cdot   d\boldsymbol w_s-\frac 12 \int_0^t |\boldsymbol \Sigma^{-1}\mathbf F(X_s^0(\mathsf x)) |^2 ds   \Big],
 \end{align*}
 is the   Doléans-Dade exponential (true) martingale, and $\boldsymbol w_s=  (B_s,W_s)^T$. 
 
  Fix $t>0$. By \eqref{eq.Girsanov},   $P_t$ is strong Feller if  $\mathsf z\in \mathbf R^{3d}\mapsto \mathbf E_{\mathsf z}[f(X_t^0)\mathscr  M_t]$ is continuous, which we shall prove now. Note that  a.s.    $ \sup_{s\in [0,t]}|X_s^0(\mathsf z)-X_s^0(\mathsf x)|\to 0$ as $\mathsf z\to \mathsf x\in \mathbf R^{3d}$. Note also that  $(X_t^0,t\ge 0)$ has a smooth density with respect to  the Lebesgue measure.  
Then, using Proposition~\ref{pr.wu1999}, we deduce that for any $f\in b\mathcal B(\mathbf R^{3d})$, $\mathsf z\in \mathbf R^{3d}\mapsto  f(X_t^0(\mathsf z))$ is continuous in $\mathbf P$-probability. 

\begin{sloppypar}
 On the other hand, $\mathsf z\in \mathbf R^{3d}\mapsto\mathscr  M_t(\mathsf z)$ is continuous in $\mathbf P$-probability.  Since $\mathbf F$ is unbounded and not globally Lipschitz, we provide   the proof of this claim. 
 Set for   $R\ge 1$,   
 $$\sigma_{\mathsf B(0,R)}^0(\mathsf z):=\inf\{s\ge 0, X_s^0(\mathsf z)\notin \mathsf  B(0,R)\},$$
  where $\mathsf  B(0,R):=\{\mathsf y\in \mathbf R^{3d}, |\mathsf y|< R\}$. 
If $|\mathsf z|\le \sqrt R$, $\mathbf P_{\mathsf z}[\sigma_{\mathsf B(0,R)}^0\le t]\le \mathbf h(R)$, where 
$$\mathbf h(R):=\mathbf  P\Big [\sup_{s\in [0,t]}|\boldsymbol  \Sigma \boldsymbol  w_s|\ge R/2-\sqrt R\Big ]+ \mathbf  P\Big [\sqrt{2\gamma}\int_0^t |W_s|ds\ge R/2-t\sqrt R   -\sqrt R\Big ].$$
We have $\mathbf h(R) \to 0$ as $R\to +\infty$. Let $\mathsf z_n\to \mathsf z\in \mathbf R^{3d}$, and $R_{\mathsf z}\ge 1$ such that  $|\mathsf z_n|+|\mathsf z|\le \sqrt{R}_{\mathsf z} $, $\forall n\ge 0$. Pick    $c>0$. Let $\epsilon>0$ and $R_\epsilon>R_{\mathsf z}$ such that $ \mathbf h(R_\epsilon)\le \epsilon/4$. Consider also a continuous bounded vector field $\mathbf F_{R_\epsilon}:\mathbf R^{3d}\to \mathbf R^{2d}$   such that $\mathbf F_{R_\epsilon}=\mathbf F $ on $ \mathsf  B(0,R_\epsilon+1)$.  Set $\boldsymbol \delta_n(\mathbf F):=|\int_0^t   \boldsymbol \Sigma^{-1}[\mathbf F(q_s^0(\mathsf z_n) )- \mathbf F(q_s^0(\mathsf z) )]\cdot   d\boldsymbol w_s|$.  Then,  for all $n\ge 0$, the quantity $ \mathbf P  [\boldsymbol \delta_n(\mathbf F)  >c ]$ is bounded by 
\begin{align*}
 &   \mathbf P[\sigma_{\mathsf B(0,R_\epsilon)}^0(\mathsf z_n)\wedge \sigma_{\mathsf B(0,R_\epsilon)}^0(\mathsf z)\le t] + \mathbf P [\boldsymbol \delta_n(\mathbf  F) >c, \ t<\sigma_{\mathsf B(0,R_\epsilon)}^0(\mathsf z)  \wedge \sigma_{\mathsf B(0,R_\epsilon)}^0(\mathsf z_n) ]\\
&\le    \mathbf P[\sigma_{\mathsf B(0,R_\epsilon)}^0(\mathsf z_n) \le t]+\mathbf P[\sigma_{\mathsf B(0,R_\epsilon)}(\mathsf z) \le t] + \mathbf P [\boldsymbol \delta_n(   \mathbf F_{R_\epsilon})> c, \ t<\sigma_{\mathsf B(0,R_\epsilon)}^0(\mathsf z)  \wedge \sigma_{\mathsf B(0,R_\epsilon)}^0(\mathsf z_n)   ]\\
&\le \frac{\epsilon}2  + \mathbf P  [\boldsymbol \delta_n(   \mathbf F_{R_\epsilon})> c  ].
\end{align*}
\end{sloppypar}
Since $\mathbf P  [\boldsymbol \delta_n( \mathbf   F_{R_\epsilon})> c  ]\to 0$ as $n\to +\infty$ by well-known convergence properties of stochastic integrals,
 we finally get that there exists $N_\epsilon\ge 0$, for all $n\ge N_\epsilon$,  $\mathbf P  [\boldsymbol \delta_n(   \mathbf F)>c  ]\le \epsilon$.  With similar arguments, we also prove that the mapping $\mathsf x\in \mathbf R^{3d}\mapsto  \int_0^t | \mathbf F(q_s^0(\mathsf x) )) |^2ds$ is continuous in $\mathbf P$-probability. The claim then follows applying  the continuous mapping theorem.

Since in addition,  $\mathbf E[\mathscr  M_t(\mathsf x)]=1$ for all $\mathsf x$,   
we have that $\mathsf x\mapsto \mathbf E[\mathscr  M_t(\mathsf x)]$ is continuous in $L^1(\mathbf P)$  (see the Vitali convergence theorem  \cite[Proposition 3.12]{kallenberg}) and hence,  $\mathsf z\in \mathbf R^{3d}\mapsto \mathbf E_{\mathsf z}[f(X_t^0)\mathscr  M_t]$ is continuous.  The proof of \textbf{(C1)} is complete when $\gamma>0$. 
 \end{proof}

  

  \noindent
  \underline{Assumption \textbf{(C1)} when $\gamma=0$}.   
   
For $\mathsf x=(x,v,z)\in \mathbf R^{3d}$, we denote by 
\begin{equation}\label{eq.BBoo}
\sigma_{\mathscr B(x,\delta)}(\mathsf x)=\inf \{t\ge 0, x_t(\mathsf x)\notin \mathscr B(x,\delta)\},
\end{equation}
 the first exit time for the process $(x_t(\mathsf x),t\ge 0)$ from the open ball $\mathscr B(x,\delta)$ in $\mathbf R^d$ centered at $x\in \mathbf R^d$ and of radius $\delta>0$.   When there is no confusion, we simply write it  $\sigma_{\mathscr B(x,\delta)}$.

 We start with  two crucial lemmas  which hold for any $\gamma\ge 0$ and which will be fundamental to prove \textbf{(C1)} when $\gamma=0$ and   \textbf{(C4)} when $\gamma\ge0$ . 
 
 \begin{lem}\label{le.C4}
 Assume {\rm\textbf{[V{\tiny loc}]}}  and $\gamma \ge0$.  It holds for all compact subset $K$ of $\mathbf R^{3d}$ and $\delta>0$, 
$$ \lim_{s\to 0^+} \sup_{\mathsf x=(x,v,z)\in K}\mathbf P_{\mathsf x}[\sigma_{\mathscr B(x,\delta)}\le s]=0.$$
\end{lem}
We will perform the proof when $\gamma=0$. It is straightforward to adapt the proof to the case   $\gamma>0$. 


 \begin{proof}
Assume $\gamma=0$.  Fix $\delta>0$ and  let $R_K>0$  such that $K \subset \mathscr H_{R_K}$. Let $\mathsf x=(x,v,z)\in K$ and $R\ge R_K$. 
We have using  \eqref{eq.tauR}, for all $s\in [0,1]$,
 \begin{align*}
 \mathbf P_{\mathsf x}\big [\sigma_{\mathscr B(x,\delta)}\le s\big ]
 &\le  \mathbf P_{\mathsf x}\big [\sigma_{\mathscr B(x,\delta)}\le s, s<\sigma_{\mathscr H_R}\big ]+ \frac{e^{c} }{R} \sup_{\mathsf  y\in K} \mathsf H_{{\rm GL}}(\mathsf y).
\end{align*}
In addition, it holds
 \begin{align*}
   \mathbf P_{\mathsf x}[\sigma_{\mathscr B(x,\delta)}\le s, s<\sigma_{\mathscr H_R}]&\le \mathbf P_{\mathsf x}\big [\sigma_{\mathscr B(x,\delta)}\le s, |x_s-x|< \delta/2, s<\sigma_{\mathscr H_R}\big ]\\
   &\quad +   \mathbf P_{\mathsf x}\big [  |x_s-x|\ge \delta/2, s<\sigma_{\mathscr H_R}\big].
   \end{align*}
   Note that when $X_0=\mathsf x$, $\sigma_{\mathscr B(x,\delta)}\le s$, and  $|x_s-x|< \delta/2$, it holds  $|x_s-x_{ \sigma_{\mathscr B(x,\delta)}} |\ge \delta/2$. Then,  by the strong Markov property: 
    \begin{align*}
     &\mathbf P_{\mathsf x}\big [\sigma_{\mathscr B(x,\delta)}\le s, |x_s-x|< \delta/2, s<\sigma_{\mathscr H_R}\big ] \\ 
     &\le  \mathbf P_{\mathsf x}\big [\sigma_{\mathscr B(x,\delta)}\le s, |x_{s}-x_{\sigma_{\mathscr B(x,\delta)}}|\ge  \delta/2, s<\sigma_{\mathscr H_R}\big ]\\
     &= \mathbf P_{\mathsf x}\big [\{\sigma_{\mathscr B(x,\delta)}\le s\}\cap \{ |x_{s}-x_{\sigma_{\mathscr B(x,\delta)}}|\ge  \delta/2\}\cap \{ s<\sigma_{\mathscr H_R}\}\cap \{X_u\in \mathscr H_R, \forall u\in [0,\sigma_{\mathscr B(x,\delta) }]\}\big ] \\
     &= \mathbf E_{\mathsf x}\Big [\mathbf1_{\sigma_{\mathscr B(x,\delta) }(\mathsf x) \le s} \mathbf 1_{X_u\in \mathscr H_R, \forall u\in [0,\sigma_{\mathscr B(x,\delta) }]} \\
     &\quad \times    \mathbf P_{X_{\sigma_{\mathscr B(x,\delta) }(\mathsf x) } (\mathsf x)}\big [|x_{s-\sigma_{\mathscr B(x,\delta)}(\mathsf x) }-x_0|\ge \delta/2,  s-\sigma_{\mathscr B(x,\delta)}(\mathsf x) <\sigma_{\mathscr H_R} \big ]  \Big ] := \mathbf p(\mathsf x, s,  R).
   \end{align*}

   
 Let us now consider $\epsilon>0$. We choose  $R_{\epsilon,K}= 3e^{c}\sup_{\mathsf y\in K} \mathsf H_{{\rm GL}}(\mathsf y)/ \epsilon+R_K$. Then, using~\eqref{eq.tauR}, one has for all $\mathsf x\in K$ and $s\in [0,1]$,
\begin{align}
 \label{eq.inter}
 \mathbf P_{\mathsf x}\big [\sigma_{\mathscr B(x,\delta)}\le s\big ]  \le  \mathbf p(\mathsf x, s,  R_{\epsilon,K}) +  \mathbf P_{\mathsf x}\big [  |x_s-x|\ge \delta/2, s<\sigma_{\mathscr H_{R_{\epsilon,K}}}\big] + \frac{\epsilon}{3}.
\end{align}
Let us now study the term $\mathbf p(\mathsf x, s,  R_{\epsilon,K})$. To this end, let $s\in [0,1]$ and let us consider an initial condition $\mathsf z=(x_{\mathsf z},  v_{\mathsf z}, z_{\mathsf z})$ such that $x_{\mathsf z}\in \partial \mathscr B(x,\delta)$, $\mathsf z\in \mathscr H_{R_{\epsilon,K}}$, and   $u\in [0,s]$. We now study the term $\mathbf P_{\mathsf z}[|x_{s-u}-x_{\mathsf z}|\ge \delta/2, s-u<\sigma_{\mathscr H_{R_{\epsilon,K}}}]$. 

Note that    $x_{s- u}=x_{\mathsf z}+\int_0^{s-u} v_t\, dt$, and thus, 
 \begin{align*}
 x_{s- u}-x_{\mathsf z}&=   \int_0^{s- u} \Big [v_{\mathsf z}-\int_0^t  \nabla V(x_a)  da \Big] dt \\
 &\quad + \int_0^{s- u}  \Big [\int_0^t \lambda \big ( z_{\mathsf z} + \int_0^a  [-\alpha z_h-\lambda v_h] dh  + \sqrt{2\alpha}B_a\big  ) da \Big] dt.
 \end{align*}
  Assume now that 
  $s- u<\sigma_{\mathscr H_{R_{\epsilon,K}}}(\mathsf z)$, 
    i.e. 
     \begin{equation}\label{eq.Y02}
 \text{$X_h(\mathsf z)\in \mathscr H_{R_{\epsilon,K}}$ for all $h\in [0,s- u]$.}
 \end{equation}  
  Let $c_{\epsilon,K}>0$ be such that for all  $\mathsf x=(x,v,z)\in \mathscr H_{R_{\epsilon,K}}$,     $
  | \nabla V(x)|+|\mathsf x| \le c_{\epsilon,K}$.
We then have $|\mathsf z|\le c_{\epsilon,K}$  and, from \eqref{eq.Y02},
  $| \nabla V(x_h)|+ |X_h(\mathsf z)|\le c_{\epsilon,K}$ for all $h\in [0,s- u]$.  
In what follows $c_{\epsilon,K}>0$ is a constant  which only  depends on $\epsilon>0$ and  $K$,   which  can  change  from occurence to another. Hence,   when $X_0=\mathsf z$, $|x_{s- u}- x_{\mathsf z} |\ge \frac{\delta}2 $, and $s-u<\sigma_{\mathscr H_{R_{\epsilon,K}}}(\mathsf z)$, it holds since $0\le s-u\le s\le 1$:
   \begin{align*}
\frac \delta2 \le  |x_{s- u}-x_{\mathsf z}|&\le       \int_0^{s- u } \Big [ |v_{\mathsf z}|+ \int_0^t |\nabla V(x_a)| da \Big] dt \\
  &\quad + \int_0^{s- u}  \Big [\int_0^t \lambda \big ( |z_{\mathsf z}| + \int_0^a  [\alpha |z_h|+\lambda |v_h|] dh  + \sqrt{2\alpha} |B_a|\big  ) da \Big] dt\\
  &\le c_{\epsilon,K} s +  \sqrt{2\alpha}\int_0^s \int_0^t |B_a|da \, dt, 
 \end{align*}
 and consequently $ \sqrt{2\alpha}\int_0^{s} \int_0^t |B_a|da \, dt\ge  \frac \delta 2 -c_{\epsilon,K} s$. 
 Set $s_{\epsilon,K,\delta}=\delta/( 4c_{\epsilon,K})$. For $s\in [0,s_{\epsilon,K,\delta}]$,  one has by Markov's inequality, 
\begin{align*}
\mathbf P_{\mathsf z}\big [|x_{s- u}-x_{\mathsf z}|\ge \delta/2,  s-u<\sigma_{\mathscr H_{R_{\epsilon,K}}} \big ]&\le \mathbf P \Big [ \sqrt{2\alpha} \int_0^{s} \int_0^t |B_u|du\,  dt\ge \frac{\delta}4 \Big ]\le  \mathbf m_s,
\end{align*}
where $\mathbf m_s=2s^2 \sqrt{2\alpha}\, /\delta$. 
Hence,   
 we deduce that for any $s\in [0,s_{\epsilon,K,\delta}]$, 
 $$\mathbf p(\mathsf x, s, R_{\epsilon,K})\le  \mathbf m_s.$$
 With the same arguments, up to choosing $s_{\epsilon,K,\delta}>0$ smaller, for every $s\in [0,s_{\epsilon,K,\delta}]$, it holds 
 $$\sup_{\mathsf x\in K} \mathbf P_{\mathsf x}  [  |x_s-x|\ge \delta/2, s<\sigma_{\mathscr H_{R_{\epsilon,K}}} ] \le\mathbf m_s.$$ 
 Coming back to \eqref{eq.inter}, we have thus proved that  for all compact set $K\subset \mathbf R^{3d}$,  $\delta>0$, $\epsilon>0$, there exists $s_{\epsilon,K,\delta}>0$, for all $s\in [0,s_{\epsilon,K,\delta}]$,  and all  $\mathsf x\in K$, 
  $\mathbf P_{\mathsf x}\big [\sigma_{\mathscr B(x,\delta)}\le s\big ]\le \epsilon$. The proof of Lemma \ref{le.C4} is   complete. 
\end{proof}
  
 Let $\mathbf b:\mathbf R^d\to \mathbf R^d$ be a globally Lipschitz vector field.  Introduce the unique strong solution $(\bar X_t= (\bar x_t,\bar v_t, \bar z_t),t\ge 0)$ to \eqref{eq.GLintro} when $\nabla V$ replaced by $\mathbf b$, i.e. the solution to
\begin{equation}\label{eq.GLR-2}
  \left\{
    \begin{array}{ll}
        d\bar x _t&= \bar v_tdt \\
        d \bar v _t&= \mathbf b (\bar  x _t)dt-\gamma \bar v_tdt+\lambda \bar z _tdt+ \sqrt{2\gamma} \, dW_t \\
        d\bar  z _t&= -\alpha \bar z _tdt-\lambda \bar v _tdt +  \sqrt{2\alpha}\, dB_t.
    \end{array}
\right.
\end{equation}
  For $\mathsf x \in \mathbf R^{3d}$, we now denote by 
 $$\bar \sigma_{\mathsf B(\mathsf x,\delta)}(\mathsf x)=\inf \{t\ge 0, \bar X_t(\mathsf x)\notin \mathsf   B(\mathsf x,\delta)\},$$
 the first exit time for the process $(\bar X_t(\mathsf x),t\ge 0)$ from the   ball  $\mathsf  B(\mathsf x,\delta)=\{\mathsf y\in \mathbf R^{3d}, |\mathsf y-\mathsf x|< \delta\}$. 
  We have the following result which will be used in the proof \textbf{(C1)}.

  \begin{lem}\label{le.C4-2}
 Assume {\rm\textbf{[V{\tiny loc}]}} and $\gamma \ge0$.  Then, it holds for all compact subset $K$ of $\mathbf R^{3d}$ and $\delta>0$,  $$\lim_{s\to 0^+} \sup_{\mathsf x_0\in K}\mathbf P_{\mathsf x_0}[\bar \sigma_{\mathsf  B(\mathsf x_0,\delta)}\le s]=0.$$ 
\end{lem}
  \begin{proof}   Let $r>0$ be such that the $\delta$-neighborhood $K_\delta$ of $K$ is included in $\mathsf B(0,r)$.   
   Since the drift   in~\eqref{eq.GLR-2}   is globally Lipschitz,  there exists $c>0$, for all $s\in [0,1]$ and  $\mathsf x\in \mathsf B(0,r)$, $|\bar X_s(\mathsf x)|\le r +c(1+\int_0^s|\bar X_u(\mathsf x)|du) + c Y$, where  $Y:= \sup_{u\in [0,1]}(|B_u|+|W_u|)$. Hence, by Grönwall's inequality, $|\bar X_s(\mathsf x)|\le  C(1+ Y)$.  We have   
   $$\mathbf P_{\mathsf x_0}[\bar \sigma_{\mathsf B(\mathsf x_0,\delta)}\le s ]\le \mathbf P_{\mathsf x_0}\big [\bar \sigma_{\mathsf B(\mathsf x_0,\delta)}\le s, |\bar X_s-\bar X_0|< \delta/2\big ]  +   \mathbf P_{\mathsf z}\big [  |\bar X_s-\bar X_0|\ge \delta/2\big],$$
    and 
     \begin{align*}
     &\mathbf P_{\mathsf x_0}\big [\bar \sigma_{\mathsf B(\mathsf x_0,\delta)}\le s, |\bar X_s-\bar X_0|< \delta/2\big ]  \\ 
     &\le    \mathbf E_{\mathsf x_0}\Big [\mathbf 1_{\bar \sigma _{\mathsf B(\mathsf x_0,\delta) }(\mathsf x_0) \le s}     \mathbf P_{\bar X_{\bar \sigma_{\mathsf B(\mathsf x_0,\delta)}(\mathsf x_0) } (\mathsf x_0)}\big [|\bar X_{s-\bar \sigma_{\mathsf B(\mathsf x_0,\delta)}(\mathsf x_0) }-\bar X _0|\ge \delta/2  ]  \Big ].
   \end{align*} 
Let  $\mathsf x_0\in K$,   $\mathsf x=(x,v,z)\in \partial \mathsf B(\mathsf x_0,\delta) $ ($\Rightarrow \mathsf x\in K_\delta$), and $u\in [0,s]$. It holds  $ \mathbf P_{\mathsf x } \big [|\bar X_{s-u }- \mathsf x|\ge \delta/2  ]\le \mathbf P_{\mathsf x } \big [|\bar x _{s-u }-  x|\ge \delta/2  ]+ \mathbf P_{\mathsf x } \big [|\bar v _{s-u }- v|\ge \delta/2  ]+ \mathbf P_{\mathsf x } \big [|\bar z _{s-u }-z|\ge \delta/2  ]$.  Moreover, there exists $C>0$,  $\mathbf P_{\mathsf x } \big [|\bar z _{s-u }-  z|\ge \delta/2  ]\le \mathbf d_s:=\mathbf P  \big [  \sqrt{2\alpha} \sup_{t\in [0,s]} |B_t| + Cs (Y+1)\ge \delta/2  ]\to 0$ when $s\to 0^+$. The other terms are treated similarly to deduce the expected result.  
  \end{proof}

  \noindent
\textbf{Note} (alternative proofs of  Lemmas \ref{le.C4} and \ref{le.C4-2}). Since the processes we consider in this work have  a.s. continuous sample paths and are non explosive, we mention that another way to prove Lemmas \ref{le.C4} and \ref{le.C4-2} is to use~\cite[Lemma 4.3]{Girsanov} which requires \textit{local uniform stochastic continuity}~\cite[Eq. (3.4)]{Girsanov}.  For instance, to prove  the local uniform stochastic continuity for the generalized Langevin  process $(X_t,t\ge 0)$ (see \eqref{eq.GLintro}), namely that for any       compact subset $K$ of $\mathbf R^{3d}$,
$$
\sup_{\mathsf x\in K}\mathbf P [|X_s(\mathsf x)-\mathsf x|>0]\to 0 \text{ as } s\to 0^+,$$
  one first writes  for $R>0$ sufficiently large, say $R>R_K$, 
 $\mathbf P_{\mathsf x}[|X_s -\mathsf x|>0]=\mathbf P_{\mathsf x}[|X_s -\mathsf x|>0, s<\sigma_{\mathscr H_R}]+\mathbf P_{\mathsf x}[|X_s -\mathsf x|>0, \sigma_{\mathscr H_R}\le s]$,    $s\in (0,1]$ (note that this is the strategy described in  \eqref{eq.splitting}). Then, the desired result follows using  \eqref{eq.tauR} (to treat the term $\mathbf P_{\mathsf x}[|X_s -\mathsf x|>0, \sigma_{\mathscr H_R}\le s]$) together with - to treat the term $\mathbf P_{\mathsf x}[|X_s -\mathsf x|>0, s< \sigma_{\mathscr H_R}]$ -  the fact that the drift is bounded when: $s<\sigma_{\mathscr H_R}(\mathsf x)$ and   all the initial conditions $\mathsf x$ are in $K$. The details are very closed to those used in the proof of Proposition \ref{pr.C2}, and are thus left to the reader. Note that one advantage of the first proofs we gave of  Lemmas~\ref{le.C4} and~\ref{le.C4-2}    can be applied to processes with jumps, see e.g.~\cite[Section 6.2]{guillin2024large}. 
 We finally mention that  Lemmas \ref{le.C4} and \ref{le.C4-2} can also be proved using the same arguments as those used to prove~\cite[Lemma 1]{guillin2024large}  which is based on  Itô calculus and provides a rate of convergence when $s\to 0$. 
 \medskip

  We are now ready to prove  the following result. 
 \begin{thm}\label{pr.SFgamma-neq-0}
 Assume {\rm\textbf{[V{\tiny loc}]}}. Let $\gamma=0$. Then,  for every $t>0$, the nonkilled semigroup $P_t$ of the generalized Langevin process~\eqref{eq.GLintro} is strong Feller over $ \mathbf R^{3d}$. 
\end{thm}


\begin{proof}
Let $f\in b\mathcal B(\mathbf R^{3d})$ and $t>0$ be fixed. 
We want to prove that $\mathsf z\mapsto \mathbf E_{\mathsf z}[f(X_t)]$ is continuous over $\mathbf R^{3d}$. 

We write for $\mathsf x\in \mathbf R^{3d}$ and $R>0$ (see \eqref{eq.splitting}):
$$\mathbf E_{\mathsf x}[f(X_t)]=\mathbf E_{\mathsf x}[f(X_t)\mathbf 1_{t<\sigma_{\mathscr H_R}}]+\mathbf E_{\mathsf x}[f(X_t)\mathbf 1_{t\ge \sigma_{\mathscr H_R}}].$$
   Since the open set $\mathscr V_R=\{y\in \mathbf R^d, V(y)<R\}$ is relatively compact and $\nabla V$ is locally Lipschitz, we can consider\footnote{See e.g. \cite{mcshane}.} a globally  Lipschitz  vector field 
$\mathbf b_R:\mathbf R^d\to \mathbf R^d$  such that $\mathbf b_R= -\nabla V$ in a neighborhood of  $\overline{\mathscr V}_R$.  
Let us then introduce the unique strong solution $(\bar X^{R}_t= (\bar x^{R}_t,\bar v^{R}_t,\bar z^{R}_t),t\ge 0)$ to \eqref{eq.GLR-2} when $\mathbf b$ replaced by $\mathbf b_R$ and $\gamma=0$. 
Note that $(x,v,z)\in \mathscr H_R\Rightarrow x\in \mathscr V_R$.
Therefore,    the processes $(X_t,t\ge 0)$ and $(\bar X^{R}_t,t\ge 0)$ coincide in law  up to their first exit time from $\mathscr H_R$, and thus, denoting by $\bar \sigma_{\mathscr H_R}:=\inf\{t\ge 0, \bar X^R_t\notin\mathscr H_R \}$, we deduce that:
\begin{equation}\label{eq.S0}
\mathbf E_{\mathsf x}[f(X_t)]=\mathbf E_{\mathsf x}[f(\bar X^R_t)\mathbf 1_{t<\bar \sigma_{\mathscr H_R}}]+\mathbf E_{\mathsf x}[f(X_t)\mathbf 1_{t\ge \sigma_{\mathscr H_R}}].
\end{equation}
 Let us now assume that $\mathsf x_n\to \mathsf x\in \mathbf R^{3d}$. Let $R_0>0$ such that $\mathsf x_n,\mathsf x\in \mathscr H_{R_0}$.

Let us now prove \eqref{eq.theta}, i.e.  that for all $R>0$, as $n\to +\infty$,
\begin{equation}\label{eq.Exn}
\mathbf E_{\mathsf x_n}[f(\bar X^{R}_t)\mathbf 1_{t<\bar \sigma_{\mathscr H_R}}] \to \mathbf E_{\mathsf x}[f(\bar X^{R}_t)\mathbf 1_{t<\bar \sigma_{\mathscr H_R}}]. 
\end{equation}
To prove \eqref{eq.Exn}, we first prove that the nonkilled semigroup of 
$(\bar X^{R}_t ,t\ge 0)$ is strong Feller, i.e. that    
\begin{equation}
\label{SF-e}
\text{for all $R>0$ }, \mathsf z\in \mathbf R^{3d}\mapsto \mathbf E_{\mathsf z}[f(\bar X^{R}_t)] \text{ is continuous.}
\end{equation}
Note that  since $\mathbf b_R$ is globally Lipschitz, it holds  as $n\to +\infty$:
\begin{equation}
\label{SF-e1}
\sup_{s\in [0,t]}|\bar X_s^{R}(\mathsf x_n)\to \bar X_s^{R}(\mathsf x)|\to 0 \text{ in $\mathbf P$-probability}.
\end{equation}  
Using again that $\mathbf b_{R}$ is globally Lipschitz over $\mathbf R^{d}$,
we can use~\cite[Theorem 1.1]{delarue} to deduce that  for all $\mathsf z\in \mathbf R^{3d}$, $\bar X_t^{R}(\mathsf z)$ admits a density $\bar  p_t^{R}(\mathsf z,\mathsf y)$ (with respect to  the Lebesgue measure $d\mathsf y$ over $\mathbf R^{3d}$) which  moreover satisfies the following Gaussian upper bound:
\begin{equation}\label{eq.GLE-gauss}
\bar p_t^{R}(\mathsf z,\mathsf y)\le C^R_t\exp( -c^R_t |\boldsymbol \theta^{R}_t(\mathsf z)-\mathsf y|^2),
\end{equation}
for some $C^R_t,c^R_t>0$ independent of $\mathsf z, \mathsf y\in \mathbf R^{3d}$, and where $(\boldsymbol  \theta_t^R(\mathsf z)=(\boldsymbol x_t(\mathsf z),\boldsymbol v^R_t(\mathsf z),\boldsymbol z^R_t(\mathsf z)) ,t\ge 0)$ is the  solution   for $t\ge 0$ of the deterministic equation: 
$$
         \dot{\boldsymbol x} ^R_t=    \boldsymbol v^R_t , \, 
       \dot{\boldsymbol v} ^R_t= \mathbf b_R (  \boldsymbol x^R_t)  +\lambda \boldsymbol z^R_t,\,         \dot{\boldsymbol z}^R_t= -\alpha \boldsymbol z^R_t-\lambda \boldsymbol v^R_t,$$
with initial condition $\boldsymbol \theta^{R}_0(\mathsf z)=\mathsf z$. 
 
On the other hand, since $(\mathsf x_n)_n$ is bounded and  $\mathsf z\mapsto \boldsymbol  \theta^{R}_t(\mathsf z)$ is continuous (because $ \mathbf b_{R}$ is globally Lipschitz), it follows that 
there  exists $K>0$ such that for all $n$, $|\boldsymbol \theta^{R}_t(\mathsf x_n)|\le K$. Thus  
\begin{equation}\label{eq.pRe}
\bar p_t^{R}(\mathsf x_n,\mathsf y)\le C^R_t \exp( c^R_t K^2)  \exp( 2c^R_t K|\mathsf y|) \exp( -c^R_t |\mathsf y|^2).
\end{equation}
Let $\mu( d\mathsf y)= \kappa \exp(-|\mathsf y|) d\mathsf y$ (where $ \kappa^{-1}  = \int \exp(-|\mathsf y|) d\mathsf y$), and set  $\bar  q^{R}_t(\mathsf z,\mathsf y) = \kappa^{-1} \exp(|\mathsf y|)\bar p_t^{R}(\mathsf z,\mathsf y)$. 
It then follows from \eqref{eq.pRe} that for any fixed $R>0$, 
\begin{align*}
 (\bar  q^{R}_t(\mathsf x_n,\mathsf y))_n \text{ is $\mu$-uniformly integrable}.
\end{align*}
Consequently, for any $f\in b\BB(\mathbf R^{3d})$,  by Proposition~\ref{pr.wu1999}, as $n\to +\infty$, $f(\bar X_t^{R}(\mathsf x_n))\to f(\bar X_t^{R}(\mathsf x))$ in probability, and then, because $f$ is bounded, this convergence  also  holds in   any $r$th mean ($r\ge 1$).  This proves \eqref{SF-e}.

 We are now in position to prove Equation \eqref{eq.Exn}. The proof is inspired from the one of~\cite[Theorem~2.2]{chung2001brownian} (see also~~\cite[Section 5.2]{Girsanov}). 
Let  $R>0$, $K$ be  a compact subset of $\mathscr H_R$, $t>0$, and $\delta:= \text{dist}(K, \partial \mathscr H_R)>0$. 
Let also  $0\le s\le t$, $\mathsf z\in \mathscr H_R$, and $f\in b\mathcal B(\mathscr H_R)$. Then, by the Markov property, it holds  
$$\mathbf E_{\mathsf z}[f(\bar X^{R}_t)\mathbf 1_{t<\bar \sigma_{\mathscr H_R}}]= \mathbf E_{\mathsf z}[\mathbf 1_{s<\bar \sigma_{\mathscr H_R}} \psi_s(\bar X^R_s)],$$ where $\psi_s(\mathsf z)=\mathbf E_{\mathsf z}[\mathbf 1_{t-s<\bar \sigma_{\mathscr H_R}} f(\bar X^R_{t-s})]$. Therefore:   
\begin{align*}
 \sup_{\mathsf z\in K}| \mathbf E_{\mathsf z}[f(\bar X^{R}_t)\mathbf 1_{t<\bar \sigma_{\mathscr H_R}}] -\mathbf E_{\mathsf z}[\psi_s(\bar X^{R}_t)]| &\le \Vert f\Vert_{\infty}\sup_{\mathsf z\in K} \mathbf P_{\mathsf z}[\bar \sigma_{\mathscr H_R}\le s] \le  \Vert f\Vert_{\infty}\sup_{\mathsf z\in K} \mathbf P_{\mathsf z}[\bar \sigma^R_{\mathsf B(\mathsf z,\delta)}\le s].
\end{align*}
which tends to $0$ as $s\to 0^+,$ thanks to Lemma~\ref{le.C4-2}, 
where $$\bar \sigma^R_{\mathsf B(\mathsf z,\delta)}(\mathsf x)=\inf \{t\ge 0, \bar X^R_t(\mathsf x)\notin \mathsf   B(\mathsf x,\delta)\}.$$
 Hence, using \eqref{SF-e},  one deduces that  $\mathsf z\in \mathscr H_R \mapsto    \mathbf E_{\mathsf z}[f(\bar X^{R}_t)\mathbf 1_{t<\bar \sigma_{\mathscr H_R}}]$ is continuous.

 
We now conclude the proof of Theorem~\ref{pr.SFgamma-neq-0}. Consider  $\epsilon >0$. Thanks to \eqref{eq.tauR} and because $\mathsf H_{{\rm GL}}$ is locally bounded, there exists  $R_\epsilon\ge R_0$ such that
\begin{equation}\label{eq.S1}
|\mathbf E_{\mathsf x}[f(X_t)\mathbf 1_{t\ge \sigma_{\mathscr H_{R_\epsilon}}}]|+\sup_{n\ge 0}|\mathbf E_{\mathsf x_n}[f(X_t)\mathbf 1_{t\ge \sigma_{\mathscr H_{R_\epsilon}}}]|\le  \epsilon/2.
\end{equation} 
Using \eqref{eq.Exn}, there exists  $N_\epsilon\in \mathbf N$ such that for all $n\ge N_\epsilon$,  $$ \big|\mathbf E_{\mathsf x_n}[f(\bar X^{R_\epsilon}_t)\mathbf 1_{t<\bar\sigma_{\mathscr H_{R_\epsilon}}}] - \mathbf E_{\mathsf x}[f(\bar X^{R_\epsilon}_t)\mathbf 1_{t<\bar\sigma_{\mathscr H_{R_\epsilon}}}]\big|\le \epsilon/2.$$
  Using \eqref{eq.S0} and \eqref{eq.S1},  for all $n\ge N_\epsilon$, $|\mathbf E_{\mathsf x_n}[f(X_t)]-\mathbf E_{\mathsf x}[f(X_t)]|\le \epsilon$. This concludes the proof of Theorem~\ref{pr.SFgamma-neq-0}. 
 \end{proof}

 \medskip

 \noindent
\textbf{On Assumption \textbf{(C4)}}.   Let us consider a subdomain $\mathscr D$ of $\mathbf R^{3d}$ of the form
 $$\mathscr D= \mathscr O\times \mathbf R^d \times \mathbf R^d,$$
 where $\mathscr O$ is a subdomain  of $\mathbf R^d$. Thanks to Lemma \ref{le.C4}, we can check \textbf{(C4)}. 
 
 \begin{prop}\label{co.SFD} Assume {\rm \textbf{[V{\tiny loc}]}}. 
Then, for all $t>0$ and $\gamma \ge 0$,  $P_t^{\mathscr D}$  is strong Feller. In particular \textbf{(C4)} holds. 
 \end{prop}

 \begin{proof}   
 Pick    a  compact subset $K$  of $\mathscr D$ and $t>0$.  Set $\delta:= \text{dist}(K, \partial \mathscr D)>0$ (where $\partial \mathscr D=\partial \mathscr O\times \mathbf R^d \times \mathbf R^d$). 
Using the Markov property, one has for $0\le s\le t$, $\mathsf x\in \mathscr D$, and $f\in b\mathcal B(\mathscr D)$, $P_t^{\mathscr D}f(\mathsf x)= \mathbf E_{\mathsf x}[\mathbf 1_{s<\sigma_{\mathscr D}} g_s(X_s)]$, where $g_s(\mathsf z)=\mathbf E_{\mathsf z}[\mathbf 1_{t-s<\sigma_{\mathscr D}} f(X_{t-s})]$. By \textbf{(C1)}, $P_sg_s$ is continuous over $\mathbf R^{3d}$. In addition, it holds:
$\sup_{\mathsf x\in K}| P_t^{\mathscr D}f(\mathsf x)-P_sg_s(\mathsf x)|\le \Vert f\Vert_{\infty}\sup_{\mathsf x\in K} \mathbf P_{\mathsf x}[\sigma_{\mathscr D}\le s]\le  \Vert f\Vert_{\infty}\sup_{\mathsf x\in K} \mathbf P_{\mathsf x}[\sigma_{\mathscr B(x,\delta)}\le s]\to 0$ as $s\to 0^+$, thanks to Lemma \ref{le.C4}. Thus $P_t^{\mathscr D}f$ is continuous. 
\end{proof}



 \subsection{On Assumption \textbf{(C5)}} 
 \label{sec.S5GL}
 
 In this section,  we check \textbf{(C5)}. To do so, the  well-known starting point is  the  knowledge  of the support of the law of the trajectories of the Brownian motion. To check \textbf{(C5)}, we thus  construct   suitable control curves.  The details of the proof of  \textbf{(C5)} will be given since  $\nabla V$ is only locally Lipschitz.  In all this section,  $\mathscr D= \mathscr O\times \mathbf R^d \times \mathbf R^d$ and $\mathscr O$ is a subdomain of $\mathbf R^d$ such that  $ \mathbf R^d\setminus \overline{\mathscr O}$ is nonempty.  

 \subsubsection{The case when $\gamma=0$}
 To prove Assumption \textbf{(C5)}, we will  construct, for any  $  T>0$ and any  two points $\mathsf x_0,\mathsf x_T\in \mathscr D$, a so-called control curve $\mathfrak{l}\in [0,T]\to  \mathscr O$ joining $\mathsf x_0$ to $\mathsf x_T$ (this will be done using local polynomial interpolations).
 Note here that we need to impose that the range of the control curve  $\gamma$  lies in $\mathscr O$ to ensure the condition that $t<\sigma_{\mathscr D}(\mathsf x_0)$ when $X_{[0,T]}(\mathsf x_0)$ is sufficiently close to $\gamma$ on $[0,T]$ (for the supremum norm over $[0,T]$). 
 
 We start with the following lemma.
 
 \begin{lem}\label{loc.construction}
 Let us consider $t_*>0$, $\mathsf x_0=(x_0,y_0,z_0)\in \mathbf R^{3d}$ and $\mathsf x_1=(x_1,v_1,z_1)\in \mathbf R^{3d}$. Then, there exists a smooth curve $ \ell:[0,t_*]\to \mathbf R^d$ such that $ \ell(0)=x_0$, $\dot \ell(0)=v_0$, $\ddot  \ell(0)=z_0$, $ \ell(t_*)=x_1$, $\dot  \ell(t_*)=v_1$, $\ddot  \ell(t_*)=z_1$, and 
 $$\sup_{s\in [0,t_*]}|\ell(s)-x_0|\le C|x_0-x_1|+ C(|v_0|+|v_1|)t_*+ C(|z_0|+|z_1|)t_*^2,$$
 where $C>0$ is a universal constant. 
 \end{lem} 
In the following, we say that such a curve $\gamma$ joins $\mathsf x_0$ to $\mathsf x_T$. 
 \begin{proof}   Choose $\ell(s)=\sum_{k=0}^5 c_k s^k$ where 
 $c_0=x_0$, $c_1=v_1$, $c_2=z_0/2$, $c_3= -10x_0/t_*^3-6v_0/t_*^2-3z_0/(2t_*)+10x_1/t_*^3-4v_1/t_*^2+z_1/(2t_*)$, $c_4= 15x_0/t_*^4+8v_0/t_*^3+3z_0/(2t_*^2) -15x_1/t_*^4+7v_1/t_*^3-z_1/t_*^2$, $c_5=-6x_0/t_*^5-3v_0/t_*^4-z_0/(2t_*^3)+6x_1/t_*^5-3v_1/t_*^4+z_1/(2t_*^3)$. 
 \end{proof}
  
 \begin{lem}\label{loc.construction2}
 Let us consider $T>0$, $\mathsf x_0=(x_0,y_0,z_0)\in \mathscr D$, and $\mathsf x_T=(x_T,v_T,z_T)\in \mathscr D$. Then, there exists a $\mathcal C^2$ and piecewise $\mathcal C^3$ curve $ \mathfrak{l}:[0,T]\to \mathbf R^d$,  such that $\mathfrak{l}(0)=x_0$, $\dot{\mathfrak{l}}(0)=v_0$, $\ddot{\mathfrak{l}}(0)=z_0$, $\mathfrak{l}(T)=x_T$, $\dot {\mathfrak{l}}(T)=v_T$, $\ddot{\mathfrak{l}}(T)=z_T$, and  
 $${\rm  Ran} (\mathfrak{l}) \subset \mathscr O.$$
 \end{lem} 
\begin{proof} 
The set $\mathscr O$ is path-connected since it is connected and open in $\mathbf R^d$. 
Let $\Psi: [0,T]\to \mathscr O$ be an injective curve such that $\Psi(0)=x_0$ and $\Psi(T)=x_T$. Let $\delta>0$ be such that  the closure of the $\delta$-neighborhood $\mathcal U_\delta$ of Ran$(\Psi)$ is included in~$\mathscr O$. 

Let $\epsilon >0$. Take a subdivision of   Ran$(\Psi)$     containing $N_\epsilon+1\in \mathbf N^*$ points $\{x_0,x_1^\epsilon,$ $\ldots, x_{N_\epsilon-1}^\epsilon, x_T\}$ with $|x_j^\epsilon-x_{j+1}^\epsilon|\le \epsilon$, $j\in \{0,\ldots, N_\epsilon-1\}$ (with $x_{0}^\epsilon:=x_0$ and $x_{N_\epsilon}^\epsilon:=x_T$).  

Note that $N_\epsilon \to +\infty$ as $\epsilon \to 0^+$. Set now $t_*^\epsilon=T/N_\epsilon$.  Pick two unimportant points $v,z\in \mathbf R^d$. By Lemma~\ref{loc.construction}, we can  consider  $N_\epsilon$ curves $\ell^\epsilon_0: [0,t_*^\epsilon]\to \mathbf R^d$ joining 
$\mathsf x_0$ to $\mathsf x^\epsilon_1=(x_1^\epsilon, v,z)$, $\ell^\epsilon_1: [0,t_*^\epsilon]\to \mathbf R^d$ joining 
 $\mathsf x_1^\epsilon$ to $\mathsf x_2^\epsilon=(x_2^\epsilon, v,z)$, \ldots, $\ell^\epsilon_{N_\epsilon-1}: [0,t_*^\epsilon]\to \mathbf R^d$ joining 
 $\mathsf x_{N_\epsilon-1}^\epsilon=(x_{N_\epsilon-1}^\epsilon, v,z)$ to $\mathsf x_T$. The curve $\mathfrak{l}^\epsilon$ is then defined on $
 [0,T]$ by:
 $$\text{for all  $t\in [0,T]$}, \mathfrak{l}^\epsilon(t)= \ell^\epsilon_j(t-jt^\epsilon_*) \text{ if } jt^\epsilon_*\le t<(j+1)t^\epsilon_*, j\in \{0, \ldots, N_\epsilon-1\},$$
 and $\mathfrak{l}^\epsilon(T)= x_T$. 
Note that by Lemma~\ref{loc.construction}, if $\epsilon$ is small enough (say $\epsilon<\epsilon_0$), for all $j\in \{0, \ldots, N_\epsilon-1\}$, $\text{ Ran} (\ell_j^\epsilon)\subset  \mathcal U_\delta$, 
so that $\text{Ran} (\mathfrak{l}^\epsilon)\subset \mathscr O$. 
The proof of the lemma is complete choosing any $\mathfrak{l}^\epsilon$, with $0<\epsilon<\epsilon_0$.  
\end{proof}

 
We are now ready to prove that Assumption \textbf{(C5)} is satisfied when
 $\gamma=0$. 

 \begin{prop}\label{pr.P2C5}
 Assume  $\gamma=0$ and {\rm\textbf{[V{\tiny loc}]}}. Then, for any $T>0$, $\mathsf x_0=(x_0,v_0,z_0)\in \mathscr D$ and any nonempty open subset $\mathsf O$ of  $\mathscr D$,  it holds:
\begin{equation}
\label{eq.P2}
\mathbf P_{\mathsf x_0}[X_T\in \mathsf O, T<\sigma_{\mathscr D}]>0.
\end{equation}
Moreover,   \textbf{(C5)} holds.
 \end{prop}
 
 
 \begin{proof}  
With Lemma~\ref{loc.construction2} at hand, such  a result can be proved with the Stroock-Varadhan support theorem~\cite{SV,ben}. Since, to the best of our knowledge,  this result is usually stated with rather  stringent assumptions on the regularity and on  the boundedness of the drift, we give a rather natural   proof which has the advantage  to only use  the support  of the law of the trajectory of the standard Brownian motion (actually we only use  the fact that sufficiently smooth curves are in its support) and which applies to locally Lipschitz drifts. 
 Pick $\mathsf x_T\in  \mathsf O$. Using  Lemma~\ref{loc.construction2}, one can consider a $\mathcal C^2$ and piecewise $\mathcal C^3$ curve $ \mathfrak{l}:[0,T]\to \mathscr O$  such that $\mathfrak{l}(0)=x_0$, $\dot{\mathfrak{l}}(0)=v_0$, $\ddot{\mathfrak{l}}(0)=\lambda z_0-\nabla V(x_0)$, $\mathfrak{l}(T)=x_T$, $\dot {\mathfrak{l}}(T)=v_T$, $\ddot{\mathfrak{l}}(T)=\lambda z_T-\nabla V(x_T)$.  Let us then define the  function $\mathsf h: [0,T]\to \mathbf R^d$  by
 $$
 \mathsf h(t) =\frac{1}{\lambda \sqrt{2\alpha}}\, \frac{d}{dt}[\ddot{\mathfrak{l}}+\nabla V({\mathfrak{l}})](t), \ t\in [0,T].
 $$
  Note that $  \mathsf  h \in \mathcal L^{\infty}([0,T],\mathbf R^d)$ since $t\mapsto [\ddot{\mathfrak{l}}+\nabla V({\mathfrak{l}})](t)$ belongs to the Sobolev space $ \mathcal W^{1,\infty}([0,T],\mathbf R^d)$. 
Recall that the support (for the supremum norm over $[0,T]$) of the law of the trajectory of the standard Brownian motion on $[0,T]$ is the closure   of the set 
\begin{equation}\label{eq.Law}
\mathcal S_{{\rm supp}}([0,T],\mathbf R^d)=\Big \{ t\in [0,T]\mapsto \int_0^t u(s) ds<+\infty,  u \in \mathcal L^2([0,T],\mathbf R^{d})\Big \}.
\end{equation}
Therefore, for all $\eta>0$,    $\mathbf P[\mathcal A_\eta]>0$, where $\mathcal A_\eta:=\{\sup_{t\in [0,T]}|B_t-\int_0^th(s)ds|<\eta\}$. 
 Let us consider   an open bounded  neighborhood  $\mathcal V_{\mathfrak{l}}$ of  Ran$(\mathfrak{l})$  in $\mathbf R^d$ such that $\overline{\mathcal V}_{\mathfrak{l}}\subset \mathscr O$ (note that in particular  $\mathsf x_0,\mathsf x_T\in \mathcal V_{\mathfrak{l}}$). 
 For all  $\epsilon>0$ small enough (say $\epsilon \in (0,\epsilon_{\mathfrak{l}})$, $\epsilon_{\mathfrak{l}}>0$), the following two conditions are satisfied: 
  \begin{itemize}
  \item[\textbf a.] For any $\mathfrak{c} : [0,T]\to \mathbf R^d$   such that  $\sup_{s\in [0,T]}|\mathfrak{c} (s)-\mathfrak{l}(s)|\le \epsilon$, Ran$(\mathfrak{c} )\subset \mathcal  V_{\mathfrak{l}}$. 
  \item[\textbf b.]  For any $\mathsf x\in \mathbf R^{3d}$ such that $\mathsf x\in \mathsf B(\mathsf x_T,\epsilon)$, $\mathsf x\in  \mathsf O$. 
   \end{itemize} 
Denote  by $\mathscr D_{\mathfrak{l}}=\mathcal V _{\mathfrak{l}}\times \mathbf R^d\times \mathbf R^d \subset \mathscr D$.  Let us  consider  a globally  Lipschitz  vector field 
$\mathbf b_{\mathfrak{l}}:\mathbf R^d\to \mathbf R^d$  such that $\mathbf b_{\mathfrak{l}}= -\nabla V$ in a neighborhood of  $\overline{\mathcal V}_{\mathfrak{l}}$.    
Recall   $(\bar X^{\mathfrak{l}}_t= (\bar x^{\mathfrak{l}}_t,\bar v^{\mathfrak{l}}_t,\bar z^{\mathfrak{l}}_t),t\ge 0)$ is the solution to \eqref{eq.GLR-2} when $\mathbf b=\mathbf b_{\mathfrak{l}}$.  
Note that $
\mathbf P_{\mathsf x_0}[X_T\in  \mathsf O, T<\sigma_{\mathscr D}]\ge \mathbf P_{\mathsf x_0}[X_T\in \mathsf O,   T<\sigma_{\mathscr D_{\mathfrak{l}}}] =\mathbf P_{\mathsf x_0}[ X^{\mathfrak{l}}_T\in  \mathsf O, T<\bar \sigma _{\mathscr D_{\mathfrak{l}}}]$ (where $\bar \sigma _{\mathscr D_{\mathfrak{l}}}$ is the first exit time of the process $\bar X^{\mathfrak{l}}$ from $\mathscr D_{\mathfrak{l}}$). To prove \eqref{eq.P2}, it is therefore enough to show  that $\mathbf P_{\mathsf x_0}[ X^{\mathfrak{l}}_T\in  \mathsf O, T<\bar \sigma _{\mathscr D_{\mathfrak{l}}}]>0$, which is the purpose of what follows.  

   Let $(\bar X^\circ_t=(x_t^\circ,v_t^\circ,z_t^\circ)\in \mathbf R^{3d},t\ge 0)$ be the solution to 
\begin{equation}\label{eq.GL-G}
        d x_t^\circ=   v_t^\circ dt, \  d  v_t^\circ = \mathbf b_{\mathfrak{l}} (  x_t^\circ)dt +\lambda  z^\circ_tdt , \ d  z^\circ_t=   \sqrt{2\alpha}\, dB_t.
\end{equation} 
Define $f_t(\mathsf x)= -\alpha z_t^\circ(\mathsf x)-\lambda v_t^\circ (\mathsf x)$, $t\ge 0$.  Recall that since $\mathbf b_{\mathfrak{l}} $ is globally Lipschitz (see the beginning of the proof of Lemma \ref{le.C4-2}),  for all $\mathsf x\in \mathbf R^{3d}$ and $T>0$, there exists $C>0$, 
$$\sup_{s\in [0,T]}|\bar X^\circ_s| \le C(1+ \mathsf Y_T)$$ where $\mathsf Y_T:= \sup_{s\in [0,T]}|B_s |$. Hence, for  $\epsilon>0$ small enough, 
$\mathbf E  [e^{\epsilon|f_t(\mathsf x)|^2} ]\le C \mathbf E  [e^{C\epsilon \mathsf Y_T^2}  ]<+\infty$.  
 Therefore, one can use~\cite[Theorem~3.1 in Section~7]{friedman1975} and~\cite[Theorem~1.1  in Section~7]{friedman1975}, to deduce that the law of $(\bar X^{\mathfrak{l}}_t,t\in [0,T])$    is equivalent 
to the law of   $(\bar X^\circ_t,t\in [0,T])$. 
Let us thus prove that $\mathbf P_{\mathsf x_0}[ \bar X^\circ_T\in  \mathsf O, T<\bar \sigma^\circ _{\mathscr D_{\mathfrak{l}}}]>0$ (where $\bar \sigma^\circ _{\mathscr D_{\mathfrak{l}}}:=\inf\{t\ge 0,\bar X^\circ_t\notin \mathscr D_{\mathfrak{l}}\}$). 

 Set $\mathfrak n(t)=(\mathfrak{l}(t),\dot{\mathfrak{l}}(t),\lambda^{-1} (\ddot{\mathfrak{l}}(t)+\nabla V({\mathfrak{l}}(t))) )$, for $t\in [0,T]$. Note that 
 for $t\in [0,T]$, 
 $$ \mathsf h(t) =\frac{1}{\lambda \sqrt{2\alpha}}\, \frac{d}{dt}[\ddot{\mathfrak{l}}-  \mathbf b_{\mathfrak{l}} ({\mathfrak{l}})](t).$$  
Then, using that  $\mathbf b_{\mathfrak{l}}$ is globally Lipschitz and that $\mathfrak n(0)= \mathsf x_0$,     it is straightforward to get that  $\sup_{s\in [0,T]}|\bar X^\circ_s-\mathfrak n(t) | \le C_T \eta$ on $\mathcal A_\eta$. Choose $\eta>0$ such that $C\eta<\epsilon_{\mathfrak{l}}$. By Item \textbf{a} above,  $T<\bar \sigma^\circ _{\mathscr D_{\mathfrak{l}}}$. In addition, since  $\mathfrak n(T)=\mathsf x_T$, by Item \textbf{b} above, $\bar X^\circ_T\in \mathsf O$. Therefore, $\mathbf P_{\mathsf x_0}[ \bar X^\circ_T\in \mathsf O, T<\bar \sigma^\circ _{\mathscr D_{\mathfrak{l}}}]>0$. This ends the proof of \eqref{eq.P2}. 

Since $ \mathbf R^d\setminus \overline{\mathscr O}\neq \emptyset$, we also prove with similar arguments  that for all  $\mathsf z\in \mathscr D$,  $\mathbf P_{\mathsf z}[\sigma_{\mathscr D}<+\infty]>0$. This ends the proof of Proposition \ref{pr.P2C5}.
 \end{proof}

 \subsubsection{The case when $\gamma>0$}
 In this section we prove the following result.  
  \begin{prop}\label{pr.SFgamma-not0}
Assume $\gamma>0$ and {\rm\textbf{[V{\tiny loc}]}}. Then,    \textbf{(C5)} holds. 
\end{prop}
\begin{proof}   
  In view of~\eqref{eq.Girsanov}, it is enough to prove  \textbf{(C5)} for the process $(X_t^0,t\ge 0)$ defined in~\eqref{eq.GL0}.
Let $T>0$, $\mathsf x_0=(x_0,v_0,z_0)\in \mathscr D$,  $ \mathsf O$ be a nonempty  open subset of   $\mathscr D$, and  $\mathsf x_T=(x_T,v_T, z_T)\in  \mathsf O$. 
Let also  $\mathfrak e :[0,T]\to \mathbf R^d$ be  $\mathcal C^1$ and piecewise $\mathcal C^2$  such that $\mathfrak e (0)=x_0$, $\dot{\mathfrak e }(0)=v_0$,  $\mathfrak e (T)=x_T$, $\dot{\mathfrak e }(T)=v_T$, and   ${\rm  Ran} (\mathfrak e )\subset \mathscr O$.   
  Let $\mathfrak u:[0,T]\to \mathbf R^d$ be a smooth curve with $\mathfrak u(0)=z_0 $ and $\mathfrak u(T)=z_T $. Define  $\mathfrak m (t)=(\mathfrak e (t),  \dot{\mathfrak e }(t),    \mathfrak u (t))$ (note that $\mathfrak m (0)=\mathsf x_0$ and $ {\mathfrak m }(T)=\mathsf x_T$). For any $\eta>0$, the event 
 $\mathcal A_\eta:= \{\sup_{t\in [0,T]}|\sqrt{2\gamma}W_t-\int_0^t\ddot{\mathfrak e } |<\eta \}\cap  \{\sup_{t\in [0,T]}|\sqrt{2\alpha}B_t-\int_0^t\dot{\mathfrak u} |<\eta \}$
   has positive probability (see~\eqref{eq.Law}). A straightforward computation shows that when $X_0^0=\mathsf x_0$, $\mathcal A_{\epsilon/C_T}\subset \{\sup_{s\in [0,T]}|X_s^0-\mathfrak m (s)|<\epsilon\}$ for some   $C_T>0$. Hence,  $\mathbf P_{\mathsf x_0}[\sup_{s\in [0,T]}|X_s^0-\mathfrak m (s)|<\epsilon]>0$. 
    This concludes the proof of Proposition~\ref{pr.SFgamma-not0}. 
\end{proof}



  \subsection{On Assumption \textbf{(C3)}} 
 
\label{sec.CC3}

  In this  section, we  construct  two Lyapunov functions $\mathsf W_\delta:\mathbf R^{3d}\to [1,+\infty)$  satisfying Assumption   \textbf{(C3)}, where $\delta>0$ is a parameter to be chosen later.  
%

 
Assume  \textbf{[V{\tiny poly-$x^k$}]}. 
 We define  the vector field $\mathsf L$ as follows. Let $\chi:\mathbf R^d\to [0,1]$ be a smooth   function such that $\chi(x)=0$ if $|x|\le 1$ and $\chi(x)=1$ if $|x|\ge 2$. 
 We define $\mathsf J(x)=x \,  |x|^{\beta-1} \chi(x), \ \beta\in [0,1]$.  Note that $\mathsf J$ is $\mathcal C^1$ and  the first derivatives of $\mathsf J$ are bounded over $\mathbf R^d$ (because $\beta\le 1$), say by $C_{\mathsf J}:=\sup_{\mathbf R^d} |\text{Jac} (\mathsf J)|_2>0$ (where $|\mathsf M|_2:=\sup \{|\mathsf My|, |y|=1\}$, $\mathsf M\in \mathcal M_d(\mathbf R)$).  
One then sets:
$$
\mathsf L = \kappa \mathsf J, \ \kappa:= \frac{\lambda}{2C_{\mathsf J}},
$$
so that (and this will be used in the case $\gamma=0$ below),
\begin{equation}\label{eq.lambda2}
\mathfrak C_{\mathsf L}:=\sup_{\mathbf R^d} |\text{Jac} (\mathsf  L)|_2\le \lambda/2.
\end{equation}
   For all $(x,v,z)\in \mathbf R^{3d}$, $\mathfrak b\ge 0$, and $\mathfrak h,\mathfrak a>0$, we define (see \eqref{eq.H-GL}), $\mathsf F_0(x,v,z)=\mathfrak h\mathsf H_{{\rm GL}}(x,v,z)+\mathfrak a\mathsf L(x)\cdot v+\mathfrak b v\cdot z$. 
The parameter $\beta>0$ will be chosen such that 
\begin{equation}\label{eq.inf}
\inf_{\mathbf R^{3d}}\mathsf F_0 \in \mathbf R.
\end{equation}
We then set  
$\mathsf F_{{\rm GL}} =\mathsf F_0 -\inf_{\mathbf R^{3d}}\mathsf F_0 +1$ and 
  \begin{equation}\label{eq.Lyapunov2}
  \mathsf W_\delta =\exp\big [ \mathsf F_{{\rm GL}} ^\delta\big ], \ \text{ where }  \frac{1-\beta}{k}<\delta  \le 1.
   \end{equation}
In the following, for ease of notation we simply write $\mathsf F$ for $\mathsf F_{{\rm GL}}$. Since  $\mathsf F^{1-\delta}\ge 1$, a straightforward computation  implies that over $\mathbf R^{3d}$,
\begin{equation}\label{eq.WF}
\frac{\mathcal L_{{\rm GL}}\mathsf W_\delta}{ \mathsf W_\delta}\le \frac{\delta }{\mathsf F^{1-\delta}}\Big[\mathcal L_{{\rm GL}} \mathsf F+ \delta \gamma |\nabla _v \mathsf F|^2+\delta \alpha |\nabla _z \mathsf F|^2\Big],
\end{equation}
  where $\mathcal L_{{\rm GL}}\mathsf F= v\cdot \nabla _x \mathsf F + (-\gamma v -  \nabla _x V +\lambda z)\cdot \nabla_v\mathsf F + \gamma \Delta_v\mathsf F- (\alpha z+\lambda v)\cdot \nabla _z\mathsf F+ \alpha \Delta_z \mathsf F$.
We also have   $\nabla _x \mathsf F(x,v,z)= \mathfrak h\nabla _x V + \mathfrak a \text{ Jac}(\mathsf L)(x)\, v  $, $\nabla _v \mathsf F(x,v,z)= \mathfrak hv + \mathfrak a\mathsf L(x)+\mathfrak b z$, and $\nabla _z \mathsf F(x,v,z)= \mathfrak hz+\mathfrak b v$.  Consequently, one has for all $(x,v,z)\in \mathbf R^{3d}$:
  \begin{align}
  \nonumber
  \mathcal L_{{\rm GL}}\mathsf F(x,v,z)&=-\lambda \mathfrak b |v|^2 -\gamma \mathfrak h |v|^2+\mathfrak a v\cdot \text{Jac}(\mathsf L)(x) v-\gamma\mathfrak a \mathsf L(x)\cdot v-\mathfrak a \nabla V(x)\cdot \mathsf L(x)\\
\nonumber
  &\quad -\alpha \mathfrak h |z|^2+ \lambda \mathfrak b |z|^2+\lambda \mathfrak a z\cdot \mathsf L(x)  -\gamma \mathfrak b v\cdot z-  \mathfrak b  \nabla V(x)\cdot z  -\alpha \mathfrak b z\cdot v\\
  \label{eq.E-l1}
  &\quad + \mathfrak h d (\gamma+\alpha)
  \end{align}
  and 
  \begin{align}\label{eq.E-l2}
  (\delta \gamma |\nabla _v \mathsf F|^2+\delta \alpha |\nabla _z \mathsf F|^2)(x,v,z)= \delta \gamma|  \mathfrak hv + \mathfrak a\mathsf L(x)+\mathfrak b z|^2+ \delta \alpha  | \mathfrak hz+\mathfrak b v|^2.
    \end{align}

 \subsubsection{The case when $\gamma>0$}
    
    In this section, $\gamma>0$, 
\begin{equation}\label{eq.Bbeta}
 \mathfrak b=0, \, k>1, \, 0< \beta< \min(1,k/2,k-1).
\end{equation}
Let us now check \eqref{eq.inf}. Let $p_0,q_0>1$ such that $1/p_0+1/q_0=1$.   
Then, using \textbf{[V{\tiny poly-$x^k$}]}, for all $(x,v,z)\in \mathbf R^{3d}$, one has  if $|x|\ge r_V$,
$$\mathsf F_0(x,v,z) \ge  c_V  {\mathfrak  h}|x|^k+ \frac {\mathfrak  h}2 |v|^2+\frac  {\mathfrak  h}2 |z|^2- 
\frac{\mathfrak a  \kappa }{p_0} |x|^{\beta p_0}-\frac{\mathfrak a \kappa }{q_0} |v|^{q_0}.$$
Pick $\epsilon >0$ small enough such that $p_0=k/\beta-\epsilon>1$ and 
$q_0=(k-\epsilon\beta)/(k-\epsilon \beta-\beta)<2$ (which is possible since the latter quantity converges to $k/(k-\beta)<2$ as $\epsilon \to 0^+$, by \eqref{eq.Bbeta}). 
Note that $\beta p_0<k$. Thus, for any $\mathfrak h, \mathfrak a>0$, there exists  $C,c>0$ such that for all $(x,v,z)\in \mathbf R^{3d}$, 
$$
 \mathsf F_0(x,v,z) \ge  c (|x|^k+  |v|^2 + |z|^2)  
 -C.
$$
 Thus   \eqref{eq.inf} holds. Note also that  for any $\mathfrak h, \mathfrak a>0$,  there exist  $c'>0$,     for all  $(x,v,z)\in \mathbf R^{3d}$, 
 \begin{equation}\label{eq.F3}
  \mathsf F_0(x,v,z) \le  c' ( |x|^k+  |v|^2+ |z|^2+1) .
  \end{equation}
 


    
  \begin{prop}\label{pr.C3-1}
  Assume $\gamma>0$, {\rm \textbf{[V{\tiny poly-$x^k$}]}},  $k>1$,  and \eqref{eq.Bbeta}.   
   Then, for any $\mathfrak h,\mathfrak a>0$  small enough (these conditions are made explicit in the proof),   \textbf{(C3)} is satisfied with the function $\mathsf W_\delta$ defined in \eqref{eq.Lyapunov2}. 
  \end{prop}
      
  
  \begin{proof} 
%
  Recall that  $\mathfrak b=0$. Let $(x,v,z)\in \mathbf R^{3d}$ with  $|x|\ge \max (2,r_V)$. 
  Using   \textbf{[V{\tiny poly-$x^k$}]}, \eqref{eq.E-l1} and \eqref {eq.E-l2},   one gets
  \begin{align*}
  &(\mathcal L_{{\rm GL}} \mathsf F+ \delta \gamma |\nabla _v \mathsf F|^2+\delta \alpha |\nabla _z \mathsf F|^2)(x,v,z)\\
  &\le -\gamma \mathfrak h |v|^2+\mathfrak a \mathfrak C_{\mathsf L}  |v|^2+\gamma\mathfrak a  \kappa |x|^{\beta}   |v|-\mathfrak a  \kappa c_V |x|^{k-1+\beta} \\
  &\quad +2\delta \gamma  \mathfrak h^2|v|^2 + 2\delta \gamma \mathfrak a^2|\mathsf L(x)|^2+\lambda \mathfrak a \kappa |x|^{\beta} |z|-\alpha \mathfrak h |z|^2 +\delta \alpha \mathfrak h^2 |z|^2+ d\mathfrak h (\alpha + \gamma) .
  \end{align*}
  Choose $\mathfrak h>0$ small enough such that $ -\gamma \mathfrak h+2\delta \gamma  \mathfrak h^2<0$ and $-\alpha \mathfrak h+\delta \alpha \mathfrak h^2<0$. Then, choose $\mathfrak a>0$ small enough such that $\mathfrak C_{\mathsf L}\mathfrak a -\gamma \mathfrak h+2\delta \gamma  \mathfrak h^2<0$. Fix now such $\mathfrak h, \mathfrak a>0$.
Recall that $   |x|^{\beta}  |y|\le  \frac{1}{p_0} |x|^{\beta p_0}+\frac{1}{q_0} |y|^{q_0}= o( |x|^k+|y|^2)$ as $|x|+|y|\to +\infty$. In addition $|\mathsf L(x)|^2=\kappa^2|x|^{2\beta}$ and $2\beta <k-1+\beta$ since $\beta<k-1$ (see \eqref{eq.Bbeta}).  This implies that there exist $C,c>0$ such that 
 \begin{align*}
  (\mathcal L_{{\rm GL}} \mathsf F+ \delta \gamma |\nabla _v \mathsf F|^2+\delta \alpha |\nabla _z \mathsf F|^2)(x,v,z)&\le -c (|v|^2+|z|^2+ |x|^{k-1+\beta}) + C.
  \end{align*}
Using in addition \eqref{eq.WF},  \eqref{eq.F3}, and \eqref{eq.Bbeta}, we deduce that 
  $ \mathcal L_{{\rm GL}}\mathsf W_\delta/ \mathsf W_\delta\to -\infty$ as $|x|+|v|+|z|\to +\infty$. This concludes the proof of the proposition.
  \end{proof}


     \subsubsection{The case when $\gamma=0$}
    
    In this section $\gamma=0$,  
   \begin{equation}\label{eq.Bbeta2}
\mathfrak b=\mathfrak a>0,\,  k\in (1,2],\text {and }   \beta=k-1 .
\end{equation}
    Let us check \eqref{eq.inf}. 
 Let $p_1=  k/(k-1)>1$  and $q_1=p_1/(p_1-1)=k\le 2$. 
Using  \textbf{[V{\tiny poly-$x^k$}]}, we have for   all $(x,v,z)\in \mathbf R^d$, if $|x|\ge r_V$,
\begin{align*}
\mathsf F_0(x,v,z) \ge  c_V  {\mathfrak  h}|x|^k+ \frac {\mathfrak  h-\mathfrak a}2 |v|^2+\frac  {\mathfrak  h-\mathfrak a}2 |z|^2- 
\frac{\mathfrak a\kappa }{p_1} |x|^{ k}-\frac{\mathfrak a \kappa }{k} |v|^{k}.
\end{align*}
Then, for $\mathfrak h>0$, choose $\mathfrak  a>0$ small enough such that  
  \begin{equation}\label{eq.abh2}
 \frac{\mathfrak a\kappa }{p_1}< c_V  {\mathfrak  h}  \text{ and } 
 \frac{\mathfrak a\kappa }{k} +\frac{\mathfrak a}{2} <\frac{\mathfrak h}{2}. 
 \end{equation}
 Then  \eqref{eq.inf} holds. Note also that when \eqref{eq.Bbeta2} is satisfied, $\mathsf F_0(x,v,z) \le  c' ( |x|^k+  |v|^2+ |z|^2)  +C'$. 
      \begin{prop}\label{pr.C3-2}
       Assume $\gamma=0$, {\rm\textbf{[V{\tiny poly-$x^k$}]}},  $k\in (1,2]$,  and \eqref{eq.Bbeta2}.       Then, for any $\mathfrak h,\mathfrak a>0$  small enough (these conditions are made explicit in the proof),   \textbf{(C3)} is satisfied with the function $\mathsf W_\delta$ defined in \eqref{eq.Lyapunov2}. 
  \end{prop}
     \begin{proof}  Recall    $\mathfrak b=\mathfrak a$. In the following, $(x,v,z)\in \mathbf R^{3d}$,  $|x|\ge \max (2,r_V)$, and $\eta =\sqrt{\mathfrak a}$.
Using  \textbf{[V{\tiny poly-$x^k$}]}, \eqref{eq.E-l1}, and \eqref {eq.E-l2}, one has:
    \begin{align*}
  & (\mathcal L_{{\rm GL}} \mathsf F+ \delta \gamma |\nabla _v \mathsf F|^2+\delta \alpha |\nabla _z \mathsf F|^2)(x,v,z)\\
  &\le \alpha \mathfrak h d -\lambda \mathfrak a |v|^2  +\mathfrak a v\cdot \text{Jac}(\mathsf L)(x) v-\mathfrak a  \nabla V(x)\cdot \mathsf L(x)-\alpha \mathfrak h |z|^2  \\
   &\quad  + \lambda \mathfrak a |z|^2+\lambda \mathfrak a  |z||\mathsf L(x)|  + \mathfrak a|\nabla V(x)||z| +\alpha \mathfrak a |z|| v| +  \delta \alpha  | \mathfrak hz+\mathfrak a v|^2\\
   &\le  \alpha \mathfrak h d -\lambda \mathfrak a |v|^2  +\mathfrak a \mathfrak C_{\mathsf L}| v|^2-\mathfrak a \kappa c_V |x|^{2(k-1)} -\alpha  \mathfrak h |z|^2  \\
   &\quad  + \lambda \mathfrak a |z|^2+\lambda  \kappa \mathfrak a   |z| |x|^{k-1}  +  \mathfrak a M_V  |x|^{k-1}|z| +\alpha \mathfrak a |z|| v| +  \delta \alpha  | \mathfrak hz+\mathfrak a v|^2\\
   &\le \alpha \mathfrak h d -\lambda \mathfrak a |v|^2  +\mathfrak a \mathfrak C_{\mathsf L}| v|^2-\mathfrak a  \kappa c_V |x|^{2(k-1)} -\alpha \mathfrak h |z|^2 + \lambda \mathfrak a |z|^2+ \lambda \mathfrak a  \kappa  {|z|^2}/{2 \eta}+  \lambda \mathfrak a \eta \kappa  {|x|^{2(k-1)}}/{2}  \\
   &\quad    +  \mathfrak a M_V \eta  {|x|^{2(k-1)}}/2+\mathfrak a M_V  {|z|^2}/{2\eta}  +\alpha \mathfrak a   {|z|^2}/{2\eta } + \alpha \mathfrak a \eta  {|v|^2}/{2}+ 2 \delta \alpha   \mathfrak h^2|z|^2+2 \delta \alpha \mathfrak a^2 |v|^2\\
   &\le  \alpha \mathfrak h d+|v|^2\big[-\lambda \mathfrak a+ \mathfrak a \mathfrak C_{\mathsf L}+2 \delta \alpha \mathfrak a^2+\alpha {\mathfrak a^{3/2}}/2 \big]+ |x|^{2(k-1)}\big [  -\kappa c_V\mathfrak a  + { \lambda \kappa  } \mathfrak a^{3/2}/2+    M_V   \mathfrak a^{3/2}/2\big]\\
   &\quad +  |z|^{2}\big [-\alpha \mathfrak h +2\delta \alpha \mathfrak h^2+ {\lambda\kappa } \sqrt{\mathfrak a } /2+  {M_V  }  \sqrt{\mathfrak a}/2+{\alpha}\sqrt{\mathfrak a}/2+\lambda \mathfrak a\big].
   \end{align*}
   Let $\mathfrak h>0$ such that  $-\alpha \mathfrak h +2\delta \alpha \mathfrak h^2<0$. 
Using also \eqref{eq.lambda2}, it holds $-\lambda \mathfrak a+ \mathfrak a \mathfrak C_{\mathsf L}\le -\mathfrak \lambda \mathfrak a/2$.  
We then choose $\mathfrak a>0$ such that \eqref{eq.abh2} holds,  $-\mathfrak \lambda \mathfrak a/2+2 \delta \alpha \mathfrak a^2+\alpha  {\mathfrak a^{3/2}}/2 <0$, $ -\kappa c_V\mathfrak a  + { \lambda \kappa  }  \mathfrak a^{3/2}/2+    {M_V  } \mathfrak a^{3/2}/2<0$, and $-\alpha \mathfrak h+2\delta \alpha \mathfrak h^2+\frac{\lambda\kappa }{2}  \sqrt{\mathfrak a } + \frac{M_V  }2 \sqrt{\mathfrak a}+\frac{\alpha}2\sqrt{\mathfrak a}+\lambda \mathfrak a<0$. With the same arguments as those used at the end of the proof of Proposition \ref{pr.C3-1} and since $\delta>(2-k)/k$, we obtain that 
  $ \mathcal L_{{\rm GL}}\mathsf W_\delta/ \mathsf W_\delta\to -\infty$ as $|x|+|v|+|z|\to +\infty$. This concludes the proof of the proposition.
     \end{proof}

 \subsection{Proof of Theorem \ref{th.1-NE}}
 \label{sec.pr2}
 To prove  Theorem \ref{th.1-NE}, one first proves   \textbf{(C3)}. This is done with  a slight adaptation of the computations made in Section \ref{sec.CC3} and it indeed turns out that the  Lyapunov functions $\mathsf W_\delta$    constructed   in Section \ref{sec.CC3} (when $\boldsymbol \ell=0$, see more precisely \eqref{eq.Lyapunov2}) still satisfy \textbf{(C3)} with the generator  $\mathcal L=\mathcal L_{{\rm GL}}+\boldsymbol \ell \cdot \nabla _v$ when  $\boldsymbol \ell$ satisfies \textbf{[$\boldsymbol b${\tiny non-gradient}]}. 
Once \textbf{(C3)} is proved, one shows   the conditions  \textbf{(C2)},  \textbf{(C1)}, \textbf{(C4)}, \textbf{(C5)} (in this order) with the same arguments as those used in Section~\ref{sec.2V} (which rely on the approach introduced in Section~\ref{sec.22}) considering the coercive  function $\mathsf W_1$ as an energy of the system    instead of the Hamiltonian $\mathsf H_{{\rm GL}}$ (see Remark~\ref{re.ReG}). The  proof of Theorem \ref{th.1-NE} is then a consequence of~\cite[Theorem~2.2]{guillinqsd}.

\section{Generalized Langevin process with singular potentials: proof of Theorem~\ref{th.2}}
\label{sec.GL-singular}

  In this section, we prove Theorem~\ref{th.2}. We will again use~\cite[Theorem~2.2]{guillinqsd}, and  the strategy thus consists  of  showing that 
 \textbf{(C1)}$\to$\textbf{(C5)} are satisfied for the process \eqref{eq.GLintro} on the state space $\mathscr E=\mathbf O_V\times  \mathbf R^{dN} \times \mathbf R^{dN}$ given by    Proposition \ref{pr.ex-GL-i}.

 
 \subsection{On Assumptions  \textbf{(C1)}, \textbf{(C2)}, \textbf{(C4)}, and \textbf{(C5)} for the generalized Langevin process with singular potentials} 
 \begin{prop}\label{pr.PP-S}
Assume that $V$  satisfies  {\rm \textbf{[V{\tiny coercive}]}}.  Then, the nonkilled  semigroup of the process \eqref{eq.GLintro}  (see Proposition \ref{pr.ex-GL-i}) satisfies \textbf{(C1)} and  \textbf{(C2)}. Let $\mathscr D=\mathscr O \times 
  \mathbf R^{dN} \times \mathbf R^{dN}$ where $\mathscr O$ is a subdomain of $\mathbf O_V$  such that $\mathbf O_V\setminus \overline{\mathscr O}$ is nonempty. Then, the  killed   process \eqref{eq.GLintro}  satisfies \textbf{(C4)} and  \textbf{(C5)}.   
 \end{prop}
 
 Before turning to the proof of Proposition \ref{pr.PP-S},  we briefly  explain why, in the case $
  \gamma>0$,   we have not been able to  adapt the  arguments used in the proof of  Proposition~\ref{pr.SFgamma0}  to  prove~\textbf{(C1)}  in the  singular potential  setting of Proposition \ref{pr.PP-S}.   First notice that  the global Girsanov formula~\eqref{eq.Girsanov}   can be extended  to the case when $V$ is singular by adding the  condition (on   its  right hand side) that  $t<  \sigma^0_{\mathbf O_V}$, where $\sigma^0_{\mathbf O_V} =\inf \{t\ge 0, x_t^0\notin \mathbf O_V\}$ (see \eqref{eq.GL0}). Recall that  $\nabla V$ is very large near   $\partial \mathbf O_V$ and  the problem is that, roughly speaking, there is no reason for the probability of the event $\{$``the process \eqref{eq.GL0} visits  the region  $\partial \mathbf O_V$''$\}$ to be   small. We  will  rather use once again the energy splitting approach introduced in Section~\ref{sec.22} as a starting point.

  \begin{proof} 
We  will only focus on the proof of \textbf{(C1)}. The conditions   \textbf{(C2)}, \textbf{(C4)}, and \textbf{(C5)} are checked   with the same tools  as those used   in the previous section when $V$ satisfies  \textbf{[V{\tiny poly-$x^k$}]}.  To prove~\textbf{(C1)} for any $\gamma\ge 0$, the starting point is  \eqref{eq.splitting}. Equation~\eqref{eq.rho} is then a consequence of  the inequality $\mathbf P_{\mathsf z}[\sigma_{\mathscr H_R}\le t]\le  {e^{ct} }  \,  \mathsf H_{{\rm GL}}(\mathsf z)/R$ derived in the proof of Proposition~\ref{pr.ex-GL-i}.  
When $\gamma=0$, \eqref{eq.theta} is proved with the same arguments as those used in the proof of Theorem~\ref{pr.SFgamma-neq-0}. To this end, we recall that one    begins by introducing  for  $R>0$ the   solution $(\bar X_t^R,t\ge 0)$ of~\eqref{eq.GLR-2} over $\mathbf R^{3dN}$  with $\mathbf b=\mathbf b_R$ and $\gamma=0$, where  
$\mathbf b_R:\mathbf R^{dN}\to \mathbf R^{dN}$  is a globally Lipschitz vector field such that  $\mathbf b_R= -\nabla V$ in a neighborhood of  $\{\mathsf x\in \mathbf R^{dN}, V(\mathsf x)<R\}$. 

 When~$\gamma>0$,   \eqref{eq.theta}  can   also be proved  with the same arguments as those used in the proof  of Theorem~\ref{pr.SFgamma-neq-0} by  extending  the Gaussian upper bound~\cite[Theorem 1.1]{delarue} to the   solution $(\bar X_t^R,t\ge 0)$ of~\eqref{eq.GLR-2} over $\mathbf R^{3dN}$  with $\mathbf b=\mathbf b_R$ and $\gamma>0$. 
One way to avoid such a technical extension is   to actually  use   the  arguments which will be  used to prove Theorem~\ref{th.NH-C} below for the Nosé-Hoover process and which do not rely on a Gaussian upper bound.    In this case, the starting point  is to introduce  for  $R>0$ the  solution $(\hat X^R_t=(\hat x_t^R,\hat v_t^R,\hat z_t^R), t\ge 0)$  over $\mathbf R^{3dN}$  to 
$$d  \hat x^R_t = \hat v^R_t dt, \ d\hat q^R_t=\mathbf F_R(\hat X_t^R)dt+ \boldsymbol  \Sigma \, d\boldsymbol w_t,$$ 
where $\hat q_t^R= (\hat v_t^R,\hat z_t^R)^T\in \mathbf R^{dN}\times  \mathbf R^{dN}$ and where $\mathbf F_R: \mathbf R^{3dN}\to \mathbf R^{2dN}$ is a bounded and globally Lipschitz  vector field   such that  $\mathbf F_R=\mathbf F$ in a neighborhood of   $\mathscr H_R$. The second step consists of using a Girsanov formula  which links   the law of the process $(\hat X^R_t=(\hat x_t^R,\hat v_t^R,\hat z_t^R), t\ge 0)$ before its  first exit time from $\mathscr H_R$ with the law of the    solution  over $\mathbf R^{3dN}$  to 
$$d  \hat x_t = \hat v_t dt, \ d\hat q_t= \boldsymbol  \Sigma \, d\boldsymbol w_t.$$  
Then, with this Girsanov formula at hand, and using the same arguments as those used in the second and third step of the proof of Theorem~\ref{th.NH-C}, one deduces \eqref{eq.theta} when $\gamma>0$. 
  \end{proof}


%
  \subsection{On Assumption \textbf{(C3)} for the generalized Langevin process with singular potentials} 
 In all this section, we assume that \textbf{[V{\tiny sing1}]} holds.
 We  will also assume without loss of generality that $V_{\mathbf p}\equiv 0$ (see \eqref{eq.V-i}), up to restricting the following computations to $\mathsf x=(x,v,z)\in \mathscr E$ with  $x\in \mathbf O_V \cap  {\rm supp}(V_{\mathbf p})^c$. 
 
 \subsubsection{Preliminary computations}
     We start by constructing two Lyapunov functions $\mathsf W_\delta:\mathscr E\to [1,+\infty)$ 
 satisfying Assumption   \textbf{(C3)}, where $0<\delta\leq 1$. 
   Recall that $\mathsf H_{{\rm GL}}(x,v,z)= V(x)+\frac 12 |v|^2+ \frac 12 |z|^2$, for $\mathsf x=(x,v,z)\in \mathscr E$ (see \eqref{eq.H-GL}). Let  $\mathfrak c\geq 0$, and $\mathfrak h\ge 0$, $\mathfrak b\ge 0$, $\mathfrak R>0$.  Consider the  function constructed in \cite{duong2023asymptotic}:
\begin{align}
\label{eq.F0-GL-I}
\mathsf F_0(x,v,z)&=\mathfrak h\mathsf H_{{\rm GL}}(x,v,z)+\mathfrak b \mathfrak R \, x\cdot v+\mathfrak c \mathfrak R^2 v\cdot z \quad -\mathfrak b  \mathsf J_\mathfrak R(x,v,z) v\cdot \mathsf G(x),
\end{align}
 where $\mathsf J_\mathfrak R:\mathscr E \to \mathbf R_+^*$ is a sufficiently smooth  function which depends on $\mathfrak R$ (so that $\mathcal L_{{\rm GL}}\mathsf F$ below is well defined) which  will be specified later (depending on the case $\gamma>0$ or $\gamma=0$), and where $\mathsf G: \mathbf O_V\to \mathbf R^{dN}$ was introduced in~\cite{lu2019geometric} and is defined by: 
 $$ \mathsf G(x)=(\mathsf G^1(x), \ldots, \mathsf G^N(x))^T, \text{ with } \  \mathsf G^i(x)= \sum_{j=1, j\neq i}^N \frac{x^i-x^j}{|x^i-x^j|}, i\in \{1,\ldots, N\}.$$ 
 The parameters  will be chosen in particular such that 
 \begin{equation}\label{eq.finite}
\inf_\mathscr E \mathsf F_0\in\mathbf{R}.
\end{equation}
 We then set $\mathsf F_{{\rm GL}}=\mathsf F_0-\inf_\mathscr E \mathsf F_0+1$ and we define over $\mathscr E$:
 \begin{equation}\label{eq.lyapunovint}
\mathsf W_\delta=\exp \big [\mathsf F_{{\rm GL}}^\delta \big ], \ \text{ where } 0<\delta\leq 1.
\end{equation}
For any fixed $\mathfrak R>0$, the function   $\mathsf J_\mathfrak R$ will satisfy $\mathsf J_\mathfrak R^2\le C \mathsf H_{\text{GL}}$ over $\mathscr E$, for some $C>0$. Consequently,  for any fixed parameters $\mathfrak h, \mathfrak b,\mathfrak R>0$, there exists $c>0$, $\mathsf F_0\le c\mathsf H_{\text{GL}}$ over $\mathscr E$, which implies that for all $\delta \in (0,1]$ (recall that $\mathsf H_{\text{GL}}\ge 1$), 
\begin{equation}\label{eq.Up-GL-i1}
\mathsf W_\delta \le e^{ c^\delta \mathsf H_{{\rm GL}} ^\delta} \text{ on } \mathscr E.
\end{equation}
In the following, and for ease of notation, we again simply write $\mathsf F$ for $\mathsf F_{{\rm GL}}$.

Recall that   $\mathcal L_{{\rm GL}}\mathsf F= v\cdot \nabla _x \mathsf F + (-\gamma v -  \nabla _x V +\lambda z)\cdot \nabla_v\mathsf F + \gamma \Delta_v\mathsf F- (\alpha z+\lambda v)\cdot \nabla _z\mathsf F+ \alpha \Delta_z \mathsf F$ and  also that  over $\mathscr E$, one has:
\begin{equation}\label{eq.WF-i}
\frac{\mathcal L_{{\rm GL}}\mathsf W_\delta}{ \mathsf W_\delta}\le \frac{\delta }{\mathsf F^{1-\delta}}\Big[\mathcal L_{{\rm GL}} \mathsf F+ \delta \gamma |\nabla _v \mathsf F|^2+\delta \alpha |\nabla _z \mathsf F|^2\Big].
\end{equation} 
On the other hand it holds on  $\mathscr E$:
$$
  \left\{
    \begin{array}{ll}
      \partial_{x^i} \mathsf F &= \mathfrak h\partial_{x^i} V  + \mathfrak b\mathfrak R v^i -\mathfrak b\,   \mathsf J_\mathfrak R  \, \partial_{x^i} (v\cdot \mathsf G)  -\mathfrak b  \,  v\cdot \mathsf G\,  \partial_{x^i} \mathsf J_\mathfrak R , \\
           \partial_{v^i } \mathsf F &= 
 \mathfrak hv^i  + \mathfrak b \mathfrak Rx^i+\mathfrak c \mathfrak R^2 z^i-\mathfrak b\mathsf J_\mathfrak R \mathsf G^i   
  -\mathfrak b (v\cdot \mathsf G ) \partial_{v^i }\mathsf J_\mathfrak R , \\
  \Delta_v \mathsf F&=  \mathfrak hNd-2\mathfrak b \nabla _v \mathsf J_{\mathfrak R} \cdot \mathsf G  -\mathfrak b (v\cdot \mathsf G ) \Delta_v \mathsf J_{\mathfrak R}, \\
        \partial_{z^i} \mathsf F &= \mathfrak hz^i+\mathfrak c \mathfrak R^2 v^i -\mathfrak b (v\cdot \mathsf G  )\partial_{z^i}\mathsf J_\mathfrak R,\\
        \Delta_z \mathsf F&= \mathfrak hNd -\mathfrak b (v\cdot \mathsf G ) \Delta_z \mathsf J_{\mathfrak R},
          \end{array}
\right.
$$
where $$\partial_{x^i} (v\cdot \mathsf G ) =\sum_{j=1; j\neq i}^N\Big[\frac{v^i -v^j}{|x^i-x^j|}-\frac{(v^i -v^j)\cdot(x^i-x^j)}{|x^i-x^j|^3}(x^i-x^j)\Big].$$

Thus, one has for all $(x,v,z)\in \mathscr E$:
\begin{align}
\nonumber
\mathcal L_{{\rm GL}}\mathsf F&=    \mathfrak h Nd(\gamma  + \alpha)+  (-\alpha \mathfrak h+\mathfrak c \mathfrak R^2 \lambda )|z|^2+  (-\gamma \mathfrak h-\lambda \mathfrak c \mathfrak R^2+ \mathfrak b \mathfrak R)|v|^2- \mathfrak b \mathfrak R\nabla _xV\cdot x\\
\nonumber
&\quad - \gamma \mathfrak b \mathfrak R\,  x\cdot v -\gamma \mathfrak c \mathfrak R^2 v\cdot z-\alpha \mathfrak c\mathfrak R^2 v\cdot z + \lambda \mathfrak b \mathfrak  R\,  x\cdot z -   \mathfrak c \mathfrak R^2  \, z\cdot \nabla _xV\\
\nonumber
&\quad - \mathfrak b (v\cdot \mathsf G) \, v\cdot \nabla_x \mathsf J_{\mathfrak R}  + \alpha  \mathfrak b (v\cdot \mathsf G) z\cdot \nabla_z \mathsf J_{\mathfrak R} + \lambda \mathfrak b (v\cdot \mathsf G) v\cdot \nabla_z \mathsf J_{\mathfrak R}  \\
\nonumber
&\quad -2 \gamma \mathfrak b\mathsf G \cdot \nabla _v\mathsf J_{\mathfrak R}     -\lambda \mathfrak b (v\cdot \mathsf G) \, z\cdot \nabla_v \mathsf J_{\mathfrak R} + \gamma \mathfrak b(v\cdot \mathsf G) \, v\cdot \nabla_v \mathsf J_{\mathfrak R} + \mathfrak b (v\cdot \mathsf G) \, \nabla_x V\cdot \nabla_v \mathsf J_{\mathfrak R}    \\
\nonumber
&\quad -\gamma \mathfrak b (v\cdot \mathsf G) \Delta_v \mathsf J_{\mathfrak R}- \alpha \mathfrak b (v\cdot \mathsf G) \Delta_z \mathsf J_{\mathfrak R}\\
\label{eq.L-i_GL}
&\quad + \gamma \mathfrak b (v\cdot \mathsf G) \mathsf J_{\mathfrak R} +  \mathfrak b  \nabla_x V\cdot \mathsf G \, \mathsf J_{\mathfrak R} -\lambda \mathfrak b z\cdot \mathsf G \, \mathsf J_{\mathfrak R} - \mathfrak b \mathsf J_{\mathfrak R} v\cdot \nabla_x(v\cdot \mathsf G ).
 \end{align}
 We conclude this section with some estimates which be used later. 
 First notice that  
 for all $\epsilon>0$, there exist $C_{\mathbf {I}},c_{\mathbf {I}}>0$, for all $y\in   \mathbf R^d\setminus \{0\}$,
\begin{equation}\label{eq.LI}
  - c_{\mathbf {I}}+   \frac{\mathsf B}{1+\epsilon}  \vert y\vert^{-\beta} \le \mathsf  V_{\mathbf{I}}(y)\le   C_{\mathbf {I}}+ (1+\epsilon)\mathsf B \vert y\vert^{-\beta},
  \end{equation}
  so that for $x\in  \mathbf O_V$:
  $$ V(x)\ge  {a_0}|x|^2/2+\frac{\mathsf B}{2}\sum_{i,j=1;i<j}^N \vert x^i-x^j\vert^{-\beta} -C.$$ 
 We also recall that from~\cite[Section 4]{lu2019geometric} (see the computation of $\mathbf p \cdot \nabla_{\mathbf q}\Psi$ there), $v\cdot \nabla_x(v\cdot \mathsf G)\ge 0$ over $\mathscr E$. Therefore, since $\mathfrak b\ge 0$ and $\mathsf J_{\mathfrak R}\ge 0$, the last term in \eqref{eq.L-i_GL} is nonpositive, i.e.
\begin{equation}\label{eq.b<==}
 - \mathfrak b \mathsf J_{\mathfrak R} v\cdot \nabla_x(v\cdot \mathsf G ) \le 0 \ \text{ over $\mathscr E$}.
 \end{equation}

 \subsubsection{The case when $\gamma>0$}
In this section, $\gamma>0$,
\begin{equation}\label{eq.Int1}
 \mathfrak h>\max(\mathfrak b,  {\mathfrak b}/{a_0}), \, \mathfrak c=0, \, \mathfrak R=1, \, \text{and}\ \mathsf J_\mathfrak R\equiv 1.
\end{equation}
We mention that extra conditions will be assumed on $\mathfrak h,\mathfrak b>0$ in the proof of Proposition~\ref{pr.Cint-1}.
From \eqref{eq.F0-GL-I}, it holds in particular $\mathsf F_0=\mathfrak h\mathsf H_{{\rm GL}}+\mathfrak b \mathfrak R \, x\cdot v -\mathfrak b  v\cdot \mathsf G$. 
Let us  first check \eqref{eq.finite}. 
Using \eqref{eq.LI}, it holds  over $ \mathscr E$,
$$\mathsf F_0 \ge  \frac{a_0\mathfrak h-\mathfrak b}{2} |x|^2+  \frac{\mathsf B\mathfrak  h}{2}\sum_{i,j=1;i<j}^N \frac{1}{|x^i-x^j|^\beta}+\frac{\mathfrak  h-\mathfrak  b}{2}|v|^2-\mathfrak b(N-1)\sum_{i=1}^N|v^i |+ \frac  {\mathfrak  h}2 |z|^2-C.$$
Therefore, in view of \eqref{eq.Int1}, the condition \eqref{eq.finite} is satisfied when \textbf{[V{\tiny sing1}]} holds.

\begin{prop}\label{pr.Cint-1}
Assume $\gamma>0$, {\rm \textbf{[V{\tiny sing1}]}}, and \eqref{eq.Int1}. Then, under additional conditions on the parameters  $\mathfrak h,\mathfrak b>0$ (these conditions are made explicit in the proof below),   Assumption \textbf{(C3)} is satisfied with the function $\mathsf W_\delta$ defined in \eqref{eq.lyapunovint} and which satisfies the upper bound~\eqref{eq.Up-GL-i1}. 
\end{prop}
\begin{proof}
Assume $\gamma>0$. The goal is to show that one can find parameters such that $ {\mathcal L_{{\rm GL}}\mathsf W_\delta}/{ \mathsf W_\delta} \to -\infty$ as  the energy $\mathsf H_{\text{GL}}(\mathsf x)\to +\infty$. 
We start by providing an upper bound on $\mathcal L_{{\rm GL}}\mathsf F$ (see \eqref{eq.WF-i}), adapting  essentially   the same  computations as in~\cite[Lemma 3.5]{duong2023asymptotic} (see also \cite{lu2019geometric}). 
 By \eqref{eq.L-i_GL}, \eqref{eq.b<==}, and \eqref{eq.Int1}, one has for all $\mathsf x=(x,v,z)\in \mathscr E$,
 \begin{align}
\nonumber
\mathcal L_{{\rm GL}}\mathsf F&\le     \mathfrak h Nd(\gamma  + \alpha) -\alpha \mathfrak h |z|^2- |v|^2(\gamma \mathfrak h - \mathfrak b )- \mathfrak b  \nabla _xV\cdot x\\
\nonumber
&\quad - \gamma \mathfrak b \,  x\cdot v  + \lambda \mathfrak b \,  x\cdot z + \gamma \mathfrak b (v\cdot \mathsf G)  +  \mathfrak b  \nabla_x V\cdot \mathsf G   -\lambda \mathfrak b z\cdot \mathsf G.
 \end{align}
 Because $C^\infty_{\mathsf G}:=\sup_{\mathscr E}|\mathsf G|<+\infty$, we deduce that 
 \begin{align}
\nonumber
\mathcal L_{{\rm GL}}\mathsf F&\le     \mathfrak h Nd(\gamma  + \alpha)-\alpha \mathfrak h |z|^2- |v|^2(\gamma \mathfrak h - \mathfrak b )- \mathfrak b  \nabla _xV\cdot (x-\mathsf G)\\
\nonumber
&\quad - \gamma \mathfrak b \,  x\cdot v  + \lambda \mathfrak b \,  x\cdot z + \gamma  \mathfrak b  C^\infty_{\mathsf G} |v|   +C^\infty_{\mathsf G} \lambda \mathfrak b |z| .
 \end{align}
 Let us now deal with the term  $\mathfrak b  \nabla _xV\cdot (x-\mathsf G)$. Recall that since $V_{\mathbf I}(x)=V_{\mathbf I}(-x)$ (see also \eqref{eq.V-i}), 
$$V(x)=\sum_{i=1}^N V_{\mathbf c}(x^i)+ \frac 12\sum_{i,j=1;  i\neq j}^N V_{\mathbf I}(x^i-x^j).$$ 
Using also that for any symmetric vector field  
 $\mathsf v:\mathbf R^d\to \mathbf R^d$, 
 $$\sum_{i,j=1; i\neq j}^N \mathsf v(x^i-x^j)\cdot x^i=\frac  12\sum_{i,j=1; i\neq j}^N \mathsf v(x^i-x^j)\cdot (x_i-x_j),$$  we deduce that  for all $x\in \mathbf O_V$:
 \begin{align}
 \nonumber
  \nabla _xV\cdot (x-\mathsf G(x))&=\sum_{i=1}^N\Big[ a_0x^i+\frac 12 \sum_{j=1;j\neq i}^N \nabla V_{\mathbf I}(x^i-x^j)\Big]\cdot\Big[  x^i- \sum_{j=1;j\neq i}^N \frac{x^i-x^j}{|x^i-x^j|}\Big]\\
\nonumber
  &=a_0|x|^2-\frac{a_0}2\sum_{i,j=1; i\neq j} ^N|x^i-x^j|+ \frac 14\sum_{i,j=1; i\neq j} ^N\nabla V_{\mathbf I}(x^i-x^j)\cdot(x^i-x^j)\\
  \label{eq.nablax-x}
  &\quad +\frac 12\sum_{i=1}^N \Big[\sum_{j=1;j\neq i}^N (-\nabla V_{\mathbf I})(x^i-x^j)\Big]\cdot \Big[\sum_{k=1;k\neq i}^N \frac{x^i-x^k}{|x^i-x^k|}\Big]. 
 \end{align}
 Denote by $I_{\mathbf I}(x)$ the last term in the previous equality. 
 Note that   $|\nabla V_{\mathbf I}(y)\cdot y|\le C_{\mathbf I}/|y|^{\beta}+ c_{\mathbf I}$ (for some $C_{\mathbf {I}},c_{\mathbf {I}}>0$).
 On the other hand, with the same computations as those made to prove~\cite[Lemma~4.2]{lu2019geometric}, for any $q>1$,
 $$ \sum_{i=1}^N \Big[\sum_{j=1;j\neq i}^N \frac{x^i-x^j}{|x^i-x^j|^{q}}\Big]\cdot \Big[\sum_{k=1;k\neq i}^N \frac{x^i-x^k}{|x^i-x^k|}\Big]\ge \mathfrak K_{q-1}(x):=\sum_{i,j=1:i\neq j}^N \frac{1}{|x^i-x^j|^{q-1}}.$$
 Therefore   $I_{\mathbf I}(x)\ge  \frac{ \mathsf B \beta}2\mathfrak K_{\beta+1}(x)-c (\mathfrak K_{q_\Phi}(x)+1)$ (for some $c>0$).
We set $$  \mathsf q^*=\beta+1, \, \mathsf c^*=\frac{ \mathsf B \beta}2.$$ 
 Therefore, using all the previous estimates and the inequality $ \kappa \mathfrak b |a\cdot b|\le \mathfrak b^{3/2} a^2+ \kappa^2\mathfrak b^{1/2} b^2/4$, we have   for all $\mathsf x=(x,v,z)\in \mathscr E$: 
\begin{align*} 
\mathcal L_{{\rm GL}}\mathsf F(\mathsf x)&\le C - (a_0\mathfrak b  -2\mathfrak b^{3/2} )|x|^2  - \mathfrak b \mathsf c^* \mathfrak K_{\mathsf q^*}(x) -(\alpha \mathfrak h - \lambda^2 \mathfrak b^{1/2}/4) |z|^2\\
&\quad - |v|^2(\gamma \mathfrak h- \mathfrak b -\gamma^2 \mathfrak b^{1/2}/4 )+ \mathfrak L( \mathsf x),
\end{align*} 
 where   $\mathfrak L\in \mathcal C^0( \mathscr E)$ and $ \mathfrak L( \mathsf x)= o( |\mathsf  x|^2 +\mathfrak K_{\mathsf q^*}(x))$   when  $\mathsf  x\to \{\infty\}$ or $x\to   \partial \mathbf O_V$, 
 i.e. when $\mathsf H_{\text{GL}}(\mathsf x)\to +\infty$. 
In addition, over $\mathscr E$, 
 $$\delta \gamma |\nabla _v \mathsf F|^2+\delta \alpha |\nabla _z \mathsf F|^2\le 3\delta\gamma \mathfrak h^2|v|^2+3\delta\gamma\mathfrak b^2|x|^2+\delta\alpha\mathfrak h^2|z|^2+C.$$ 
 This implies that $\mathcal L_{{\rm GL}}\mathsf F(\mathsf x)+\delta \gamma |\nabla _v \mathsf F|^2+\delta \alpha |\nabla _z \mathsf F|^2\le C - (a_0\mathfrak b  -2\mathfrak b^{3/2}-3\delta\gamma\mathfrak b^2 )|x|^2  - \mathfrak b \mathsf c^* \mathfrak K_{\mathsf q^*}(x) -(\alpha \mathfrak h -\delta\alpha\mathfrak h^2- \lambda^2 \mathfrak b^{1/2}/4) |z|^2
  - |v|^2(\gamma \mathfrak h -\mathfrak b -3\delta\gamma \mathfrak h^2-\gamma^2 \mathfrak b^{1/2}/4 )+ \mathfrak L(\mathsf  x)$. Choose $\mathfrak h>0$ small enough such that $\gamma \mathfrak h  -3\delta\gamma \mathfrak h^2>0$ and 
 $\alpha \mathfrak h -\delta\alpha\mathfrak h^2>0$. Then, in addition to $\mathfrak h>\max(\mathfrak b,  {\mathfrak b}/{a_0})$, choose $\mathfrak b>0$ small enough such that $a_0\mathfrak b  -2\mathfrak b^{3/2}-3\delta\gamma\mathfrak b^2  >0$, $\alpha \mathfrak h -\delta\alpha\mathfrak h^2-  \lambda^2 \mathfrak b^{1/2}/4>0$,   and $\gamma \mathfrak h -\mathfrak b -3\delta\gamma \mathfrak h^2-\gamma^2 \mathfrak b^{1/2}/4>0 $. Then, one has  over $\mathscr E$,
 $$\mathcal L_{{\rm GL}}\mathsf F(\mathsf x)+\delta \gamma |\nabla _v \mathsf F|^2+\delta \alpha |\nabla _z \mathsf F|^2\le C -c| \mathsf x|^2 - \mathfrak b \mathsf c^*\mathfrak K_{\mathsf q^*}(x) + \mathfrak L( \mathsf  x).$$ 
 
On the other hand,     $\mathsf F \le C(|\mathsf x|^2+ \mathfrak K_{\beta}(x)+1)$ over $\mathscr E$.
Since $0\le 1-\delta<1$, by \eqref{eq.WF-i}, we deduce that $ {\mathcal L_{{\rm GL}}\mathsf W_\delta}/{ \mathsf W_\delta} \to -\infty$ as    $\mathsf  x\to \{\infty\}$ or $x\to   \partial \mathbf O_V$. This concludes the proof of the proposition. 
 \end{proof}
 
     
 \subsubsection{The case when $\gamma=0$}
 In this section, $\gamma=0$, 
\begin{equation}\label{eq.Int3}
 \mathfrak R>0, \ \mathfrak c=\mathfrak b, \ \, \mathsf J_\mathfrak R^2(\mathsf x)=\mathfrak R^6|z|^2+|v|^2+2  V(x)+\mathfrak R^2, \ \text{ for }  \mathsf x=(x,v,z)\in \mathscr E.
\end{equation} 
The function $\mathsf J_\mathfrak R $ above was introduced  in~\cite{duong2023asymptotic} when $\gamma=0$  to cancel  the effect caused by the term $|z|/|x|^{\beta+1}$  which comes from the term  $z\cdot \nabla V_{\mathbf I}$ in the computations of $\mathcal L_{{\rm GL}}(v\cdot z)$, see at the end of the second line of \eqref{eq.L-i_GL} (recall that this term does not exist when $\gamma>0$ because we chose $\mathfrak c=0$ in this case).   

We start by checking \eqref{eq.finite}. Using that $\sqrt{a+b}\le \sqrt a+\sqrt b$ ($a,b\ge 0$), one has:
\begin{align*}
\mathsf F_0&\ge \big[\mathfrak h -\frac{\mathfrak bC^\infty_{\mathsf G}}{\sqrt 2}\big]V(x)- \frac{\mathfrak b\mathfrak R}{2}|x|^2+\big[\frac{\mathfrak h}{2} -\frac{\mathfrak b\mathfrak R}{2}-\frac{\mathfrak b\mathfrak R^2}{2}-\frac{\mathfrak bC^\infty_{\mathsf G}}{2}(\mathfrak R^3+\sqrt 2 +2)\big]|v|^2\\
&\quad- \mathfrak b\mathfrak RC^\infty_{\mathsf G} |v| + \big[\frac{\mathfrak h}{2}  - \frac{\mathfrak b\mathfrak R^2}{2}-\frac{\mathfrak b\mathfrak R^3C^\infty_{\mathsf G}}{2} \big]|z|^2.
\end{align*}
We have 
$$\big[\mathfrak h -\frac{\mathfrak bC^\infty_{\mathsf G}}{\sqrt 2}\big]V(x)- \frac{\mathfrak b\mathfrak R}{2}|x|^2= \frac{a_0}{2}\big[\mathfrak h -\frac{\mathfrak bC^\infty_{\mathsf G}}{\sqrt 2}\big]|x|^2- \frac{\mathfrak b\mathfrak R}{2}|x|^2 + \big[\mathfrak h -\frac{\mathfrak bC^\infty_{\mathsf G}}{\sqrt 2}\big] \sum_{i,j=1;i<j}^NV_{\mathbf I}(x^i-x^j).$$ 
Then, \eqref{eq.finite} holds if $\mathfrak h,\mathfrak b, \mathfrak R>0$ satisfy:
\begin{equation} \label{eq.cond-h1}
\frac{a_0}{2}\big[\mathfrak h -\frac{\mathfrak bC^\infty_{\mathsf G}}{\sqrt 2}\big] - \frac{\mathfrak b\mathfrak R}{2}>0,
\end{equation}
and
\begin{equation} \label{eq.cond-h1-2}
 \frac{\mathfrak h}{2}  - \frac{\mathfrak b\mathfrak R^2}{2}-\frac{\mathfrak b\mathfrak R^3C^\infty_{\mathsf G}}{2}>0, \ \frac{\mathfrak h}{2} -\frac{\mathfrak b\mathfrak R}{2}-\frac{\mathfrak b\mathfrak R^2}{2}-\frac{\mathfrak bC^\infty_{\mathsf G}}{2}(\mathfrak R^3+\sqrt 2 +2)>0.
\end{equation}
  
\begin{prop}\label{pr.Cint-2}
Assume $\gamma=0$ and {\rm \textbf{[V{\tiny sing1}]}}. Then, for some appropriately chosen $\mathfrak h,\mathfrak b, \mathfrak R>0$ (these conditions are made explicit in the proof),   \textbf{(C3)} is satisfied with the function $\mathsf W_\delta$ defined in \eqref{eq.lyapunovint}. 
\end{prop}


\begin{proof}
We start by providing an upper bound on  the term $\mathcal L_{{\rm GL}}\mathsf F$ in \eqref{eq.WF-i}, for which we need to track the involved constants to get explicit conditions on the parameters $\mathfrak h, \mathfrak b, \mathfrak R>0$ (the computations are at some places similar  to  those made in the proof of \cite[Lemma 3.6]{duong2023asymptotic}). The following identities will be needed:
$$\nabla _{x} \mathsf J_\mathfrak R=  \frac{\nabla _xV}{ \mathsf J_\mathfrak R}, \, 
      \nabla _{v} \mathsf J_\mathfrak R =  \frac{v}{\mathsf J_\mathfrak R}, 
      \nabla _{z} \mathsf J_\mathfrak R =  \frac{\mathfrak R^6\, z}{ \mathsf J_\mathfrak R}, \text{and } \Delta _{z} \mathsf J_\mathfrak R= \frac{\mathfrak R^6 dN}{ \mathsf J_\mathfrak R}-  \frac{\mathfrak R^{12}|z|^2}{\mathsf J_\mathfrak R^3}.$$
     By \eqref{eq.L-i_GL} and \eqref{eq.b<==}, when $\gamma=0$ and $\mathfrak c=\mathfrak b$, we have   for all $\mathsf x=(x,v,z)\in \mathscr E$,
\begin{align}
\nonumber
\mathcal L_{{\rm GL}}\mathsf F&\le     \mathfrak h Nd  \alpha+  (-\alpha \mathfrak h+\mathfrak b \mathfrak R^2 \lambda )|z|^2+ (-\lambda \mathfrak b \mathfrak R^2+\mathfrak b \mathfrak R) |v|^2- \mathfrak b \mathfrak R\nabla _xV\cdot x\\
\nonumber
&\quad  -\alpha \mathfrak b\mathfrak R^2 v\cdot z + \lambda \mathfrak b \mathfrak  R\,  x\cdot z -\mathfrak R^2 \mathfrak b z\cdot \nabla _xV\\
\nonumber
&\quad - \mathfrak b (v\cdot \mathsf G) \, v\cdot \nabla_x \mathsf J_{\mathfrak R}  + \alpha  \mathfrak b (v\cdot \mathsf G) z\cdot \nabla_z \mathsf J_{\mathfrak R} + \lambda \mathfrak b (v\cdot \mathsf G) v\cdot \nabla_z \mathsf J_{\mathfrak R}  \\
\nonumber
&\quad   -\lambda \mathfrak b (v\cdot \mathsf G) \, z\cdot \nabla_v \mathsf J_{\mathfrak R} + \mathfrak b (v\cdot \mathsf G) \, \nabla_x V\cdot \nabla_v \mathsf J_{\mathfrak R}     - \alpha \mathfrak b (v\cdot \mathsf G) \Delta_z \mathsf J_{\mathfrak R}\\
\label{eq.L-i_GL-gamma0}
&\quad   +  \mathfrak b  \nabla_x V\cdot \mathsf G \, \mathsf J_{\mathfrak R} -\lambda \mathfrak b z\cdot \mathsf G \, \mathsf J_{\mathfrak R} .
 \end{align}
We now provide upper bounds on the terms appearing in  \eqref{eq.L-i_GL-gamma0} over $\mathscr E$ as follows: 
 \begin{itemize}
 \item Following the computations in \eqref{eq.nablax-x}, we have 
\begin{align*}
 - \mathfrak b \mathfrak R\nabla _xV\cdot x&=- a_0 \mathfrak b \mathfrak R|x|^2-\frac{ \mathfrak b \mathfrak R}4 \sum_{i,j=1; i\neq j} ^N\nabla V_{\mathbf I}(x^i-x^j)\cdot(x^i-x^j)\\
&\le - a_0 \mathfrak b \mathfrak R |x|^2 + \mathfrak L( x)
\end{align*}
 and  $ \nabla_x V\cdot \mathsf G \, \mathsf J_{\mathfrak R}\le -\big[\mathsf c^*\mathfrak K_{\mathsf q^*}(x)-\mathfrak L( x) \big]\mathsf J_{\mathfrak R}$, 
 where $\mathfrak L\in \mathcal C^0( \mathbf O_V)$ and $\mathfrak L(   x)= o( |x|^2 +\mathfrak K_{\mathsf q^*}(x))$   when  $ x\to \{\infty\}\cup  \partial \mathbf O_V$. 
 In particular, it holds over $\mathscr E$:
 $$\mathsf c^*\mathfrak K_{\mathsf q^*}(x)-\mathfrak L( x) \ge \frac{\mathsf c^*}{2}\mathfrak K_{\mathsf q^*}(x)-C.$$ 
 In addition, since $\mathsf J_\mathfrak R\ge \mathfrak R \mathsf f (z)$ (where $\mathsf f(z)=\sqrt{\mathfrak R^{4}|z|^2+1}$), one deduces that: 
  $$\mathfrak b  \nabla_x V\cdot \mathsf G \, \mathsf J_{\mathfrak R}\le -  \frac{\mathsf c^*\mathfrak b \mathfrak R}{2}\mathfrak K_{\mathsf q^*}(x) \mathsf f(z)  +C\mathfrak b \mathfrak R \mathsf f(z).$$
Roughly speaking, the term $\mathfrak b  \nabla_x V\cdot \mathsf G \, \mathsf J_{\mathfrak R}$ can absorb the bad one $- \mathfrak R^2 \mathfrak b z\cdot \nabla _xV$ when $x\to \partial \mathbf O_V$.  
Note that there exists $C>0$ such that $|\nabla V_{\mathbf I}(y)|\le  \mathsf A /|y|^{\mathsf q^*} +C$ over $\mathbf R^d\setminus \{0\}$ ($\mathsf A=\mathsf B \beta+1$). 
Then,  one has 
\begin{align*}
|\mathfrak R^2\mathfrak b z\cdot \nabla _xV|&\le  
\mathfrak R^2\mathfrak ba_0|z||x|+ \mathsf A \mathfrak R^2 \mathfrak b |z|\mathfrak K_{\mathsf q^*}(x)+C\mathfrak R^2 \mathfrak b |z|\\
&\le   
\mathfrak R^2 \mathfrak ba_0|z||x|+  \mathsf A\mathfrak b \mathsf f(z)\mathfrak K_{\mathsf q^*}(x)+C\mathfrak R^2 \mathfrak b |z|,
\end{align*}
 so that the function $\mathfrak b  \nabla_x V\cdot \mathsf G \, \mathsf J_{\mathfrak R}- \mathfrak R^2 \mathfrak b z\cdot \nabla _xV$ is bounded from above over $\mathscr E$ by 
\begin{align*}
\ \ \ \  \ \ \ \ -  \mathfrak b \Big[ \frac{\mathsf c^* \mathfrak R}{2}-  \mathsf A \Big]\mathfrak K_{\mathsf q^*}(x)\mathsf f(z)    +C\mathfrak b \mathfrak R \mathsf f(z) + \mathfrak R^2 a_0\frac{\mathfrak b^{1/2}}{2}|z|^2+\mathfrak R^2 a_0\frac{\mathfrak b^{3/2}}{2} |x|^2.
\end{align*}
Note also that $ \mathsf f(z)=o( |\mathsf x|^2 +\mathfrak K_{\mathsf q^*}(x))$ when $\mathsf x\to \{\infty\}$ or $x\to \partial \mathbf O_V$ (i.e. when $\mathsf H_{\text{GL}}(\mathsf x)\to +\infty$). 
  \item For any  $\epsilon_0,\epsilon_1>0$ to be chosen later,  
$$\alpha \mathfrak b\mathfrak R^2 |v\cdot z| + \lambda \mathfrak b \mathfrak  R\,  |x\cdot z|\le \alpha \mathfrak b\mathfrak R^2\Big [\frac{|v|^2}{2\epsilon_0}+ \frac{\epsilon_0}2|z|^2\Big ]+ \lambda \mathfrak b \mathfrak  R\Big [\frac{|x|^2}{2\epsilon_1}+ \frac{\epsilon_1}2|z|^2\Big ].$$   
  \item Using that $\mathsf G$ is bounded, we have:
$$
|   (v\cdot \mathsf G) \, v\cdot \nabla_x \mathsf J_{\mathfrak R}|\le  C^\infty_\mathsf G   \frac{|v|^2(a_0|x|+C)}{\sqrt {|v|^2+2V(x)+\mathfrak R^2}}+ C^\infty_\mathsf G \frac{C|v|^2\mathfrak K_{\beta+1}(x)}{\sqrt {|v|^2+ 2V(x)+\mathfrak R^2}},$$
where we recall that over $\mathbf O_V$, $V(x)\ge  a_0|x|^2/2+\mathsf B\mathfrak K_{\beta}(x)/2-C$. 
Hence,
$$|   (v\cdot \mathsf G) \, v\cdot \nabla_x \mathsf J_{\mathfrak R}| = o( |\mathsf x|^2+ \mathfrak K_{\mathsf q^*}(x)) \text{ as }  \mathsf H_{\text{GL}}(\mathsf x)\to +\infty.$$
  With similar arguments, one also has:  
  \begin{align*}
  &|(v\cdot \mathsf G) z\cdot \nabla_z \mathsf J_{\mathfrak R}|+|(v\cdot \mathsf G) v\cdot \nabla_z \mathsf J_{\mathfrak R}| +|(v\cdot \mathsf G) z\cdot \nabla_v \mathsf J_{\mathfrak R}|+|(v\cdot \mathsf G) \, \nabla_x V\cdot \nabla_v \mathsf J_{\mathfrak R} | \\
  &\quad + |\mathfrak b (v\cdot \mathsf G) \Delta_z \mathsf J_{\mathfrak R}| = o( |\mathsf x|^2+ \mathfrak K_{\mathsf q^*}(x)) \text{ as } \mathsf H_{\text{GL}}(\mathsf x)\to +\infty. 
     \end{align*}

   \item Using that $ \mathsf J_{\mathfrak R}\le \mathfrak R^3|z|+|v|+\sqrt 2\, \sqrt  V(x)+\mathfrak R$ and $V(x)\le a_0|x|^2/2+ S_{\mathbf{int}}(x)+C$ over $\mathbf O_V$ (where $S_{\mathbf{int}}= 2\mathfrak K_{\beta}$), one has:
  \begin{align*}
\ \ \ \ \ \ \ \ \  \lambda \mathfrak b \Big|z\cdot \frac{\mathsf G}{C^\infty_\mathsf G } \, \mathsf J_{\mathfrak R} \Big|&\le \lambda \mathfrak b \mathfrak R^3|z|^2+ \lambda \mathfrak b |z||v|+ \lambda \mathfrak b \sqrt 2\, |z|\sqrt{V}(x)+\lambda \mathfrak b \mathfrak R|z|\\
  &\le   |z|^2\Big[\lambda \mathfrak b \mathfrak R^3+ \frac{\lambda \mathfrak b    \sqrt{a_0}}{2\epsilon_2} + \frac{\lambda \mathfrak b }{2\epsilon_3}+   \lambda \frac{\mathfrak b }{\sqrt 2}\Big]+ \lambda \mathfrak b (\mathfrak R+C   \sqrt 2 )|z|  \\ &\quad  +   \frac{\lambda \mathfrak b  \epsilon_2 \sqrt{a_0}}{2} |x|^2+
  \frac{\lambda \mathfrak b\epsilon_3 }{2 }|v|^2+\frac{\lambda \mathfrak b}{\sqrt 2}     S_{\mathbf{int}}(x).
  \end{align*}
   Note that $| S_{\mathbf{int}}(x)|= o( |\mathsf x|^2+ \mathfrak K_{\mathsf q^*}(x))$ when $\mathsf H_{\text{GL}}(\mathsf x)\to +\infty$.
 \end{itemize}
Finally, setting $M_\infty:=\sup_{\mathsf x=(x,v,z)\in \mathscr E}|v||z| /\mathsf J^2_{\mathfrak R }(\mathsf x)<+\infty$, one has for all  $\mathsf x=(x,v,z)\in \mathscr E$:
\begin{align*}
\delta \alpha |\nabla _z \mathsf F|^2&\le 3\mathfrak h^2  \delta \alpha    |z|^2+ 3 \mathfrak b^2 \mathfrak R^4 \delta \alpha  |v|^2 + 3 \mathfrak b^2 |C_{\mathsf G}^\infty|^2   M_\infty\,\delta \alpha \mathfrak R^{12}  |v||z| \\
&\le  3\delta \alpha  \big (\mathfrak h^2 +\mathfrak R^{12}\mathfrak b^2 |C_{\mathsf G}^\infty|^2   M_\infty/2\big )    |z|^2+  3\delta \alpha \mathfrak b^2 \big (  \mathfrak R^4  +  \mathfrak R^{12}|C_{\mathsf G}^\infty|^2   M_\infty/2\big )|v|^2.
\end{align*}
We now gather all the previous upper bounds, to get over $\mathscr E$, 
\begin{align*}
\mathcal L_{{\rm GL}} \mathsf F +\delta \alpha |\nabla _z \mathsf F|^2&\le   c_z |z|^2 +c_v |v|^2+c_x  |x|^2 - \mathfrak b \Big[  \frac{\mathsf c^* \mathfrak R}{2}-  \mathsf A \Big]\mathfrak K_{\mathsf q^*}(x)\mathsf f(z)   + o( |\mathsf x|^2+ \mathfrak K_{\mathsf q^*}(x)),
\end{align*} 
where  
\begin{itemize}
\item $ c_z=-\alpha \mathfrak h+ 3\delta \alpha  \mathfrak h^2+\mathfrak b \mathfrak R^2 \lambda +   \alpha \mathfrak b\mathfrak R^2  \frac{\epsilon_0}2+   \lambda \mathfrak b \mathfrak  R   \frac{\epsilon_1}2 + \lambda \mathfrak b \mathfrak R^3C^\infty_\mathsf G +  \mathfrak bC^\infty_\mathsf G  \frac{\lambda     \sqrt{a_0}}{2\epsilon_2} + \mathfrak b  C^\infty_\mathsf G \frac{\lambda  }{2\epsilon_3}+ \mathfrak  R ^2 a_0\frac{\mathfrak b^{1/2}}{2}+ 3\delta \alpha  \mathfrak b^2 \mathfrak R^{12}|C_{\mathsf G}^\infty|^2   M_\infty/2+\lambda C^\infty_\mathsf G  \frac{\mathfrak b }{\sqrt 2}$,
\item $ c_v=-  \mathfrak b \mathfrak R (\lambda\mathfrak R-1)+ \mathfrak b\mathfrak R^2  \frac{ \alpha }{2\epsilon_0}  + \mathfrak b C^\infty_\mathsf G \frac{\lambda  \epsilon_3 }{2 } +  3\delta \alpha \mathfrak b^2 \big (  \mathfrak R^4  +  |C_{\mathsf G}^\infty|^2 \mathfrak R^{12} M_\infty/2\big )$, 
\item 
 $ c_x=-a_0 \mathfrak b \mathfrak R+  \mathfrak bC^\infty_\mathsf G \frac{\lambda   \epsilon_2 \sqrt{a_0}}{2} +  \mathfrak b \mathfrak  R \frac{ \lambda}{2\epsilon_1}+  \mathfrak R^2 a_0\frac{\mathfrak b^{3/2}}{2}$.
 \end{itemize}
 Let 
 $\epsilon_0=\mathfrak b^{-1/2}$, $\epsilon_1= \mathfrak b^{-1/2}$, $\epsilon_2= \mathfrak b^{1/2}$, and $\epsilon_3= \mathfrak b^{1/2}$. 
  Pick and fix  $\mathfrak h>0$ small enough such that $-\alpha \mathfrak h+ 3\delta \alpha  \mathfrak h^2<0$ and $\mathfrak R>0$ large enough such that 
 $\frac{\mathsf c^* \mathfrak R}{2}-   \mathsf A>0$ and $\lambda\mathfrak R-1>0$. 
 Then, one has when $\mathfrak b\to 0^+$, $c_z= -\alpha \mathfrak h+ 3\delta \alpha  \mathfrak h^2 +o(1)$,  $c_v=-  \mathfrak b [\mathfrak R (\lambda\mathfrak R-1)+ o(1)]$, and $c_x=-a_0 \mathfrak b [\mathfrak R+o(1)]$. One chooses $\mathfrak b>0$ small enough such that $c_x,c_v,c_z<0$ and \eqref{eq.cond-h1}-\eqref{eq.cond-h1-2} hold. This leads to  the existence of  $c_1,c_2>0$ such that over $\mathscr E$:
 \begin{equation}\label{eq.cequonveut-I}
\mathcal L_{{\rm GL}} \mathsf F +\delta \alpha |\nabla _z \mathsf F|^2\le -c_1|\mathsf x|^2- c_2\mathfrak K_{\mathsf q^*}(x)\mathsf f(z) + o( |\mathsf x|^2+ \mathfrak K_{\mathsf q^*}(x)), 
 \end{equation}
 as $ |\mathsf x|\to +\infty$    or $x\to \partial \mathbf O_V$, 
or equivalently when $\mathsf H_{\text{GL}}(\mathsf x)\to +\infty$.
 With the same arguments as those used to conclude the proof of Proposition \ref{pr.Cint-1}, we deduce that $ {\mathcal L_{{\rm GL}}\mathsf W_\delta}/{ \mathsf W_\delta} \to -\infty$ as  $|\mathsf x|^2 +\mathfrak K_{\mathsf q^*}(x)\to +\infty$. This ends the proof of the proposition.    
\end{proof}

\section{The Nos\'e-Hoover process: proof of Theorem~\ref{th.3}}
\label{sec.NH-p}

\label{sec.NHC}
In this section, we prove Theorem~\ref{th.3}.
We recall  that, when $V$ satisfies \textbf{[V{\tiny coercive}]},  the Nos\'e-Hoover process process \eqref{eq.NHintro} evolves on  
$\mathscr E= \mathbf O_V\times \mathbf R^{dN}\times \mathbf R$ (see \eqref{eq.EE-NH} and Proposition \ref{pr.existence-NHC2}). 
  As it will be seen below,  \textbf{(C5)} is not satisfied for this process. To this end, to prove Theorem~\ref{th.3}, we will use an extension (Theorem~\ref{th.ext} below) of~\cite[Theorem 2.2]{guillinqsd}, and we will thus prove that this process satisfies \textbf{(C1)},  \textbf{(C2)}, \textbf{(C3)}, \textbf{(C4)}, and a new assumption, namely \textbf{(C5')}, which  is defined at the beginning of Section \ref{sec.extC5} below.


\subsection{On Assumptions \textbf{(C1)}, \textbf{(C2)}, \textbf{(C4)}, and \textbf{(C5')} for the process \eqref{eq.NHintro}}
\label{sec.4-1}

In this section, $V$ satisfies \textbf{[V{\tiny coercive}]} and we consider the solution $(X_t=(x_t,v_t,y_t),t\ge 0)$  to  \eqref{eq.NHintro}, see Proposition \ref{pr.existence-NHC2}.

 The  approach  we will use to prove    \textbf{(C1)} and \textbf{(C4)}   is the  energy splitting approach introduced in  Section~\ref{sec.22} (see \eqref{eq.splitting}).  The  uniform integrability of the transition probabilities  was   obtained in  Theorem~\ref{pr.SFgamma-neq-0} with a Gaussian upper bound. However,    Gaussian upper bound does not hold anymore for the Nosé-Hoover process.  Hence,   the proof of \eqref{eq.theta}  relies on different arguments than those used  in Theorem~\ref{pr.SFgamma-neq-0}.

\begin{thm}\label{th.NH-C}
Assume  that $V$ satisfies {\rm \textbf{[V{\tiny coercive}]}}. Then, the nonkilled semigroup $(P_t,t\ge 0)$ of the process \eqref{eq.NHintro} satisfies \textbf{(C1)}.
\end{thm}

\begin{proof}     
We want to prove that, for $f\in b\mathcal B(\mathscr E)$ and $t>0$, the following function 
\begin{equation}\label{eq.EC1}
\mathsf z\in \mathscr E\mapsto \mathbf E_{\mathsf z}\big [ f(X_t)] \text{  is continuous}. 
\end{equation}
 Pick  $R>0$. We will start by proving \eqref{eq.theta}, which 
  is the purpose of Step \textbf{2} below. Then we will prove 
 \eqref{eq.EC1} in Step \textbf{3}. Before doing that, we need a local  Girsanov formula~\eqref{eq.Gir-NH}, which is the aim of Step \textbf{1}. 
\medskip
   
\noindent
\textbf{Step 1.}  
In this step we derive \eqref{eq.Gir-NH} below. We recall that   (see also~\eqref{eq.H-NH}), 
$$\mathscr H_R=\{\mathsf x\in  \mathbf R^{3d}, \mathsf H_{{\rm NH}}(\mathsf x)<R\},$$ 
is an open bounded domain of $\mathscr E$.  Set $\Sigma=\sqrt{2\gamma}$.
Recall also that $\sigma_{\mathscr H_R}$ is  the first exit time for the process $(X_t,t\ge 0)$  from $\mathscr H_R$ and a.s. $\sigma=\lim_{R\to +\infty} \sigma_{\mathscr H_R}=\sup_{R>0} \sigma_{\mathscr H_R}=+\infty$.   For $\mathsf z=(x_z,v_z,y_z)\in  \mathscr E$, set also 
$$\boldsymbol f(\mathsf z)= -\nabla V(x_z)-\gamma v_z- v_zy_z \in \mathbf R^{dN}.$$
  Let us consider a   globally Lipschitz and bounded $\mathbf R^{dN}$-valued vector field $\boldsymbol f_R$ over  $\mathbf R^{dN}\times \mathbf R^{dN}\times  \mathbf R$   such that   $\boldsymbol f_R=\boldsymbol f$ in 
 a neighborhood of $\mathscr H_R$ in $\mathscr E$. Let also $ \boldsymbol  \chi_R :\mathbf R^{dN}\to \mathbf R$  be a globally Lipschitz  $\mathcal C^\infty$  function   such that  $\boldsymbol \chi_R(v)=|v|^2$ if $|v|^2/2\le R+1$. Denote by $(\hat X^R_t,t\ge 0)$ the  strong solution\footnote{Existence and uniqueness are ensured by the fact that  the coefficients  are  globally Lipschitz.}
 on $\mathbf R^{dN}\times \mathbf R^{dN}\times  \mathbf R$ of
  \begin{equation} \label{eq.NHstar}
 d  \hat x^R_t = \hat v^R_t dt,\ 
d\hat v^R_t=\boldsymbol f_R(\hat X_t^R)dt+  \Sigma \, dB_t,\ 
d\hat y^R_t=  \boldsymbol \chi_R (\hat v^R_t) dt -  d N dt.
\end{equation}  
  Let $\hat \sigma_{\mathscr H_R} =\inf\{t\ge 0, \hat X_t^R  \not \in \mathscr H_R\}$.  When $\mathsf z\in  \mathscr H_R$ and $0\le t<\hat \sigma_{\mathscr H_R}$, it holds 
 $\boldsymbol \chi_R(\hat v^R_t(\mathsf z))=|\hat v^R_t(\mathsf z)|^2$ and $\boldsymbol f_R(\hat X_t^R(\mathsf z))= \boldsymbol f(\hat X_t^R(\mathsf z))$. Therefore, one has for all  $g\in b\mathcal B(\mathscr H_R)$, $t>0$, and $\mathsf z=(x_z,v_z,y_z)\in  \mathscr H_R$, 
 $$\mathbf E_{\mathsf z}\big [ g(X_t)\mathbf 1_{t<\sigma_{\mathscr H_R}}\big ]=\mathbf E_{\mathsf z}\big [ g(\hat X_t^R)\mathbf 1_{t<\hat \sigma_{\mathscr H_R}}\big ],$$
  since the laws of the two processes coincide up to their first exit time from $\mathscr H_R$. 
 Let us now consider the     solution  $ (\tilde X^R_t=(\tilde x^R_t,\tilde v^R_t,\tilde y^R_t),t\ge 0)$ on  $\mathbf R^{dN}\times \mathbf R^{dN}\times  \mathbf R$ to 
$$
 d \tilde  x^R_t = \tilde v^R_t dt, \, d\tilde v^R_t= \Sigma \, dB_t, \, 
d\tilde y^R_t=  \boldsymbol \chi_R(\tilde v^R_t) dt -   d N dt.
$$
Set $\tilde \sigma_{\mathscr H_R} :=\inf\{t\ge 0, \tilde X_t^R \not \in \mathscr H_R\}$.  
Since $\boldsymbol f_R$ is bounded, $\sup_{s\in [0,t]}\mathbf E_{\mathsf z}[e^{|\boldsymbol f_R(\tilde X^R_s)|^2}]<+\infty$ for all $t>0$. 
Hence, using~\cite[Theorem~3.1 in Section 7]{friedman1975} and~\cite[Theorem~1.1  in Section~7]{friedman1975},  one has for all  $g\in b\mathcal B(\mathscr H_R)$, $t>0$, and $\mathsf z=(x_z,v_z,y_z)\in  \mathscr H_R$, 
$$\mathbf E_{\mathsf z}\big [ g(\hat X^R_t)\mathbf 1_{t<\hat \sigma_{\mathscr H_R}}\big ]=\mathbf E_{\mathsf z}\big [ g(\tilde X^R_t)\, \mathbf 1_{t<\tilde \sigma_{\mathscr H_R}}\, \tilde{\mathfrak M}_t(\boldsymbol f_R)\big ],$$
where  $\tilde{\mathfrak M}_t(\boldsymbol f_R)$ is the  exponential (true) martingale defined by $\tilde{\mathfrak M}_t(\boldsymbol f_R)=  \exp \tilde{\mathfrak L}_t(\boldsymbol f_R)$ and 
\begin{align*}
\tilde{\mathfrak L}_t (\boldsymbol f_R)&=  \int_0^t   \Sigma^{-1}\boldsymbol f_R (\tilde X_s^R) dB_s  -\frac 12 \int_0^t   \vert \Sigma^{-1} \boldsymbol f_R (\tilde X_s^R) \vert^2 ds.
\end{align*}   
Introduce finally the      solution $(X^0_t=(x^0_t,v^0_t,y^0_t),t\ge 0)$ on  $\mathbf R^{dN}\times \mathbf R^{dN}\times  \mathbf R$ to 
$$
 d  x^0_t = v^0_t dt,\  dv^0_t= \Sigma \, dB_t,\  dy^0_t=  |v^0_t|^2 dt - d N dt. 
$$
Because the laws of   $(\tilde X^R_t,t\ge 0)$ and $(X^0_t,t\ge 0)$ coincide up to their first exit time from $\mathscr H_R$, one finally gets that for any $\mathsf z\in \mathscr H_R$,
\begin{equation}
\label{eq.Gir-NH}
\mathbf E_{\mathsf z}\big [ g(X_t)\mathbf 1_{t<\sigma_{\mathscr H_R}}\big ]=\mathbf E_{\mathsf z}\big [ g(X^{ 0}_t)\, \mathbf 1_{t<\sigma_{\mathscr H_R}^{0}}\, \mathfrak M^0_t(\boldsymbol f_R)\big ],
\end{equation}
where $\mathfrak M^0_t(\boldsymbol f_R)=  \exp \mathfrak L_t^0(\boldsymbol f_R)$,   $ \mathfrak L_t^0(\boldsymbol f_R)=  \int_0^t   \Sigma^{-1}\boldsymbol f_R(X_s^0) dB_s  -\frac 12 \int_0^t   \vert \Sigma^{-1} \boldsymbol f_R(X_s^0) \vert^2 ds$, and $\sigma_{\mathscr H_R}^0:=\inf\{t\ge 0, X_t^0 \not \in \mathscr H_R\}$. 
\medskip

\noindent
\textbf{Step 2.}  In this step we show that for all $R>0$, $t>0$, and all $f\in b\mathcal B(\mathscr H_R)$,  the function
$$
\mathsf z\in \mathscr H_R\mapsto \mathbf E_{\mathsf z}\big [ f(X_t)\mathbf 1_{t<\sigma_{\mathscr H_R}}\big ] \text{ is continuous}.$$
In view of \eqref{eq.Gir-NH}, let us    show that  \begin{equation}\label{eq.EqSF}
\mathsf z\in \mathscr H_R\mapsto \mathbf E_{\mathsf z}\big [ f(X^{ 0}_t)\, \mathbf 1_{t<\sigma_{\mathscr H_R}^{0}}\, \mathfrak M^0_t(\boldsymbol f_R)\big ] \text{ is continuous}.
\end{equation}
We first claim that 
\begin{enumerate}
\item[\textbf i.] For $t>0$ and  $\mathsf x_n\to  \mathsf x $, $\sup_{s\in [0,t]}| X^0_s(\mathsf x_n)-X_s^0(\mathsf x)|\to 0$ a.s. as $n\to +\infty$. 
\item[\textbf{ii}.] For all   $\psi\in b\mathcal B(\mathbf R^{dN}\times \mathbf R^{dN}\times  \mathbf R)$ and $t> 0$, the following function 
$$\mathsf z\in \mathbf R^{dN}\times \mathbf R^{dN}\times  \mathbf R \mapsto \wp_t(\psi)(\mathsf z)= \mathbf E_{\mathsf z}\big [e^{ \mathfrak L_t^0(\boldsymbol f_R)} \psi(X^0_t) \big ]  \text{ is continuous}.$$ 
\item[\textbf {iii}.] For all $\delta>0$ and all compact subset $K$ of $\mathbf R^{dN}\times \mathbf R^{dN}\times  \mathbf R $, 
$$ \lim_{s\to 0^+} \sup_{\mathsf z\in K}\mathbf P_{\mathsf z}[\sigma^0_{\mathsf B(\mathsf z,\delta)}\le s]=0,$$
 where $\mathsf B(\mathsf z,\delta)$ is the open ball of $\mathbf R^{dN}\times \mathbf R^{dN}\times  \mathbf R $ centered at $\mathsf z$ of radius $\delta$, and $\sigma^0_{\mathsf B(\mathsf z,\delta)}$ is the first exit time from $\mathsf B(\mathsf z,\delta)$ for the process $(X^0_t,t\ge 0)$. 
\end{enumerate} 
 The proof of Item \textbf{i} is straightforward. Let us now prove Item \textbf{ii}. Let $\mathsf z_n\to \mathsf z\in  \mathbf R^{dN}\times \mathbf R^{dN}\times \mathbf R$.  By H\"ormander's theory,  for all $t>0$, $X_t^0$ admits   a smooth density with respect to  the Lebesgue measure over $\mathbf R^{dN}\times \mathbf R^{dN}\times  \mathbf R$.   
 Using also Item~\textbf{i} and Proposition~\ref{pr.wu1999}, we deduce that $\psi(X_t^0(\mathsf z_n))\to \psi(X_t^0(\mathsf z))$ in $\mathbf P$-probability. In addition, using Item~\textbf{i} and the regularity of $\boldsymbol f_R$, $\mathfrak L_t^0(\boldsymbol f_R)  (\mathsf z_n)\to \mathfrak L_t^0(\boldsymbol f_R)  (\mathsf z)$ in $\mathbf P$-probability. By the continuous mapping theorem, $e^{ \mathfrak L_t^0(\boldsymbol f_R) } (\mathsf z_n)\to e^{ \mathfrak L_t^0(\boldsymbol f_R)} (\mathsf z)$ in $\mathbf P$-probability. Since $\mathbf E_{\mathsf z_n}[e^{ \mathfrak L_t^0(\boldsymbol f_R)} ]=\mathbf E_{\mathsf z}[e^{ \mathfrak L_t^0(\boldsymbol f_R)} ]=1$, one deduces that 
 $$\mathbf E_{\mathsf z_n}[e^{ \mathfrak L_t^0(\boldsymbol f_R)} ]\to \mathbf E_{\mathsf z}[e^{ \mathfrak L_t^0(\boldsymbol f_R)} ] \text{ in } L^1(\mathbf P).$$ This ends the proof of Item \textbf{ii}.  
 Let us now prove Item \textbf{iii}. Fix $\delta>0$ and a compact subset  $K$ of $\mathbf R^{dN}\times \mathbf R^{dN}\times  \mathbf R $. With the same arguments as those used at the beginning of  the proof of Lemma \ref{le.C4}, we have
  \begin{align*}
   \mathbf P_{\mathsf z}[\sigma^0_{\mathsf B(\mathsf z,\delta)}\le s ]&\le \mathbf P_{\mathsf z}\big [\sigma^0_{\mathsf B(\mathsf z,\delta)}\le s, |X^0_s-X^0_0|< \delta/2\big ]  +   \mathbf P_{\mathsf z}\big [  |X^0_s-X^0_0|\ge \delta/2\big]
   \end{align*}
and by the strong Markov property, 
   \begin{align*}
     & \mathbf P_{\mathsf z}\big [\sigma^0_{\mathsf B(\mathsf z,\delta)}\le s, |X^0_s-X^0_0|< \delta/2\big ]   \\ 
     &\le   \mathbf p(s,\mathsf z):=\mathbf E_{\mathsf z}\Big [\mathbf 1_{\sigma^0_{\mathsf B(\mathsf z,\delta) }(\mathsf z) \le s}     \mathbf P_{X^0_{\sigma^0_{\mathsf B(\mathsf z,\delta)}(\mathsf z) } (\mathsf z)}\big [|X^0_{s-\sigma^0_{\mathsf B(\mathsf z,\delta)}(\mathsf z) }-X^0_0|\ge \delta/2  ]  \Big ].
   \end{align*}
 Let us deal with $ \mathbf p(s,\mathsf z)$. For any $\mathsf z\in K$,   $\mathsf x=(x,v,z)\in \partial \mathsf B(\mathsf z,\delta) $, and $u\in [0,s]$,  $ \mathbf P_{\mathsf x } \big [|X^0_{s-u }- \mathsf x|\ge \delta/2  ]\le  \mathbf P_{\mathsf x } \big [|x^0_{s-u }- x|\ge \delta/2  ]+\mathbf P_{\mathsf x } \big [|v^0_{s-u }- v|\ge \delta/2  ]+ \mathbf P_{\mathsf x } \big [|y^0_{s-u }- y|\ge \delta/2  ]$. Note that $|\mathsf x|\le  c_{\delta,K}:=\delta/2 + r_K$, where $r_K>0$ is such that $K\subset \mathsf B(0,r_K)$. 
First, it holds $\mathbf P_{\mathsf x } \big [|v^0_{s-u }- v|\ge \delta/2  ]\le    \mathbf a_s:=\mathbf P[\sup_{t\in [0,s]}|B_t|\ge \delta/(2\Sigma)]$. On the other hand, 
$ \mathbf P_{\mathsf x } \big [|x^0_{s-u }- x|\ge \delta/2  ] \le \mathbf  b_s:=\mathbf P  \big [ sc_{\delta,K}+ \int_0^s \Sigma|B_t|dt \ge \delta/2 ]$. Finally, $\mathbf P_{\mathsf x } \big [|y^0_{s-u }- y|\ge \delta/2  ]\le \mathbf P_{\mathsf x } \big [ \int_0^s |v_t|^2 dt +s  dN\ge \delta/2 ]\le  \mathbf c_s:=\mathbf P \big [ 2\Sigma^2\int_0^s |B_t|^2 dt + 2 sc_{\delta,K}^2 + s  dN\ge \delta/2 ]$.  Thus, one has as $s\to 0^+$: 
$$\sup_{\mathsf z\in K} \mathbf p(s,\mathsf z) \le \mathbf a_s+\mathbf b_s+\mathbf c_s\to 0.$$ 
Similarly, we have $\sup_{\mathsf z\in K}\mathbf P_{\mathsf z}\big [  |X^0_s-X^0_0|\ge \delta/2\big]\to 0$ as $s\to 0^+$. This concludes the proof of   Item \textbf{iii}.

We will use the fact that Item \textbf{iii} implies that for all compact subset $K$ of $\mathscr H_R$:   
\begin{equation}\label{eq.ItemIII}
 \lim_{s\to 0^+} \sup_{\mathsf z\in K}\mathbf P_{\mathsf z}[\sigma^0_{\mathscr H_R}\le s]=0.
   \end{equation}
   
We are now in position  to prove  \eqref{eq.EqSF}.   We have by the Markov property, for $0\le s\le t$ and $\mathsf z\in \mathscr H_R$:   
\begin{align}
\label{eq.=Ch}
&\mathbf E_{\mathsf z}\big [ f(X^{ 0}_t)\, \mathbf 1_{t<\sigma_{\mathscr H_R}^{0}}\, \mathfrak M^0_t(\boldsymbol f_R)\big ] =\mathbf E_{\mathsf z}\big [ \mathbf 1_{s<\sigma_{\mathscr H_R}^{0}} \mathfrak M^0_{s}(\boldsymbol f_R)\,  \mathbf E_{X_s^0}\big[f(X^{ 0}_{t-s})\, \mathbf 1_{t-s<\sigma_{\mathscr H_R}^{0}}\, \mathfrak M^0_{t-s}(\boldsymbol f_R)\big]\big ].
\end{align}
Define $\psi_s$ by $\psi_s(\mathsf z)=\mathbf E_{\mathsf z}\big[f(X^{ 0}_{t-s})\, \mathbf 1_{t-s<\sigma_{\mathscr H_R}^{0}}\, \mathfrak M^0_{t-s}(\boldsymbol f_R)\big] $ for $\mathsf z\in \mathscr H_R$ (we extend $\psi_s$ by $0$ outside $\mathscr H_R$). The function $\psi_s$ is measurable and bounded (by $ \Vert f\Vert_\infty  \times \mathbf E_{\mathsf z}\big[ \mathfrak M^0_{t-s}(\boldsymbol f_R) \big] = \Vert f\Vert_\infty $). By Item \textbf{ii}, $\mathsf z\mapsto \wp_s(\psi_s)(\mathsf z)=\mathbf E_{\mathsf z}\big [ \mathfrak M^0_{s}(\boldsymbol f_R) \psi_s(X_s^0)\big ]$ is continuous. In addition, from~\eqref{eq.=Ch} and Doob's martingale inequality, one has for $\mathsf z\in \mathscr H_R$ and $0\le s\le t$:
\begin{align}
\nonumber
 \big | \mathbf E_{\mathsf z}\big [ f(X^{ 0}_t)\, \mathbf 1_{t<\sigma_{\mathscr H_R}^{0}}\, \mathfrak M^0_t(\boldsymbol f_R)\big ]-  \wp_s(\psi_s)(\mathsf z)\big| 
&= | \mathbf E_{\mathsf z}\big [ \mathbf 1_{  s>\sigma_{\mathscr H_R}^{0}}\mathfrak M^0_{s}(\boldsymbol f_R) \psi_s(X^0_s)\big ]|\\
\nonumber
&\le   \Vert \psi_s \Vert_\infty  \sqrt{ \mathbf E_{\mathsf z}\big [\sup_{s\in [0,t]}|\mathfrak M^0_{s}(\boldsymbol f_R) |^2\big]}\, \sqrt{ \mathbf P_{\mathsf z}\big [\sigma_{\mathscr H_R}^{0}\le s\big ]  }\\
\label{eq.C-HH}
&\le   \Vert \psi_s\Vert_\infty   \sqrt{ 4\mathbf E_{\mathsf z}\big [|\mathfrak M^0_{t}(\boldsymbol f_R) |^2\big]}   \mathbf P_{\mathsf z}\big [\sigma_{\mathscr H_R}^{0}\le s\big ]   .
\end{align}
We claim  that $ \sup_{\mathsf z}\mathbf E_{\mathsf z}\big [|\mathfrak M^0_{t}(\boldsymbol f_R) |^2\big]<+\infty$. Let us prove this claim. 
We have 
\begin{align*}
\mathbf E_{\mathsf z}\big [|\mathfrak M^0_{t}(\boldsymbol f_R) |^2\big]&=\mathbf E_{\mathsf z}\big [\exp (2 \mathfrak L_t^0(\boldsymbol f_R))\big]\\
&=\mathbf E_{\mathsf z}\Big [\exp (  \mathfrak L_t^0(2\boldsymbol f_R)) \exp \Big ( \int_0^t   \vert \Sigma^{-1} \boldsymbol f_R(X_s^R) \vert^2ds\Big )\Big ]\\
&\le \exp( \Sigma^{-2}t \Vert \boldsymbol f_R\Vert^2_\infty ) \times  \mathbf E_{\mathsf z}\Big [\exp (  \mathfrak L_t^0(2\boldsymbol f_R)) \Big ] =  \exp( \Sigma^{-1}t \Vert \boldsymbol f_R\Vert_\infty ),
\end{align*}
where we have used that $\mathbf E_{\mathsf z}  [\exp (  \mathfrak L_t^0(2\boldsymbol f_R))  ] =1$ (apply for instance   \cite[Theorem~1.1  in Section 7]{friedman1975}). 
This proves that $\mathbf E_{\mathsf z}\big [|\mathfrak M^0_{t}(\boldsymbol f_R) |^2\big]<+\infty$. Combining \eqref{eq.C-HH} and \eqref{eq.ItemIII}, we deduce that the continuous function $\mathsf z\mapsto \wp_s(\psi_s)(\mathsf z) $ converges uniformly as $s\to 0^+$ to  the function 
$$\mathsf z\mapsto \mathbf E_{\mathsf z}\big [ f(X^{ 0}_t)\, \mathbf 1_{t<\sigma_{\mathscr H_R}^{0}}\, \mathfrak M^0_t(\boldsymbol f_R)\big ]= \mathbf E_{\mathsf z}\big [ f(X_t)\mathbf 1_{t<\sigma_{\mathscr H_R}}\big ]$$
uniformly on the  compact subsets  $K$  of $\mathscr H_R$. This ends the proof of  \eqref{eq.EqSF}. 

Note that we actually proved that for any nonempty bounded open subset $\mathscr H$ of $\mathscr E$,
\begin{align}\label{eq.=EH}
\mathsf z\in \mathscr H \mapsto \mathbf E_{\mathsf z}\big [ f(X_t)\mathbf 1_{t<\sigma_{\mathscr H}}\big ] \text{ is continuous}.
\end{align}

\noindent
\textbf{Step 3.} Let us end the proof of   \textbf{(C1)}, i.e. let us prove \eqref{eq.EC1}. Fix $t>0$ and   a compact subset $K$ of $\mathscr E$. Let $R_K>0$ be such that $K\subset \mathscr H_{R_K}$ (thus $K\subset\mathscr H_{R}$ for all $R\ge R_K$). We have  for all $R\ge R_K$ and $\mathsf z\in K$: 
$$ |\mathbf E_{\mathsf z}\big [ f(X_t)]-\mathbf E_{\mathsf z}\big [ f(X_t)\mathbf 1_{t<\sigma_{\mathscr H_R}}\big ]|\le \Vert f\Vert_\infty \sup_{\mathsf z\in  K} \mathbf P_{\mathsf z}[\sigma_{\mathscr H_R}\le t]  \to 0 \text{ as } R\to +\infty,$$
since $\mathbf P_{\mathsf z}[\sigma_{\mathscr H_R}\le t]\le \frac{e^{ct} }{R} \,  \mathsf H_{{\rm NH}}(\mathsf z)$. Therefore, the continuous function $\mathsf z\mapsto \mathbf E_{\mathsf z}\big [ f(X_t)\mathbf 1_{t<\sigma_{\mathscr H_R}}\big]$ converges uniformly  over $K$ to $\mathsf z\mapsto \mathbf E_{\mathsf z}\big [ f(X_t)]$ as $R\to +\infty$. This concludes the proof of Equation~\eqref{eq.EC1}. \end{proof}



\begin{prop}\label{pr.Vs3-t}
 Assume  that $V$ satisfies {\rm \textbf{[V{\tiny coercive}]}}. Then, the nonkilled semigroup $(P_t,t\ge 0)$ of the solution of~\eqref{eq.NHintro} over $\mathscr E=\mathbf O_V\times \mathbf R^{dN}\times \mathbf R$ satisfies  \textbf{(C2)}. Let $\mathscr D=\mathscr O\times \mathbf R^{dN}\times \mathbf R$ where $\mathscr O$ is a subdomain of $\mathbf O_V$. Then, the killed semigroup $(P^{\mathscr D}_t,t\ge 0)$
of the process~\eqref{eq.NHintro} on $\mathscr D$ satisfies  \textbf{(C4)}. If   $\mathbf O_V\setminus \overline{\mathscr O}$ is nonempty,  $(P^{\mathscr D}_t,t\ge 0)$ satisfies  \textbf{(C5')} (see Section \ref{sec.extC5}).  
\end{prop}

\begin{proof}  
Assumption \textbf{(C2)} is proved with the same arguments as those used to prove Proposition \ref{pr.C2}. Let us now prove \textbf{(C4)}. Let $t>0$ and $f\in b\mathcal B(\mathscr D)$. Consider  $\mathsf x\in \mathscr D$ and $r>0$ such that  $\mathsf B(\mathsf x,r)\subset \mathscr D$.  Set  $\mathscr D_R=\mathscr D \cap \mathscr H_R$ which is a  bounded open subset of $\mathscr E$ and contains $\mathsf B(\mathsf x,r)$ for all $R>0$ large enough.  Since $\sigma_{\mathscr D_R}\le \sigma_{\mathscr D}$, one has $\mathbf 1_{t<\sigma_{\mathscr D}}\mathbf 1_{t<\sigma_{\mathscr D_R}}= \mathbf 1_{t<\sigma_{\mathscr D_R}}$. Consequently, one has for all $\mathsf z\in \mathsf B(\mathsf x,r)$,   
\begin{align*}
| \mathbf E_{\mathsf z}\big [ f(X_t)\, \mathbf 1_{t<\sigma_{\mathscr D}}\big]-  \mathbf E_{\mathsf z}\big [ f(X_t)\, \mathbf 1_{t<\sigma_{\mathscr D_R}}\big]|&=|  \mathbf E_{\mathsf z}\big [ f(X_t)\,  \mathbf 1_{t<\sigma_{\mathscr D}}\mathbf 1_{t\ge \sigma_{\mathscr D_R}}\big]|\\&\le \Vert f\Vert_\infty \sup_{\mathsf z\in \mathsf B(\mathsf x,r)} \mathbf P_{\mathsf z}[t<\sigma_{\mathscr D},t\ge \sigma_{\mathscr D_R}]\\
&\le \Vert f\Vert_\infty \sup_{\mathsf z\in \mathsf B(\mathsf x,r)} \mathbf P_{\mathsf z}[t\ge \sigma_{\mathscr H_R}] \to 0
\end{align*}
 as $R\to +\infty$. The proof of~\textbf{(C4)} is complete using \eqref{eq.=EH}. 
Let us now check  \textbf{(C5')}. Fix $t>0$ and $\mathsf x_*=(x_*,v_*,y_*)\in \mathscr D$. With the analysis led in Section \ref{sec.S5GL}, it is not difficult to see that   $\mathfrak A_{\mathsf x_*,t}:=\{\mathsf x=(x,v,y)\in \mathscr D, y\ge y_*+t^{-1}\mathsf d_{\mathscr O}^2(x,x_*) - t dN\}\subset  {\rm supp}\, P^{\mathscr D}_t(\mathsf x_*,d\mathsf y)$, where $\mathsf d_{\mathscr O}(x,x_*) $ is the geodesic distance between $x$ and $x_*$ in $\mathscr O$ (actually $\mathfrak A_{\mathsf x_*,t}={\rm supp}\, P^{\mathscr D}_t(\mathsf x_*,d\mathsf y)$ thanks to~\cite[Proposition~2.1~(ii)]{herzogNH}). Then, for any open ball $\mathsf B=\mathsf B(\mathsf y,r)\subset \mathscr D$, there exists $ t(\mathsf x_* ,\mathsf B)$ such that for all $t\ge t(\mathsf x_* ,\mathsf B)$,  $\mathsf B(\mathsf y,r)\subset \mathfrak A_{\mathsf x_*,t}$. This ends the proof of~ \textbf{(C5')}.
\end{proof}

\subsection{On Assumption \textbf{(C3)} for the process \eqref{eq.NHintro}}
In this section, we consider a potential $V$ satisfying \textbf{[V{\tiny sing2}]}. 
Recall that  the Hamiltonian of the process   \eqref{eq.NHintro} is given by $\mathsf H_{{\rm NH}}  (x,v,y)=  V(x)+\frac 12 |v|^2+ \frac 12 |y|^2$ ($(x,v,y)\in \mathscr E$) and its infinitesimal generator is $\mathcal L_{{\rm NH}}= v\cdot \nabla _x  - ( y+\gamma)v\cdot \nabla_v-\nabla V\cdot  \nabla_v + \gamma   \Delta_v +  (|v|^2-d N  )\partial_y$ (see  \eqref{eq.H-NH} and \eqref{eq.L-NH}). Note that 
\begin{equation}\label{eq.CHH}
\mathcal L_{{\rm NH}}\mathsf H_{{\rm NH}}=- y  Nd- \gamma |v|^2+\gamma   dN.
\end{equation}
The purpose of  \cite{herzogNH} is to show the ergodicity of 
the (nonkilled) Nos\'e-Hoover process on $\mathscr E$. For that purpose, and based on  an extension of Harris’ ergodic theorem~\cite{hairer2011yet}, the strategy  consists in particular in constructing a Lyapunov function $\mathsf w$ such that at high energy $\mathsf H_{{\rm NH}}$ and on $\mathscr E$:
\begin{equation}\label{eq.Obh}
\mathcal L_{{\rm NH}}\mathsf w/\mathsf w\le -\alpha \text{ for some } \alpha>0.
\end{equation}
As observed there, the  natural candidate  $e^{ \mathsf H_{{\rm NH}}}$  does not satisfy such an estimate.  In view of  \eqref{eq.CHH},   though there exists $\alpha>0$ such that on the region $\mathcal R_0=\{(x,v,y), y\ge c_* \text{ or }  |v|^2\ge y_*\sqrt{|y|^2+1}\}$, 
$\mathcal L_{{\rm NH}}\mathsf H_{{\rm NH}}\le -\alpha$ (choosing $c_*,y_*>0$ large enough), one cannot obtain \eqref{eq.Obh} in $\mathcal R_0^c$.    
Nevertheless, by constructing  a suitable perturbation  $\Psi$ of $\mathsf H_{{\rm NH}}$ and by considering  the Lyapunov function $\mathsf w=e^{\mathfrak h_*\mathsf H_{{\rm NH}}+\Psi}$ ($\mathfrak h_*>0$), D.P. Herzog~\cite{herzogNH}  managed to obtain   $\mathcal L_{{\rm NH}}\mathsf w\le -\alpha$ on the  region $\mathcal R_0^c$. 
Nonetheless, in order to get \textbf{(C3)} (i.e. to have a Lyapunov function $\mathsf W$ such that $\mathcal L_{{\rm NH}}\mathsf W/\mathsf W\to -\infty$ as $\mathsf H_{{\rm NH}}\to +\infty$), we have to introduce another perturbation $\Phi$ of $\mathsf H_{{\rm NH}}$, especially to get that $\mathcal L_{{\rm NH}}\mathsf W/\mathsf W\to -\infty$ as $V(x)\to +\infty$.

Before building  $\Phi$, we start by recalling the construction of $\Psi$ in   \cite{herzogNH}.  To this end,  introduce the Dawson's integral $D:\mathbf R\to \mathbf R$ defined by  $D(z)= e^{-z^2}\int_0^z e^{u^2}du$ and set $ D_{\text{m}}:= \sup D $. 
Let us consider $k_*>0$ and $y_*>3\gamma+1$. Consider the following cutoff functions in $\mathcal C^\infty(\mathbf R, [0,1])$: 
\begin{itemize}
\item[\textbf -] $\mathfrak f_0(z)=1$ if $z\le -1$, $\mathfrak f_0(z)=0$ if $z\ge 0$, $\mathfrak f_0'\le 0$, and $|\mathfrak f_0'|\le 2$. 
\item[\textbf -]  $\mathfrak f_1(z)=1$ if $z\le k_*$ and $\mathfrak   f_1(z)=0$ if $z\ge k_*+1$.

\item[\textbf -] $\mathfrak   f_2(z)=1$ if $|z|\le 1$ and $\mathfrak   f_2(z)=0$ if $|z|\ge 2$.
\item[\textbf -] $\mathfrak f_3(z)=1$ if $|z|\ge 2$ and $\mathfrak f_3(z)=0$ if $|z|\le 1$.
\item[\textbf -] $\mathfrak h_1(z)=1$ if $z\le -y_*-1$, $\mathfrak h_1(z)=0$ if $z\ge -y_*$,   and $|\mathfrak h_1'|\le 2$. 
\item[\textbf -] $\mathfrak h_3(z)=1$ if $|z|\le 3$ and $\mathfrak h_3(z)=0$ if $|z|\ge 4$.
\end{itemize}
Pick now $p_*,u_*>0$, and  set for $\mathsf x=(x,v,y)\in \mathscr E$, $\Theta(v,y)=  |v|^2/(p_*\sqrt{|y|^2+1}) $ and $\Upsilon(x,y)= |\nabla V(x)|^2/[u_*(|y|^2+1)]$, 
$$  \mathfrak   g_1(\mathsf x)=\mathfrak   f_1(y) \, \mathfrak   f_2 (\Theta(v,y))\,  \mathfrak f_3(\Upsilon(x,y)) \text{ and } \mathfrak   g_2(\mathsf x)=\mathfrak h_1(y)\,  \mathfrak   f_2(\Theta(v,y)) \,  \mathfrak h_3(\Upsilon(x,y)).$$
 Define for $z\in \mathbf R$ and $\mathsf x=(x,v,y)\in \mathscr E$:
 $$\mathfrak F(z)= -\frac{1}{2D^2_{\text{max}}}\int_0^ze^{-|y|^2}\int_0^y e^{u^2}du dy \text{ and } \Psi_0(\mathsf x)=\delta_* \mathfrak f_0(y) \frac{|y|^2}2,$$
 where $\delta_*>0$. 
For $k\in \{1,\ldots,N\}$, we denote by $x^k=(x^{k,1},\ldots,x^{k,d})$ and by $v^k=(v^{k,1},\ldots,v^{k,d})$ the $k^{\text{th}}$-coordinate in $\mathbf R^d$ of $x=(x^1,\ldots,x^N)\in \mathbf O_V$ and of $v=(v^1,\ldots,v^N)\in \mathbf R^{dN}$, respectively.  For $\mathsf x=(x,v,y)\in \mathscr E$, one sets:
$$\Psi_1(\mathsf x)=\alpha_*\mathbf 1_{\{|\nabla V(x)|^2\ge u_*/2\}} \mathfrak   g_1(\mathsf x) \, \sqrt{|y|^2+1} \ \frac{v\cdot \nabla V(x)}{|\nabla V(x)|^2} \text{ where } \alpha_*>0.$$
For $k\in \{1,\ldots,N\}$ and $l\in  \{1,\ldots,d\}$, one  defines  
$$  \Psi_2^{k,l}(\mathsf x)=\mathbf 1_{\{y\le -3\gamma\}} \mathfrak   g_2(\mathsf x)\mathfrak F\Big(\frac{|y+\gamma|^{\frac 12}}{(2\gamma)^{\frac 12}}  ( v^{k,l}-\frac{\partial_{x^{k,l}}V(x) }{|y+\gamma|}   )\Big), \text{ and } \Psi_2(\mathsf x)=\sum_{k,l}\Psi_2^{k,l}(\mathsf x).$$
Finally, set $\Psi=\Psi_0+\Psi_1+\Psi_2$.  We now briefly recall the result of \cite{herzogNH} we will need. Let $\mathfrak h_*>0$ be such that:
\begin{equation}\label{eq.Tee}
\mathfrak h_*<\frac{1}{8 D_{\text{m}}^2}.
\end{equation}
Define on the following function on $\mathscr E$:  $\mathsf L_{{\rm NH}}=\mathfrak h_*\mathsf H_{{\rm NH}}+\Psi$.  Then, from~\cite[item (i) in Theorem 4.1, Eq. (4.4), and Eq. (4.8)]{herzogNH} (see also the last equation in the proof of~\cite[Lemma 4.2]{herzogNH}), for some  $\delta_0>0$ and  any $\delta_*\in (0,\delta_0)$, there exists  $\varepsilon_0>0$ such that  for all $\varepsilon_*\in  (0,\varepsilon_0)$, there exists $c_0>0$, for all $\alpha_*>  c_0$ and $k_*>c_0$,  choosing $p_*,u_*,y_*>0$   large enough,  it holds over $\mathscr E$:
\begin{equation}\label{eq.ePsi}
|\Psi|\le \varepsilon_* \mathsf H_{{\rm NH}} \text{ and } |\nabla _v \mathsf L_{{\rm NH}}|\le \mathfrak h_*|v|+ \frac{\mathfrak   g_2 Nd |y+\gamma|^{\frac 12} }{2D_{\text{m}}(2\gamma)^{\frac 12}}+ (Nd+1)\varepsilon_*,
\end{equation}  
and (see the equation after~\cite[Eq. (4.32)]{herzogNH})
\begin{align} 
\nonumber
 \mathcal L_{{\rm NH}}\mathsf L_{{\rm NH}}+ \gamma  |\nabla _v \mathsf L_{{\rm NH}}|^2 &\le -\frac{\gamma \mathfrak h_*}2 (1- \mathfrak h_*)|v|^2- \mathfrak h_*dN (1-\mathfrak f_0)|y| -\frac{\mathfrak f_0}2\delta_*|y||v|^2\\
 \label{eq.ePsi2}
 &\quad - \big[\mathfrak   g_1\frac{\alpha_*}2+\mathfrak   g_2\frac{dN}{8D_{\text{m}}^2} - \mathfrak f_0(\mathfrak h_*+\delta_*+\varepsilon_*) dN \big] |y|+ \mathfrak c,  
\end{align} 
where $\mathfrak c>0$ is independent of $\mathsf x\in \mathscr E$, which, in the following,  can change from one occurence to another. 

Let us now construct $\Phi$. 
 Let $R_1>1$ such that $|\nabla V(x)|\ge 1$ if $V(x)\ge R_1-1$. Consider the following two cutoff functions in $\mathcal C^\infty(\mathbf R, [0,1])$: 
\begin{itemize}
\item[\textbf -] $\mathfrak h(z)=1$ if $z\ge R_1$, $\mathfrak h(z)=0$ if $z\le R_1-1$, and  $|\mathfrak h'|\le 2$. 
\item[\textbf -]  $\mathfrak h_0(z)=1$ if $z\ge 2$ and $\mathfrak h_0(z)=0$ if $z\le 1$, $\mathfrak h_0'\ge 0$, $|\mathfrak h_0'|\le 2$. 
\end{itemize}
Let $\zeta\in (1,2)$ be as in \textbf{[V{\tiny sing2}]}. 
Define  for $\mathsf x\in \mathscr E$, 
$$\Phi(\mathsf x)=\mathfrak h(V(x)) \,  \frac{v\cdot \nabla V(x)}{ |\nabla V(x)|^{\zeta}}-\mathfrak h_0(y)|y|^2.$$
\begin{remarks}
The sign of the second term $-\mathfrak h_0(y)|y|^2$  in the definition of $\Phi$ might look strange. However we will rather consider instead of $\Phi$, $\epsilon_\Phi\Phi$ (with   $\epsilon_\Phi>0$ small enough) as a perturbation of $\mathsf L_{{\rm NH}}$ (see indeed the definition of $\mathsf F_{{\rm NH}} $ below). Hence, $\mathsf H_{{\rm NH}}+\Psi+ \epsilon_\Phi\Phi$ will stay lower bounded over $\mathscr E$ (see the proof of \eqref{eq.inf-NH} below). 
\end{remarks}
Let us now provide some estimates on $\Phi$, $\nabla _v\Phi$, and $ \mathcal L_{{\rm NH}}\Phi$ which will be used later on. 

On the one hand,  since $\zeta\in (1,2)$, it holds over $\mathscr E$:
\begin{equation}\label{eq.Phi1}
|\Phi|\le  \mathfrak h\,  |v||\nabla V|^{1-\zeta}+|y|^2\le |v|+ |y|^2\le 2 \mathsf H_{{\rm NH}} \text{ and } |\nabla _v\Phi|= \mathfrak h\, |\nabla V|^{1-\zeta}\le 1.
\end{equation}

On the other hand, using \textbf{[V{\tiny sing2}]}, there exists  $M>0$  such that  
$$\mathfrak h(V(x)) | \text{Hess }V(x) | / |\nabla V(x)|^{\zeta}\le M.$$ 
Set 
$$\mathfrak c_V= 3M+  2\sup_{  \{x, V(x)\in [R_1-1,R_1]\}}  |\nabla V|^{2-\zeta}. $$
Note that because  $|\nabla V|\ge 1$ on the support of $\mathfrak h\circ V$ and since  $\zeta-1>0$, we have that  $1/|\nabla V|^{\zeta-1}\le 1$. Hence, it holds:
\begin{align*} 
 \mathcal L_{{\rm NH}}\Phi&= v\cdot \Big[ \mathfrak h'\circ V\, \nabla V \,  \frac{v\cdot \nabla V }{ |\nabla V |^{\zeta}} + \mathfrak h\circ V   \frac{\text{Hess }V  v}{ |\nabla V |^{\zeta}}  -\zeta\mathfrak h\circ V\, \frac{v\cdot  \nabla V\,  \text{Hess }V    \, \nabla V }{ |\nabla V |^{2+\zeta}} \Big]\\
 &\quad - (yv+\gamma v+\nabla V)\cdot \mathfrak h\circ V\,  \frac{  \nabla V(x)}{ |\nabla V(x)|^{\zeta}} - (|v|^2-  dN)(\mathfrak h_0' |y|^2+2 \mathfrak h_0y)\\
 &\le |v|^2\Big[|\mathfrak h'\circ V|  |\nabla V|^{2-\zeta} + \mathfrak h\circ V \frac{|\text{Hess }V  |}{ |\nabla V |^{\zeta}} +2\mathfrak h\circ V   \frac{|\text{Hess }V  |}{ |\nabla V |^{\zeta}} \Big] \\
 &\quad + (|y||v|+\gamma| v|) \frac{\mathfrak h\circ V}{|\nabla V|^{\zeta-1}}- \mathfrak h\circ V |\nabla V|^{2-\zeta} -\mathfrak h_0' |y|^2|v|^2 -2\mathfrak h_0 y|v|^2\\
 &\quad +  dN(\mathfrak h_0' |y|^2+2 \mathfrak h_0y)\\
 &\le \mathfrak c_V |v|^2 - \mathfrak h\circ V |\nabla V|^{2-\zeta} +(|y||v|+\gamma| v|) \mathfrak h\circ V  -2\mathfrak h_0 y|v|^2+ 2\mathfrak h_0  dN y +\mathfrak c.
  \end{align*}
  Hence,
  one has:
  \begin{align}
 \label{eq.LPhi}
 \mathcal L_{{\rm NH}}\Phi&\le \mathfrak c_V |v|^2 -2  \mathfrak h_0 y|v|^2 + \gamma|v|- \mathfrak h\circ V |\nabla V|^{2-\zeta }    +|y||v|  \mathfrak h\circ V+ 2   dN \mathfrak h_0 y + \mathfrak c.
 \end{align}

For all $\mathsf x\in \mathscr E$, $\epsilon_\Phi>0$, set $\mathsf F_0(x,v,y)= \mathsf L_{{\rm NH}} +\epsilon_\Phi \Phi$. 
The parameter $\epsilon_\Phi>0$ will be chosen such that 
\begin{equation}\label{eq.inf-NH}
\inf_{\mathscr E}\mathsf F_0 \in \mathbf R.
\end{equation}
Then, we define on $\mathscr E$:
$\mathsf F_{{\rm NH}} =\mathsf F_0 -\inf_{\mathscr E}\mathsf F_0 +1$ and 
  \begin{equation}\label{eq.Lyapunov-NH}
  \mathsf W_\delta =\exp\big [ \mathsf F_{{\rm NH}} ^\delta\big ],   \text{ where }  \delta\in (1/2,1] \text{ is  the parameter appearing in \textbf{[V{\tiny sing2}]}}.
   \end{equation}
Before going through the computations of $\mathcal L_{{\rm NH}} \mathsf W_\delta$, let us deal with \eqref{eq.inf-NH}. We have for all $\mathsf x\in \mathscr E$, $\mathsf F_0(\mathsf x)\ge (\mathfrak h_*-\varepsilon_*) \mathsf H_{{\rm NH}}(\mathsf x)-\epsilon_\Phi|v| -\epsilon_\Phi|y|^2$ which is lower bounded if $\epsilon_\Phi\in (0,(\mathfrak h_*-\varepsilon_*)/2)$. From now on $\epsilon_\Phi\in (0,(\mathfrak h_*-\varepsilon_*)/2)$.  
Note   that  there exists $c>0$ such that for all $\delta\in (1/2,1]$:
\begin{equation}\label{eq.Upp-NH}
\mathsf W_\delta \le e^{ c^\delta \mathsf H_{{\rm NH}} ^\delta} \text{ on } \mathscr E.
\end{equation}

\begin{prop}
Assume that $V$ satisfies {\rm \textbf{[V{\tiny sing2}]}}. Then, there are  parameters such that \textbf{(C3)} is satisfied with the Lyapunov function $\mathsf W_\delta$ defined in \eqref{eq.Lyapunov-NH}. 
\end{prop}

\begin{proof}
We now simply write $\mathsf L$, $\mathsf H$, and $\mathsf F$ for $\mathsf L_{{\rm NH}}$, $\mathsf H_{{\rm NH}}$, and $\mathsf F_{{\rm NH}}$ respectively. 
From the first inequalities in \eqref{eq.ePsi} and \eqref{eq.Phi1},
\begin{equation}\label{eq.Ff}
|\mathsf F|\le (\mathfrak h_*+\epsilon_0+2\epsilon_\Phi)\mathsf H +  |\inf_{\mathscr E}\mathsf F_0| +1\le c\mathsf H,
\end{equation}
for some $c>0$. 
We will reduce  $\varepsilon_*,\alpha_*,\delta_*,  \epsilon_\Phi>0$ and finally increase $p_*>0$  such that the Lyapunov function $\mathsf W_\delta$ defined in \eqref{eq.Lyapunov-NH} satisfies \textbf{(C3)} on $\mathscr E$.
We have over $\mathscr E$:
\begin{equation}\label{eq.WF-NH}
\frac{\mathcal L_{{\rm NH}}\mathsf W_\delta}{ \mathsf W_\delta}\le \frac{\delta }{\mathsf F^{1-\delta}}\Big[\mathcal L_{{\rm NH}} \mathsf F+  \gamma   |\nabla _v \mathsf F|^2\Big],
\end{equation}
Using \eqref{eq.ePsi} and \eqref{eq.Phi1}, it holds for some $\mathfrak K>0$:
$$
2\gamma \epsilon_\Phi |\nabla _v \mathsf L ||\nabla _v\Phi|\le  \epsilon_\Phi \mathfrak K (|v|+  \mathfrak   g_2  |y+\gamma|^{\frac 12} +  1).
$$
Using in addition \eqref{eq.ePsi2}, \eqref{eq.Phi1}, and \eqref{eq.LPhi},
one has over $\mathscr E$:
\begin{align*}
\mathcal L_{{\rm NH}} \mathsf F+  \gamma   |\nabla _v \mathsf F|^2&=\mathcal L_{{\rm NH}} (\mathsf L+ \epsilon_\Phi \Phi)+  \gamma   |\nabla _v \mathsf L+ \epsilon_\Phi \nabla _v \Phi|^2\\
&\le \mathcal L_{{\rm NH}} \mathsf L+  \gamma   |\nabla _v \mathsf L|^2 +  \epsilon_\Phi \mathcal L_{{\rm NH}} \Phi +  \gamma  \epsilon_\Phi^2 |\nabla _v  \Phi|^2 + 2\gamma \epsilon_\Phi |\nabla _v \mathsf L ||\nabla _v\Phi|\\
&\le -\frac{\gamma \mathfrak h_*}2 (1- \mathfrak h_*)|v|^2- \mathfrak h_*dN (1-\mathfrak f_0)|y| -\frac{\mathfrak f_0}2\delta_*|y||v|^2\\
 &\quad - \big[\mathfrak   g_1\frac{\alpha_*}2+\mathfrak   g_2\frac{dN}{8D_{\text{m}}^2} - \mathfrak f_0(\mathfrak h_*+\delta_*+\varepsilon_*) dN \big] |y| + \epsilon_\Phi  \mathfrak c_V  |v|^2 \\
 &\quad -2 \epsilon_\Phi  \mathfrak h_0 y|v|^2 +\epsilon_\Phi (\gamma+ \mathfrak K)|v|-  \epsilon_\Phi \mathfrak h\circ V |\nabla V|^{2-\zeta }   \\
 &\quad + \epsilon_\Phi |y|  (|v|^2 +1)+ 2 \epsilon_\Phi    dN  \mathfrak h_0\, y+ \epsilon_\Phi \mathfrak K \mathfrak   g_2  |y+\gamma|^{\frac 12} +\mathfrak c. 
\end{align*} 
Choose $\alpha_*> {dN}/(4D_{\text{m}}^2)$ so that  $ \min ( {\alpha_*}/2,{dN}/{(8D_{\text{m}}^2}))= {dN}/(8D_{\text{m}}^2)$.  
Note that   $\epsilon_\Phi (\gamma+ \mathfrak K)|v|\le \epsilon_\Phi   |v|^2  + \mathfrak c$.   
Therefore, using also that $\epsilon_\Phi |y|  |v|^2\le \epsilon_\Phi \mathbf 1_{y\notin [-1,2]} |y|  |v|^2 + 2\epsilon_\Phi   |v|^2$,  it holds
\begin{align*}
\mathcal L_{{\rm NH}} \mathsf F+  \gamma   |\nabla _v \mathsf F|^2& \le -\frac{\gamma \mathfrak h_*}2 (1- \mathfrak h_*) |v|^2 + \mathfrak c_V \epsilon_\Phi  |v|^2+ 3\epsilon_\Phi   |v|^2\\ 
&\quad  -\frac{\delta_*   }2 \mathfrak f_0\, |y||v|^2+ \epsilon_\Phi \mathbf 1_{y\notin [-1,2]} |y|  |v|^2 -2 \epsilon_\Phi \mathfrak h_0 y|v|^2\\
 &\quad - \mathfrak h_*dN (1-\mathfrak f_0)\, |y| -  \frac{dN}{8D_{\text{m}}^2} (\mathfrak   g_1+\mathfrak   g_2)|y|  + (\mathfrak h_*+\delta_*+\varepsilon_*) dN \mathfrak f_0\,  |y| \\
 &\quad +  \epsilon_\Phi |y|   +   2 \epsilon_\Phi    dN \mathfrak h_0\,  y      -  \epsilon_\Phi \mathfrak h\circ V |\nabla V|^{2-\zeta }  +   \epsilon_\Phi \mathfrak K \mathfrak   g_2  |y+\gamma|^{\frac 12} +\mathfrak c.  
\end{align*}
Note that $1-\mathfrak f_0 \ge \mathbf 1_{[1,+\infty)}$, $\mathfrak f_0\ge  \mathbf 1_{(-\infty,-1]}$, $\mathfrak f_0\le  \mathbf 1_{\mathbf R_-}$, and $\mathfrak h_0\ge \mathbf 1_{[2,+\infty)}$.
Choose $\epsilon_\Phi>0$ small enough such that $ \mathfrak m:=\frac{\gamma \mathfrak h_*}2 (1- \mathfrak h_*)-\epsilon_\Phi (\mathfrak c_V   +  3)   >0$. Then, one has for all $\mathsf x=(x,v,y)\in \mathscr E$:
\begin{align}\label{eq.BB}
\mathcal L_{{\rm NH}} \mathsf F(\mathsf x)+  \gamma   |\nabla _v \mathsf F|^2(\mathsf x)& \le \mathfrak B (\mathsf x):=-\mathfrak m |v|^2  + \mathfrak N(\mathsf x) -  \epsilon_\Phi \mathfrak h\circ V |\nabla V|^{2-\zeta }  +\mathfrak c,  
\end{align} 
where 
\begin{align*}
\mathfrak N(\mathsf x)&= -\frac{\delta_*   }2\mathbf 1_{y\le -1} \, |y||v|^2+ \epsilon_\Phi\mathbf 1_{y\notin [-1,2]}  |y|  |v|^2 -2 \epsilon_\Phi \mathbf 1_{y\ge 2} y|v|^2\\
 &\quad - \mathfrak h_*dN \mathbf 1_{y\ge 1} |y| - \frac{dN}{8D_{\text{m}}^2}  (\mathfrak   g_1+\mathfrak   g_2)|y|  + (\mathfrak h_*+\delta_*+\varepsilon_*) dN \mathbf 1_{y\le 0}  |y| \\
 &\quad +  \epsilon_\Phi\mathbf 1_{y\le 0}   |y| + \epsilon_\Phi \mathbf 1_{y\ge 0} |y|   +   2 \epsilon_\Phi    dN \mathbf 1_{y\ge 0} y  + \epsilon_\Phi \mathfrak K  \mathbf 1_{y\le 0}  |y| . 
\end{align*} 
We now study the function $\mathfrak N$ over $\mathscr E$. There are two  cases:
\begin{enumerate}
 \item[A.]  The case when $ y\ge 0$. In this case, one has 
 \begin{align*}
\mathfrak N(\mathsf x)&\le   \epsilon_\Phi  \mathbf 1_{y\ge 2} |y|  |v|^2 -2 \epsilon_\Phi \mathbf 1_{y\ge 2} y|v|^2  - \mathfrak h_*dN \mathbf 1_{y\ge 1} |y|   + \epsilon_\Phi ( 2   dN+ 1) \mathbf 1_{y\ge 0} |y|.
 \end{align*} 
Reduce  $\epsilon_\Phi>0$  such that $\mathfrak  b:= \mathfrak h_*dN  - \epsilon_\Phi (2     dN +1)>0$. Then, in this case, one has:
 $\mathfrak N(\mathsf x)\le -\epsilon_\Phi \mathbf 1_{y\ge 2} |y||v|^2-\mathfrak b  \mathbf 1_{y\ge 0} |y| +\mathfrak c$. 
 \item[B.] The case when $ y\le 0$. Then, 
 \begin{align}
 \nonumber
\mathfrak N(\mathsf x)&\le  -\frac{\delta_*   }2\mathbf 1_{y\le -1} \, |y||v|^2+ \epsilon_\Phi\mathbf 1_{y\le -1}  |y|  |v|^2   - \frac{dN}{8D_{\text{m}}^2} (\mathfrak   g_1+\mathfrak   g_2)|y|\\
\label{eq.Lo}
 &\quad     + (\mathfrak h_*+\delta_*+\varepsilon_*) dN \mathbf 1_{y\le -1}  |y|  +  \epsilon_\Phi (1+ \mathfrak K)  \mathbf 1_{y\le 0}  |y|. 
\end{align}
 We then distinguish between two subcases:
 \begin{enumerate}
 \item[\textbf -]  The case when $ {|v|^2}/{ (p_*\sqrt{|y|^2+1})}\le 1$. In this case one has  $(\mathfrak   g_1+\mathfrak   g_2) \ge  \mathbf 1_{y\le -y_*-1}$. Then, it holds:
  \begin{align*}
\mathfrak N(\mathsf x)&\le  -\frac{\delta_*} 2\mathbf 1_{y\le -1} \, |y||v|^2+ \epsilon_\Phi\mathbf 1_{y\le -1}  |y|  |v|^2  \\
&\quad  - \Big[\frac{dN}{8D_{\text{m}}^2}      - (\mathfrak h_*+\delta_*+\varepsilon_*) dN  -  \epsilon_\Phi (1+ \mathfrak K)  \Big]  \mathbf 1_{y\le -y_*-1} |y| +\mathfrak c. 
\end{align*}
In view of \eqref{eq.Tee}, we can reduce $\delta_*>0$ and then $\varepsilon_*>0$ such that $\frac{dN}{8D_{\text{m}}^2}      - (\mathfrak h_*+\delta_*+\varepsilon_*) dN>0$.  Then,   reduce  again $\epsilon_\Phi>0$  such that $\epsilon_\Phi< {\delta_*}/4$ and such that 
$$\mathfrak n:=\frac{dN}{8D_{\text{m}}^2}      - (\mathfrak h_*+\delta_*+\varepsilon_*) dN  -  \epsilon_\Phi (1+ \mathfrak K)>0.$$ Hence, in this case, it holds:
 $\mathfrak N(\mathsf x)\le  -\frac{\delta_*} 4\mathbf 1_{y\le -1} \, |y||v|^2  -\mathfrak n  \mathbf 1_{y\le 0} |y| +\mathfrak c$. 
 
 \item[\textbf -]  The case when $ {|v|^2}/{ (p_*\sqrt{|y|^2+1})}\ge 1$. In this case $|v|^2\ge p_*$ and thus,  using \eqref{eq.Lo}, it holds (since $\epsilon_\Phi< {\delta_*} /4$):
 $$\mathfrak N(\mathsf x)\le-\frac{\delta_*  p_* }4\mathbf 1_{y\le -1} \, |y| + \Big[ (\mathfrak h_*+\delta_*+\varepsilon_*) dN  +  \epsilon_\Phi (1+ \mathfrak K)  \Big] \mathbf 1_{y\le -1} |y|+\mathfrak c.$$
 Choose $p_*>0$ larger  such that $\frac{\delta_*  p_* }4>\frac{dN}{8D_{\text{m}}^2}$. In this case, $\frac{\delta_*  p_* }4>(\mathfrak h_*+\delta_*+\varepsilon_*) dN$. Then, choose $\epsilon_\Phi>0$ smaller such that 
 $$\mathfrak j:=\frac{\delta_*  p_* }4-(\mathfrak h_*+\delta_*+\varepsilon_*) dN-  \epsilon_\Phi (1+ \mathfrak K)>0.$$ Consequently, in this case, it holds:
 $\mathfrak N(\mathsf x)\le  -\mathfrak j \mathbf 1_{y\le 0} \, |y|  +\mathfrak c$.

  \end{enumerate} 
  
 \end{enumerate} 
 Recall $\delta\in (1/2,1]$. Gathering the estimates obtained in Items A and B above, and using \eqref{eq.BB}, \eqref{eq.Ff}, and \eqref{eq.WF-NH},  we deduce that $ {\mathcal L_{{\rm NH}}\mathsf W_\delta}/{ \mathsf W_\delta}\to -\infty$ as $\mathsf H(\mathsf x)\to +\infty$ (i.e. when $V(x)+ |v|+|y|\to +\infty$). This concludes the proof of the proposition.
\end{proof}


\subsection{Extension of \cite[Theorem 2.2]{guillinqsd}}
\label{sec.extC5}
In this section, we extend \cite[Theorem 2.2]{guillinqsd} when \textbf{(C5)} is replaced by the less stringent  following assumption:
\begin{itemize}
\item[{\bf (C5')}] For all ${\mathsf x}\in \mathscr D$ and nonempty open subset $\mathsf O$ of $\mathscr D$, there exists $t(\mathsf x,\mathsf O)\ge 0$, for all $t\ge t(\mathsf x,\mathsf O)$, 
$P^{\mathscr D}_t({\mathsf x},\mathsf O)>0$. In addition, there exists  $\mathsf x_0\in \mathscr D$ such that $\mathbf P_{\mathsf x_0}(\sigma_{\mathscr D}<+\infty)>0$.
\end{itemize}

To this end, we first extend \cite[Theorem 4.1]{guillinqsd} as follows.

\bthm\label{th.PF} 
Let $\mathcal S$ be a nonempty open subset of $\mathbf R^d$,   $Q=Q({\mathsf x},d{\mathsf y})$ be a positive bounded kernel on $\mathcal S$,  and  $\mathsf W:\mathcal S\to [1,+\infty)$   a continuous  function. Assume  that:
\benu
\item There exists $N_1\ge 1$ such that $Q^k$ is strong Feller for all $k\ge N_1$, i.e. $Q^kf\in \mathcal C^b(\mathcal S)$ if $f\in b{\mathcal B}(\mathcal S)$.
\item For any ${\mathsf x}\in \mathcal S$ and nonempty open subset $\mathsf O$ of $\mathcal S$, there exists $q(\mathsf x,\mathsf O) \in \mathbf N$ such that for all $k\ge q(\mathsf x,\mathsf O)$,
$$
Q^{k} ({\mathsf x},\mathsf O)>0.
$$
\item For some $p>1$ and constant $C>0$, it holds:
$$
Q\mathsf W^p \le C \mathsf W^p.
$$
Notice that this implies that $Q$ is well defined and bounded on $b_{\mathsf W}\BB(\mathcal S)$.
\item $Q$ has a spectral gap in $b_{\mathsf W}\BB(\mathcal S)$,
\bbeq\label{thm52a}
\mathsf r_{ess}(Q|_{b_{\mathsf W}\BB(\mathcal S)})<\mathsf r_{sp}(Q|_{b_{\mathsf W}\BB(\mathcal S)}).
\nneq
 \nenu
 Then, all the conclusions of~\cite[Theorem 4.1]{guillinqsd} hold true.
\nthm
 
 \begin{proof}
We adopt the same notations as those used to prove~\cite[Theorem 4.1]{guillinqsd}.  
 The proof of Theorem~\ref{th.PF} is the same as the one made to prove~\cite[Theorem 4.1]{guillinqsd} except the end of the second step there. More precisely,  we only have to prove that under our new Assumptions $(1)$ and $(2)$, it still holds:
  \begin{enumerate}
 \item[(a)] $\mu(\mathsf O)>0$ for all nonempty open subset $\mathsf O$ of $\mathcal S$,
 \item[(b)] $\varphi(\mathsf x)>0$ for all $\mathsf x\in \mathcal S$, 
 \end{enumerate}
 where we recall that $\mu$ (resp. $\varphi$) is nonzero and nonnegative measure which is a left eigenvector (resp. a nonzero and nonnegative function which is a right eigenvector) of  $Q^m$ (where $m\in \mathbf N^*$ is the  integer chosen  at the beginning of the second step of the proof of~\cite[Theorem 4.1]{guillinqsd}). 
 
 Let us first  prove Item (a) above. We have $\mu Q^m=\mu$ and thus $\mu Q^{mn}=\mu$ for all $n\in  \mathbf N$. Let $n\in \mathbf N^*$ such that $ n\ge N_1/m +1$, so that, by  Assumption $(1)$, $Q^{mn}$ is strong Feller. In addition, since $\mu$ is a   nonzero and nonnegative measure over $\mathcal S$, its (topological) support is nonempty.  
 Let $\mathsf z$ belong to the support of $\mu$. By definition, $\mu(\mathsf B(\mathsf z,\delta))>0$ for all $\delta>0$ small enough such that $\mathsf B(\mathsf z,\delta)\subset \mathcal S$, where   $\mathsf B(\mathsf z,\delta)$ is the open ball in $\mathbf R^d$ of radius $\delta$ centered at $\mathsf z$. Assume also that $n\in \mathbf N^*$ is large enough such that $mn\ge q(\mathsf z,O)$ (see Assumption $(2)$). 
 
 It holds for any nonempty open subset $\mathsf O$ of $\mathcal S$:
 $$\mu(\mathsf O)=\int_{\mathcal S}Q^{mn}(\mathsf x,\mathsf O)\mu(d\mathsf x) \ge 0.$$
 Assume that $\mu(\mathsf O)=0$. Then, $Q^{mn}(\cdot ,\mathsf O)=0$ $\mu$-almost surely.  Consequently, there exists $\mathsf z_n^\delta\in \mathsf B(\mathsf z,\delta)$ such that $Q^{mn}(\mathsf z_n^\delta ,\mathsf O)=0$.  Note that $\mathsf z_n^\delta\to \mathsf z$ as $\delta \to 0$, and since $Q^{mn}$ is strong Feller, it holds $Q^{mn}(\mathsf z ,\mathsf O)=0$.  This leads to a contradiction in view of Assumption $(2)$ and because $mn\ge q(\mathsf z,\mathsf O)$. The proof of   Item (a) above is complete.
 
 Let us now prove Item (b) above. We know that there exists $\mathsf x_0\in \mathcal S$ such that $\varphi(\mathsf x_0)>0$. Thus, by continuity of $\varphi$, there exists $\delta_0>0$ and $c_0>0$ such that $\mathsf B(\mathsf x_0,\delta_0)\subset \mathcal S$ and $\varphi\ge c_0$ on $\mathsf B(\mathsf x_0,\delta_0)$. Let $\mathsf x\in \mathcal S$. Pick $n\in \mathbf N^*$ such that $mn\ge q(\mathsf x,\mathsf B(\mathsf x_0,\delta_0))$. Since $Q^{mn}\varphi=\varphi$, one has, by  Assumption $(2)$,
 $$\varphi(\mathsf x)=\int_{\mathcal S} \varphi(\mathsf y) Q^{mn}(\mathsf x,d\mathsf y)\ge c_0Q^{mn}(\mathsf x, \mathsf B(\mathsf x_0,\delta_0))>0.$$
 This concludes the proof of Item (b) above.
 \end{proof}
Following exactly the same arguments as in the proof of  \cite[Theorem 2.2]{guillinqsd}, we deduce the following extension of \cite[Theorem 2.2]{guillinqsd}. 
 
  \begin{thm}\label{th.ext}
 All the conclusions of \cite[Theorem 2.2]{guillinqsd} are still valid when {\bf (C5)} is replaced by {\bf (C5')} there. 
   \end{thm}

 
 \section{Extensions of  existing results}

 In this section,  we   extend the  existing results on the existence and uniqueness of the quasi-stationary distributions for the kinetic Langevin process as well as for elliptic processes.  
 
 \subsection{Quasi-stationary distribution for the kinetic Langevin process} Consider    the solution $((x_t,v_t),t\ge 0)$    in $\mathbf R^d \times \mathbf R^d $ to the kinetic Langevin equation
\begin{equation}\label{eq.K2}
 dx_t=v_tdt, \ dv_t=\boldsymbol b(x_t)dt-\gamma v_tdt +\sqrt{2}\,  dW_t. 
 \end{equation}

\noindent
Introduce the following assumption:
\medskip

\noindent 
\textbf{Assumption {[$\boldsymbol b${\tiny bounded+lip}]}}.  \textit{The set  $\mathscr O$ is   a bounded subdomain of $\mathbf R^d$ and $\boldsymbol b:\overline{\mathscr O}\to \mathbf R^d$ is   a  Lipschitz vector field.}
 \medskip
 
In~\cite{ramilarxiv2},   existence and  uniqueness of the quasi-stationary distribution of the process \eqref{eq.K2} as well as the exponential convergence were obtained under assumption  {\rm \textbf{[$\boldsymbol b${\tiny bounded+lip}]}} and with the two extra assumptions that   $\boldsymbol b\in \mathcal C^\infty(\overline{\mathscr O})$ and $\partial \mathscr O$ is $\mathcal C^2$. With the techniques used in this work, we are able  to extend this result to the very weak  regularity setting where only {\rm \textbf{[$\boldsymbol b${\tiny bounded+lip}]}} is assumed. Indeed, we have the following result. 
 
  \begin{thm}\label{th.K}
  Assume {\rm \textbf{[$\boldsymbol b${\tiny bounded+lip}]}} and let $\mathscr D=\mathscr O\times \mathbf R^d$. Assume also that $\gamma \in \mathbf R$. 
  Then,  there exists a unique quasi-stationary distribution $\theta_{\mathscr D}$ for the process \eqref{eq.K2} on $\mathscr D$  in the whole  space of probability measures $\mathcal P(\mathscr D)$ over $\mathscr D$. Furthermore: 
 \begin{enumerate}
 \item There exists $\kappa_{\mathscr D}>0$  such that   for all $t\ge0$,  the spectral radius of $P_t^{\mathscr D}$ on $b \mathcal B(\mathscr D)$ is $e^{-\kappa_{\mathscr D} t}$,  where $( P_t^{\mathscr D},t\ge 0)$ is the killed semigroup of the   solution to~\eqref{eq.K2} (see \eqref{eq.KS}).   In addition, for all $t\ge 0$,  $\theta _{\mathscr D} P_t^{\mathscr D}=e^{-\kappa_{\mathscr D}  t}\theta _{\mathscr D}$  and $\theta_{\mathscr D} (\mathsf O)>0$ for all nonempty open subsets $ \mathsf O$ of $\mathscr D$. There is also a unique continuous function $\psi $ in  $b\mathcal B(\mathscr D)$   such that  $\theta _{\mathscr D}(\psi)=1$, $\psi >0$ on $\mathscr D$,  and
 $
 P_t^{\mathscr D}\psi = e^{-\kappa_{\mathscr D} t} \psi \text{ on }\mathscr D, \ \forall t\ge0$.

 \item  There exist  $L>0$ and $c\ge 1$ such that for all   $\nu\in \mathcal P (\mathscr D)$,
$$
\big |\mathbf P_\nu[X_t\in \mathcal A| t<\sigma_{\mathscr D}]-\theta _{\mathscr D}(\mathcal A)\big |\le  \frac{c e^{-L t} }{\nu(\psi )}, \ \forall \mathcal A\in\mathcal B(\mathscr D), t>0.
$$ 
 \item For some $t_0>0$, $ P_t^{\mathscr D}:b\mathcal B(\mathscr D)\to b\mathcal B(\mathscr D) $ is compact, for all $t\ge t_0$. 
 \end{enumerate}
 Finally,  $\mathbf P_{\mathsf x}(\sigma_{\mathscr D}<+\infty)=1$ for all ${\mathsf x}\in \mathscr D$. 
  \end{thm}


  \begin{proof} 
We  extend $\boldsymbol b$ outside $\overline{\mathscr O}$ as a globally Lipschitz and bounded  $\mathbf R^d$-vector field over $\mathbf R^d$. 
Assumption \textbf{(C2)} is a consequence of the fact that for all $t\ge 0$,   a.s. $X_s(\mathsf x_n )\to X_s(\mathsf x) $ uniformly on $[0,t]$  when $\mathsf x_n\to \mathsf x\in \mathbf R^{2d}$. Assumption \textbf{(C1)} is proved using this result, together with the Gaussian upper bound~\eqref{eq.Gau} below and Proposition~\ref{pr.wu1999}.  
On the other hand, since $\boldsymbol b$ is globally Lipschitz, using similar  arguments as those  used in  the proof of Lemma \ref{le.C4}, one deduces that   for all compact subset $K$ of $\mathbf R^{2d}$ and $\delta>0$, it holds:
$$ \lim_{s\to 0^+} \sup_{\mathsf x=(x,v)\in K}\mathbf P_{\mathsf x}[\sigma_{  \mathscr  B( x,\delta)}\le s]=0,$$
 where   for $\mathsf x=(x,v)\in \mathbf R^d$, $\sigma_{\mathscr B( x,\delta)}(\mathsf x):=\inf\{s\ge 0,x_s(\mathsf x)\notin   \mathscr  B( x,\delta)\}$ and where we recall that $\mathscr  B(x,\delta)$ is the open ball of $\mathbf R^{d}$ of radius $\delta>0$ centered at $ x\in \mathbf R^{d}$. 
Assumption {\bf (C4)}   is thus proved with the same arguments as those used to prove   Proposition \ref{co.SFD}. Assumption  \textbf{(C5)} follows from the same arguments as those used in Section~\ref{sec.S5GL}.  

The  proof of the results stated up to Item (2) (included) in Theorem~\ref{th.K} will then be a consequence of \cite[Theorem 4.1]{guillinqsd},  if we prove that for some $t>0$, the essential spectral radius of $P_t^{\mathscr D}$ on $b  \mathcal B(\mathscr D)$ is zero (see indeed   the proof of  \cite[Theorem 2.2]{guillinqsd}   made in  \cite[Section 5]{guillinqsd}), i.e.
    \begin{equation}\label{eq.RS}
\mathsf r_{ess}(P_t^{\mathscr D}|_{b  \mathcal B(\mathscr D)})=0.
\end{equation}
Let us recall the definition of  the measure of
non-$w$-compactness  $\beta_w(Q)$  of a bounded nonnegative kernel $Q=Q(\mathsf x,d\mathsf y)$ on a polish space $\mathscr S$  introduced in~\cite{Wu2004}:
 \begin{equation}\label{eq.defBeta}
\beta_w(Q):= \inf_{K\subset \mathscr S} \sup_{\mathsf x\in \mathscr S} Q(\mathsf x, \mathscr S\setminus K),
\end{equation}
 where the infimum is on the set of  compact subsets $K$  of $\mathscr S$.
Since for $t>0$, $P_t $ is strongly Feller, it  satisfies the assumption \textbf{(A1)} in~\cite[Section~3.2]{Wu2004}, and hence so does $P_t^{\mathscr D}$ over $\mathbf R^{2d}$. We can thus use~\cite[Theorem 3.5]{Wu2004} to deduce that   \eqref{eq.RS} is satisfied if for $t>0$:
 \begin{equation}\label{eq.betatau}
 \beta_w(P_t^{\mathscr D})=0, 
\end{equation}
where (see \eqref{eq.defBeta})
$$ \beta_w(P_t^{\mathscr D})= \inf_{K\subset \mathbf R^{2d}} \sup_{\mathsf x=(x,v)\in \mathbf R^{2d}} P_t^{\mathscr D} (\mathsf x, \mathbf R^{2d} \setminus K)= \inf_{K\subset \mathbf R^{2d}} \sup_{\mathsf x=(x,v)\in \mathscr D} P_t^{\mathscr D} (\mathsf x, \mathscr D\setminus K).$$
Choosing $K_R=\overline{\mathscr O} \times \{|v|\le R\}$, one has 
\begin{align}
\nonumber
 \beta_w(P_t^{\mathscr D})&\le \inf_{R> 0} \ \sup_{\mathsf x=(x,v)\in \mathscr O\times \mathbf R^d}\  P_t^{\mathscr D}(\mathsf x,  \mathscr O\times \{|v|> R\})\\
 \label{eq.bR}
 &\le  \inf_{R> 0} \ \sup_{\mathsf x=(x,v)\in \mathscr O\times \mathbf R^d}\  P_t(\mathsf x,  \mathscr O\times \{|v|> R\}).   
\end{align}
 Let us now prove \eqref{eq.betatau}.  
Let $t>0$ be fixed. 
 In what follows $ c_i$'s are positive constants which are  independent of $\mathsf x=(x,v), \mathsf y=(x',v')\in \mathbf R^{2d}$. 
   By~\cite[Theorem 1.1]{delarue},  for all $\mathsf x \in \mathbf R^{2d}$, $ X_t (\mathsf x)$ admits a density $ p_t (\mathsf x,\mathsf y)$ (with respect to  the Lebesgue measure $d\mathsf y$ over $\mathbf R^{2d}$) which  satisfies the following Gaussian upper bound: for all $\mathsf x,\mathsf y \in \mathbf R^{3d}$,
\begin{equation}\label{eq.Gau}
 p_t (\mathsf x,\mathsf y)\le c_0 \exp( -c_1  |\boldsymbol \psi_t (\mathsf x) -\mathsf y|^2),
 \end{equation}
 where $( \boldsymbol \psi _s(\mathsf x)=(\boldsymbol x_s(\mathsf x), \boldsymbol v_s(\mathsf x)) ,s\ge 0)$ is the deterministic solution for $s\ge 0$   of
$$
         \dot{\boldsymbol x} _s=   \boldsymbol v_s , \, 
        \dot{\boldsymbol v}_s=  \boldsymbol b(\boldsymbol x_s) -\gamma \boldsymbol v_s \text{ with } \boldsymbol \psi _0(\mathsf x)=\mathsf x.$$ 
  Assume that $\gamma\neq 0$ (the case when $\gamma=0$ is treated similarly). 
Let us now consider  the  solution  $( \boldsymbol \Psi_s(\mathsf x)=(\boldsymbol X_s(\mathsf x), \boldsymbol V_s(\mathsf x)) ,s\ge 0)$  for $s\ge 0$   of the equation 
$$
         \dot{\boldsymbol X} _s=   \boldsymbol V_s , \, 
        \dot{\boldsymbol V}_s=   -\gamma \boldsymbol V_s \text{ with } \boldsymbol \Psi _0(\mathsf x)=\mathsf x,$$
          i.e $\boldsymbol V_s(\mathsf x)=ve^{-\gamma s}$ and $\boldsymbol X_s(\mathsf x)= x+ v\gamma ^{-1}(1-e^{-\gamma s})$. Since $\boldsymbol b$ is bounded,  using Grönwall's inequality, it holds
           $$\sup_{s\in [0,t]}|\boldsymbol \psi _s(\mathsf x)-\boldsymbol \Psi _s(\mathsf x)|\le c_2.$$ 
         This previous bounds leads to, for $s\in [0,t]$,  $ |\boldsymbol \psi _s(\mathsf x)  -\mathsf y|^2 \ge  |\boldsymbol \Psi _s(\mathsf x) -\mathsf y|^2 -2 c_2  |\boldsymbol \Psi _s(\mathsf x) -\mathsf y|$. Since $\mathscr O$ is bounded, there exists  $c_5>0$ such that $|x-x'|\le c_5$ for all $x,x'\in \mathscr O$.  Setting $c_3:=\gamma ^{-1}(1-e^{-\gamma t})>0$ and $c_4:=e^{-\gamma t}>0$, it thus  holds for all $\epsilon>0$ and all $\mathsf x,\mathsf y\in \mathscr D$:
              \begin{align*}
              |\boldsymbol \psi _s(\mathsf x) -\mathsf y|^2&\ge   |x-x'+c_3v |^2-2c_2|x-x'+c_3v| \\
              &\quad +  | v' -c_4v |^2-2c_2|v'-c_4v  |   ,\\
              &\ge  c_3^2|v|^2-2c_3 (c_2+c_5)|v| -2 c_2c_5 \\
              &\quad +   | v' |^2 + c_4^2|v |^2    -2 c_4v\cdot v'-2c_2|v'|- 2c_2c_4|v  |     ,\\
              &\ge  c_3^2|v|^2-2c_3 (c_2+c_5)|v| -2 c_2c_5 \\
              &\quad +  | v' |^2 + c_4^2|v |^2    -2 c_4|v|\, |v'|-2c_2|v'|- 2c_2c_4|v  |     ,\\  
              &\ge  c_3^2|v|^2-2c_3 (c_2+c_5)|v| -2 c_2c_5 \\
              &\quad +  | v' |^2 + c_4^2|v |^2    - c_4 \epsilon^{-1} |v|^2- \epsilon c_4 |v'|^2-2c_2|v'|- 2c_2c_4|v  |     \\
              &= |v|^2( c_3^2+c_4^2-\epsilon^{-1} c_4)-2c_3 (c_2+c_5)|v|-2c_2c_4|v  |   \\
              &\quad + |v'|^2(1-\epsilon c_4)-2c_2|v'|    -2 c_2c_5   .
              \end{align*} 
%
Choose   $\epsilon \in (0,1/c_4)$ such that $c_6=c_3^2+c_4^2-\epsilon^{-1} c_4>0$ (this is indeed possible since  $c_6 \to c_3^2>0$ as $\epsilon\to ( 1/c_4)^-$). 
Using the previous computations, we have since $\mathscr O$ is bounded, for all $\mathsf x=(x,v)\in \mathscr O\times \mathbf R^d$:
   \begin{align*}
        P_t(\mathsf x,  \mathscr O\times  \{|v|\ge R\})&\le c_7e^{-c_1c_6|v|^2+ c_8 |v|}\int_{ |v'|\ge R} e^{-c_1(1-\epsilon c_4)|v'|^2+2c_1c_2|v'|  }dv'\\
        &\le c_9 \int_{ |v'|\ge R} e^{-c_1(1-\epsilon c_4)|v'|^2+2c_1c_2|v'|  }dv' \to 0 \text{ as } R\to +\infty.
      \end{align*}
By \eqref{eq.bR}, this implies that $ \beta_w(P_t^{\mathscr D})=0$ and then  \eqref{eq.RS} holds. The proof of the theorem is complete.
\end{proof}
We end this section by  considering  the case when the position domain might be unbounded and the drift is singular:  we state without proof the following extension of the main results of~\cite{guillinqsd2}. 
  \begin{thm}\label{th.Ks} 
 \cite[Theorems 2.4 and 3.2]{guillinqsd2} are both still valid without any regularity assumption on the boundary of the position   domain $\mathsf O$, where $\mathsf O\times \mathbf R^d$   is the subdomain where  the quasi-stationary distribution is considered   there. 
   \end{thm} 
 

 \subsection{Elliptic processes in  bounded   domains}
 In this section, we consider  the solution $(Y_t,t\ge 0)$    in $\mathbf R^d   $ to the equation
\begin{equation}\label{eq.elliptique}
 dY_t =\boldsymbol b(Y_t)dt +\sqrt{2}\,  dW_t. 
 \end{equation}
When $\mathscr O$ is an open subset of $\mathbf R^d$, we  denote by   $(P_t^{\mathscr O},t\ge 0)$  the semigroup of the killed process $(Y_t,t\ge 0)$: 
$
P_t^{\mathscr O}f(y)=\mathbf E_{y}[f(Y_t)\mathbf 1_{t<\sigma_{\mathscr O}}]$, for $ f\in b\mathcal B (\mathscr O)$, $y\in \mathscr O$,  and  where $\sigma_{\mathscr O}=\inf\{t\ge 0, Y_t\notin \mathscr O\}$
 is the first exit time of the process   $(Y_t,t\ge 0)$  from~$\mathscr O$. 
Using the same tools as those used in this work, we are able to get  the  following result which holds without  any regularity assumption on~$\partial \mathscr O$:

  \begin{thm}\label{th.ell}
 Assume {\rm \textbf{[$\boldsymbol b${\tiny bounded+lip}]}}. 
  Then,  there exists a unique quasi-stationary distribution $\nu_{\mathscr O}$ for the process \eqref{eq.elliptique} on $\mathscr O$  in the whole  space of probability measures $\mathcal P(\mathscr O)$ over $\mathscr O$. Furthermore:
  \begin{enumerate}
 \item There exists $\alpha_{\mathscr O}>0$  such that   for all $t\ge0$,  the spectral radius of $P_t^{\mathscr O}$ on $b \mathcal B(\mathscr O)$ is $e^{-\alpha_{\mathscr O} t}$,  where $( P_t^{\mathscr O},t\ge 0)$ is the killed semigroup of the process solution to~\eqref{eq.K2} (see \eqref{eq.KS}).   In addition, for all $t\ge 0$,  $\nu _{\mathscr O} P_t^{\mathscr O}=e^{-\alpha_{\mathscr O}  t}\nu _{\mathscr O}$  and $\nu_{\mathscr O} (\mathsf O)>0$ for all nonempty open subsets $ \mathsf O$ of $\mathscr O$. There is also a unique continuous function $\phi $ in  $b\mathcal B(\mathscr O)$   such that  $\nu _{\mathscr O}(\phi)=1$, $\phi >0$ on $\mathscr O$,  and
 $
 P_t^{\mathscr O} \phi = e^{-\alpha_{\mathscr O} t} \phi  \text{ on }\mathscr O, \ \forall t\ge0$.

 \item  There exist  $m_1>0$ and $m_2\ge 1$ such that for all   $\nu\in \mathcal P (\mathscr O)$,
$$
\big |\mathbf P_\nu[Y_t\in \mathcal A| t<\sigma_{\mathscr O}]-\nu_{\mathscr O}(\mathcal A)\big |\le  \frac{m_2 e^{-m_1 t} }{\nu(\phi )}, \ \forall \mathcal A\in\mathcal B(\mathscr O), t>0.
$$ 
 \end{enumerate}
Eventually,  $\mathbf P_y(\sigma_{\mathscr O}<+\infty)=1$ for all $y\in \mathscr O$.  
\end{thm}
\begin{proof}
Theorem~\ref{th.ell} is proved using  the same arguments as those used to prove Theorem~\ref{th.K} above. 
\end{proof}
Theorem~\ref{th.ell} can be extended to the case where there is a uniformly elliptic and Lipschitz  diffusion coefficient in \eqref{eq.elliptique}. Theorem~\ref{th.ell}  was already derived in \cite{champagnat2017general} (see Theorem 1.1 there) with different techniques.  
We end this work with the   following result about existence and uniqueness of quasi-stationary distributions for elliptic processes on possible unbounded domains and  with singular potentials. 
  \begin{thm}\label{th.ell2} 
 \cite[Proposition 4.2]{guillinqsd2} is  still valid without any regularity assumption on the boundary of the subdomain $\mathfrak D$ on which is considered the quasi-stationary distribution  there. 
   \end{thm} 


\medskip

\noindent
 \textbf{Acknowledgement.}\\
 {\small A. Guillin is supported by the ANR-23-CE-40003, Conviviality, and the Institut Universitaire de France. 
 D. Lu is supported by the China Scholarship Council (CSC, Grant No. 202206060077). 
B.N. is  supported by  the grant  IA20Nectoux from the Projet I-SITE Clermont CAP 20-25 and   by  the ANR-19-CE40-0010, Analyse Quantitative de Processus M\'etastables (QuAMProcs). The authors are  grateful to Laurent Michel for pointing them out the Nos\'e-Hoover process and the work of Loïs Delande~\cite{lois}.}


 \bibliography{GenertalizedLangevin} 
 
\bibliographystyle{plain}

 \end{document}